\documentclass{article} 
\usepackage{graphicx}
\usepackage{amsmath}
\usepackage{epstopdf} % needed if you have eps figures in a pdflatex manuscript
\usepackage{mathptmx}      % use Times fonts if available on your TeX system
\usepackage{subfigure}
\usepackage[svgnames,hyperref]{xcolor}
\usepackage{pgf}
\usepackage{amsmath}
\usepackage{amssymb}
\usepackage{amsthm}
\usepackage[numbers,sort&compress]{natbib}

\usepackage[pdftex,bookmarks=true,pdfstartview=FitH,bookmarksopen=true,	backref,colorlinks,citecolor=Blue,linkcolor=Red,urlcolor=Green,	pagebackref=false,plainpages=false,pdfpagelabels,unicode=true,breaklinks=true,naturalnames=true,hypertexnames=false]{hyperref}
\usepackage{hypcap}

\numberwithin{equation}{section}

\newtheorem{theorem}{Theorem}[section]

\newtheorem{acknowledgements}{Acknowledgements}

\newcommand{\setR}{\mathbb{R}} 
\newcommand{\setC}{\mathbb{C}}

\newcommand{\setN}{\mathbb{N}}
 
\newcommand{\setSp}{\sigma}
\newcommand{\svec}[1]{\left(\begin{smallmatrix}#1 \end{smallmatrix}\right)}
\newcommand{\Xspace}{\mathcal{X}}
\newcommand{\Yspace}{\mathcal{Y}}
\newcommand{\Vspace}{\mathcal{V}}
\newcommand{\Vint}{\Vspace^{\rm int}}
\newcommand{\Vext}{\Vspace^{\rm ext}}
\newcommand{\tracespace}{\mathcal{Z}}
\newcommand{\ui}{u^{\rm inc}}
\newcommand{\us}{u^{\rm sc}}
\newcommand{\uint}{u^{\rm int}}
\newcommand{\uext}{u^{\rm ext}}
\newcommand{\vint}{v^{\rm int}}
\newcommand{\vext}{v^{\rm ext}}
\newcommand{\uML}{U}
\newcommand{\vML}{V}
\newcommand{\nrGM}{M}
\newcommand{\stot}{s}
\newcommand{\atot}{a}
\newcommand{\btot}{b}
\newcommand{\sext}{s^{\rm ext}}
\newcommand{\sint}{s^{\rm int}}
\newcommand{\aext}{a^{\rm ext}}
\newcommand{\aint}{a^{\rm int}}
\newcommand{\bext}{b^{\rm ext}}
\newcommand{\bint}{b^{\rm int}}
\newcommand{\Cs}{C_{s}}

\newcommand{\mupar}{\kappa}
\newcommand{\kapn}{\kappa_0}
\newcommand{\Oi}{\Omega_{\rm int}} 
\newcommand{\Oe}{\Omega_{\rm ext}}
\newcommand{\smat}[1]{\left(\begin{smallmatrix} #1 \end{smallmatrix}\right)}
\newcommand{\paren}[1]{\left( #1 \right)}
\newcommand{\braces}[1]{\left\{ #1 \right\}}
\newcommand{\norm}[1]{\left\| #1 \right\|}
\newcommand{\abs}[1]{\left| #1 \right|}
\newcommand{\lsp}{\left\langle}
\newcommand{\rsp}{\right\rangle}
\newcommand{\diffq}[2]{\frac{\partial #1}{\partial #2}}

\newcommand{\calD}{\mathcal{D}}
\newcommand{\calO}{\mathcal{O}}

\newcommand{\SpLam}{\Sigma}

\newcommand{\bpM}{\Gamma}
\newcommand{\bp}{\tilde \Gamma}

\DeclareMathOperator{\tr}{tr}
\DeclareMathOperator{\Span}{span}
\DeclareMathOperator{\MT}{ {\cal M}_{\kapn} }

\DeclareMathOperator{\LT}{{\cal L}}
\DeclareMathOperator{\OpT}{\cal T}

\DeclareMathOperator{\OpC}{\cal C }

\DeclareMathOperator{\OpS}{\rm S}

\DeclareMathOperator{\Bdv}{\mathcal{B}}

\DeclareMathOperator{\Id}{I}

\DeclareMathOperator{\supp}{supp}
\DeclareMathOperator{\dist}{dist}

\newtheorem{lemm}[theorem]{Lemma}
\newtheorem{prop}[theorem]{Proposition}
\newtheorem{defi}[theorem]{Definition}
\newtheorem{rem}[theorem]{Remark}

\def\vect#1#2{\left(\!\!\begin{array}{c} #1 \\ #2 \end{array}\!\!\right)}

\begin{document}

\title{Convergence of infinite element methods for scalar waveguide problems}


\author{Thorsten Hohage         \and
        Lothar Nannen %etc.
}
\maketitle

\begin{abstract}
We consider the numerical solution of scalar wave equations in domains which 
are the union of a bounded domain and a finite number of infinite cylindrical waveguides. The aim of this paper is to provide 
a new convergence analysis of both the Perfectly Matched Layer (PML) method and the 
Hardy space infinite element method in a unified framework. 
We treat both diffraction and resonance problems. 
The theoretical error bounds are compared with errors in numerical experiments. 
\end{abstract}

\section{Introduction}
\label{intro}
We consider the numerical solution of time harmonic wave equations in domains which are 
the union of some bounded interior domain and a finite number of semi-infinite waveguides 
(see Fig.~\ref{Fig:sketchproblem}). We consider both the case of excitation by 
incoming propagating modes in one of the waveguides or by a source in the interior domain 
and the case of resonance problems. 
For the analysis of existence, uniqueness and properties of solutions to such problems 
we refer to \cite{NP:94,BNiBonLeg:2004} and the references therein.

\begin{figure}
  \capstart
 \begin{center}
 %\resizebox{0.4\textwidth}{!}{\input{sketch_problem2.pdf_t}} 
\scalebox{0.25}{
\begin{picture}(0,0)
\includegraphics{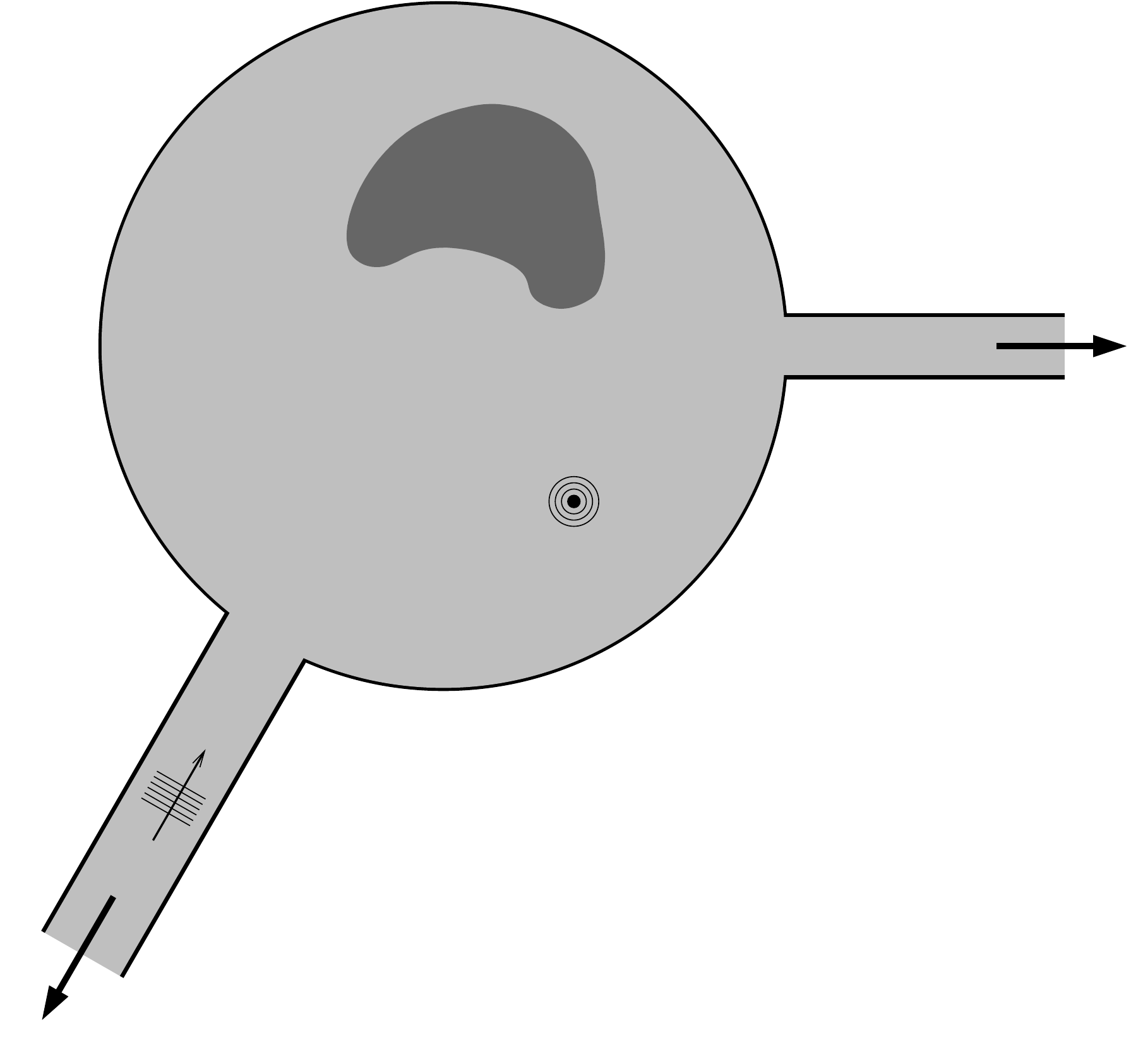}
\end{picture}
\setlength{\unitlength}{4144sp}
\begin{picture}(8265,7699)(841,-7280)
%\put(5041,-3616){\makebox(0,0)[lb]{\smash{{{\fontsize{30}{24.0} \color[rgb]{0,0,0}$f$}}}}}
\put(9091,-2176){\makebox(0,0)[lb]{\smash{{{\fontsize{30}{24.0} \color[rgb]{0,0,0}$u_{\rm s}$}}}}}
\put(2251,-4876){\makebox(0,0)[lb]{\smash{{{\fontsize{30}{24.0}\color[rgb]{0,0,0}$u_{\rm i}$}}}}}
\put(856,-7171){\makebox(0,0)[lb]{\smash{{{\fontsize{30}{24.0}\color[rgb]{0,0,0}$u_{\rm s}$}}}}}
\put(2341,-1501){\makebox(0,0)[lb]{\smash{{{\fontsize{30}{24.0}\color[rgb]{0,0,0}$\Oi$}}}}}
\put(2056,-2581){\makebox(0,0)[lb]{\smash{{{\fontsize{40}{24.0}\color[rgb]{0,0,0}$-\Delta u - \kappa^2 u = f$}}}}}
\end{picture}
} 
 \caption{sketch of the waveguide problem under consideration}
 \label{Fig:sketchproblem}
\end{center}
\end{figure}
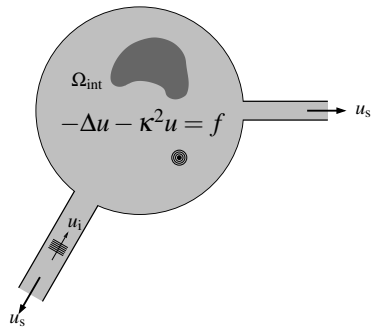

If such problems are solved numerically by finite element methods, the waveguides require a special treatment to avoid reflections at artificial boundaries 
in the waveguides.  
A simple option is to precompute the propagating modes by 
solving an eigenvalue problem on the cross section of each waveguide and use this 
to construct an approximation to the Dirichlet-to-Neumann map. However, the 
Dirichlet-to-Neumann map depends in a non-polynomial way on the wave number. 
For resonance problems this destroys the eigenvalue structure of the problem. 
Nevertheless, there exist alternative numerical methods for waveguide 
resonance problems, e.g.\ using Greens functions \cite{Rotteretal:04} 
or eigenfunction expansions in the interior domain \cite{LM:08,Racecetal:09}.

In this paper we analyze the convergence of numerical methods which are
based on a variational formulation in the waveguides. We present two general 
convergence theorems based on $\OpS$-coercivity arguments \cite{BonnetBenDhia:10}. 
It is used to prove both convergence of the Perfectly Matched 
Layer (PML) method and the Hardy space infinite element method (HSM). This is the first complete 
convergence analysis of the Hardy space method in dimension greater than 1. 
Moreover, it differs from previous convergence results for the PML method 
\cite{LassasSomersalo:98,CW:03,BNiBonLeg:2004,PC2,KimPasciak:09,kalvin:11} 
in the fact that the truncation of the PML layer (with Dirichlet boundary conditions) 
is treated as an approximation error, not as an error on a continuous level. 
In this sense we interpret PML as an infinite element method, i.e.~as a
conforming discretization of a variational formulation of the original problem 
on an unbounded domain. Therefore no modeling error has to be taken into account. 
Moreover, it gives rise to a unified treatment of PML and HSM.
Finally, we discuss a method to treat frequencies close to Wood anomalies by 
the Hardy space method. 

The plan of this paper is as follows: After a general formulation of the problem in Sec.~\ref{sec:setting} 
we state in Sec.~3 the main convergence theorems for diffraction and resonance problems in an abstract framework, which are proved in Sec.~4. 
In the following we apply
the convergence theorems to the PML (Sec.~5) and to the Hardy space method (Sec.~6) 
both for scalar Helmholtz diffraction and resonance problems. 
In the last section we give numerical convergence studies for the Hardy space method
and show that the method is applicable to resonance problems. 

\section{Formulation of the problem}
\label{sec:setting}
Let $\Omega= \Oi\cup\bigcup_{l=1}^L (W_l\cup \bpM_l)\subset\setR^d$ be a Lipschitz domain, which is the disjoint union of a 
bounded Lipschitz domain $\Omega_{\rm int}$, $L$ semi-infinite cylinders (waveguides) $W_1,\dots, W_L$ and interfaces
$\bpM_l$. More precisely, the $W_l$ and $\bpM_l$ are of the form $W_l= \eta_l((0,\infty)\times \bp_l)$ 
and $\bpM_l:=\eta_l(\{0\}\times\bp_l)$ where $\eta_l:\setR^d\to\setR^d$ is 
a Euclidean motion and $\bp_l\subset\setR^{d-1}$ is a bounded Lipschitz domain. The interfaces 
are assumed to be contained in $\overline{\Oi}$. The exterior domain is defined as 
$\Oe:= \bigcup_{l=1}^LW_l$. 

For the sake of simplicity of exposition we will consider the standard 
Helmholtz equation in all our examples. However, we will formulate our convergence 
results in an abstract framework which includes certain variable coefficients 
in the interior domain and in the lateral directions of the waveguides.
Consider the diffraction problem 
\begin{subequations}\label{eqs:scat_prbl}
\begin{align}
\label{eq:scat_prbl_pde}
&-\Delta u - \kappa^2u = f&&\mbox{in }\Omega \\
\label{eq:scat_prbl_bc}
&\Bdv u = g && \mbox{on }\partial \Omega\\
\label{eq:scat_prbl_rc}
&u-\ui &&\mbox{satisfies a radiation condition in }\Oe.
\end{align}
\end{subequations}
Here $\kappa>0$ is a given wave number, 
$\Bdv$ is a trace operator, e.g.\ the Dirichlet trace operator $\Bdv u = u|_{\partial\Omega}$ or 
the Neumann trace operator $\Bdv u = \diffq{u}{\nu}|_{\partial\Omega}$, and we assume that 
$\supp f$ and $\supp g$ are contained in $\overline{\Oi}$. Moreover, $\ui$ is some given incident field 
in $\Oe$ satisfying $(\Delta+\kappa^2)\ui=0$ in $\Oe$ and $\Bdv \ui=0$ on $\partial\Oe \setminus \bigcup_{l=1}^L \bpM_l$. The terms 
radiation condition and incident will be defined in Definition \ref{defi:modal_rad_cond} below. 

We will also consider resonance problems, which have the form  \eqref{eqs:scat_prbl}, but $f,g$, and $\ui$ vanish, 
$\kappa$ may be complex valued, and both $\kappa$ and $u\neq 0$ are considered as unknowns. 

In this paper we will consider several equivalent formulations of the radiation condition leading to different 
numerical algorithms. We start with the most standard one based on a series expansion of the solution.
We may assume w.l.o.g.\ that $W_l= \{0\}\times \bp_l$ for some $l=1,\dots,L$ (otherwise change to the coordinate system 
given by $\eta_l$). Moreover, we assume that the coefficients of $\Bdv$ are constant on $W_l$ and that the negative Laplacian 
$-\Delta_l: \calD(-\Delta_l) \subset L^2(\bp_l) \to L^2(\bp_l)$ with a 
domain of definition $\calD(-\Delta_l)$ incorporating $\Bdv$ is self-adjoint and has a compact resolvent. 
For the Dirichlet trace operator this is the case with 
$\calD(-\Delta_l)=H^2(\bp_l)\cap H^1_0(\bp_l)$, and for the Neumann trace operator 
with $\calD(-\Delta_l) = \{v\in H^2(\bp_l):\diffq{v}{\nu}=0 \mbox{ on }\partial\bp_l\}$. 
Then there exists a complete orthonormal set $\{\varphi_n:n\in\setN\}\subset L^2(\bp_l)$ of eigenfunctions, 
$-\Delta_l\varphi_n = \lambda_n\varphi_n$ with $\lambda_n\geq 0$. (Here and in the following we omit the index 
$l$.) We generally assume in this paper that 
\begin{equation}\label{eq:noWoodAnomaly}
\kappa^2\notin \bigcup_{l=1}^L\setSp\paren{-\Delta_l}.
\end{equation}
Then by separation of variables every solution to \eqref{eq:scat_prbl_pde} and \eqref{eq:scat_prbl_bc} 
with $\kappa>0$ has the form
\begin{equation}\label{eq:expansion_u}
u(x,y) = \sum_{n=1}^\infty \paren{c_n\exp\paren{i\kappa_nx}+d_n\exp(-i\kappa_nx)}\varphi_n(y)\qquad \mbox{in } W_l
\end{equation}
where $c_n$ and $d_n$ are complex coefficients, $x\in (0,\infty)$, $y\in \bp_l$ and
\begin{equation}
\label{eq:defkn}
 \kappa_n:=\begin{cases}
            \sqrt{\kappa^2 - \lambda_n},\qquad &\kappa^2>\lambda_n\\
            i\sqrt{\lambda_n-\kappa^2},\qquad &\kappa^2<\lambda_n
           \end{cases}.
\end{equation}
The functions $\exp(i\kappa_nx)\varphi_n(y)$ and $\exp(-i\kappa_nx)\varphi_n(y)$ are called waveguide modes. 
If $\kappa^2<\lambda_n$, then $\exp(-i\kappa_n x)$ is exponentially growing as $x\to\infty$ whereas $\exp(i\kappa_n x)$ is 
exponentially decaying. The functions $\exp(i\kappa_nx)\varphi_n(y)$ are called \emph{evanescent modes}. 
Since we expect a physical solution to be bounded, we require that $d_n=0$ for such $n$. 
The modes $\exp(\pm i\kappa_nx)\varphi_n(y)$ with $\kappa^2>\lambda_n$ are called \emph{propagating modes}. 
Since $\lim_{n\to\infty}\lambda_n=\infty$ every waveguide $W_l$ supports at most a finite number of propagating modes. 
If the time dependence is given by $\exp(-i\omega t)$ then $\exp(i(\kappa_n x-\omega t))$ is propagating to the right whereas 
$\exp(-i(\kappa_n x+\omega t))$ is propagating to the left. Moreover, if $u$ is an acoustic and transverse magnetic electric field, 
then $\mathcal{J}_l(u) = \Im \int_{\bpM_l} \overline{u} \diffq{u}{x} ds$ can be interpreted as average outward energy flux through 
$\bpM_l$, and 
$\mathcal{J}_l(e^{i\kappa_n x}\varphi_n(y))>0$ whereas
$\mathcal{J}_l(e^{-i\kappa_n x}\varphi_n(y))<0$. Therefore, we call $\exp(i\kappa_nx)\varphi_n(y)$ an 
\emph{outward propagating mode} and $\exp(-i\kappa_nx)\varphi_n(y)$ an \emph{inward propagating mode}.

\begin{defi}[modal radiation condition]\label{defi:modal_rad_cond} 
Let $u$ be a solution to \eqref{eq:scat_prbl_pde} and \eqref{eq:scat_prbl_bc} with $\kappa>0$ and assume \eqref{eq:noWoodAnomaly}. 
We say that $u$  satisfies the 
\emph{(modal) radiation condition} if it is a linear combination of evanescent and outward propagating modes 
in each waveguide $W_l$, $l=1,\dots,L$. 
$u$ is called an \emph{incident field} if it is a linear combination of inward propagating modes 
in each waveguide $W_l$, $l=1,\dots,L$. 
\end{defi}

\section{Formulation of the main convergence theorems}
\label{sec:ConvTheo}
We first formulate the assumptions of our general convergence theorem. To illustrate
and motivate these assumptions we show in this section that they are satisfied in the 
simplest case 
\begin{equation}\label{eq:kappa_bd}
\kappa^2<\inf \bigcup_{l=1}^L \setSp(-\Delta_l), 
\end{equation}
i.e.\ that none of the waveguides supports a propagating mode. We assume that $\Bdv$ is 
the Dirichlet trace operator $g=0$, and of course $\ui=0$. Moreover, let $L=1$ and 
$W:=(0,\infty)\times \bp$ and set $W:=W_1$, $\bp:=\bp_1$, and 
$\Delta_{\bp}:=\Delta_1$. 
Then we obtain 
the following variational formulation of \eqref{eqs:scat_prbl} in $\Vspace = H^1_0(\Omega)$:
\begin{equation}\label{eq:noguidedmodes}
\int_{\Omega} \paren{\nabla u\cdot \nabla \overline{v} - \kappa^2 u\overline{v}} \,dx 
= \int_{\Oi} f\overline{v}\,dx
\end{equation}

\hypertarget{AssA}{}
{\bf Assumption~A: Exterior and interior spaces.} 
\emph{Let $\Vint$ and $\Vext$ be two Hilbert spaces, let $\tracespace$ be another Hilbert space 
(a trace space), and consider bounded, linear, surjective (trace) operators
$\tr_+:\Vext\to \tracespace$ and $\tr_-:\Vint\to\tracespace$. We set
\[
\Vspace:=\paren{\svec{\uint\\\uext}\in\Vint\oplus\Vext:\tr_+\uext = \tr_-\uint}.
\]
Moreover, there exist Hilbert spaces $\Xspace^1_l, \Xspace^2_l, \Yspace^1_l$, and $\Yspace^2_l$ for $l=1,\dots, L$ 
such that $\Xspace^2_l \subset \Xspace^1_l$ and 
$\Yspace^2_l \subset \Yspace^1_l$ are densely and continuously embedded, and 
\begin{equation}\label{eq:defi_Vext}
\begin{aligned}
&\Vext=\bigoplus_{l=1}^L \Vext_l,\qquad 
\Vext_l:=\Xspace^2_l \otimes \Yspace^1_l \cap \Xspace^1_l \otimes \Yspace^2_l,\\
&\lsp\uext,\vext \rsp_{\Vext}=\sum_{l=1}^L\paren{\lsp\uext_l,\vext_l \rsp_{\Xspace^2_l \otimes \Yspace^1_l}
+\lsp\uext_l,\vext_l \rsp_{\Xspace^1_l \otimes \Yspace^2_l}}.
\end{aligned}
\end{equation}
Finally, let $\sint:\Vint\times\Vint\to \setC$ and $\sext:\Vext\times \Vext \to \setC$ 
be bounded sesquilinear forms and set
\[
\stot:\Vspace\times \Vspace\to\setC, \qquad \stot\paren{\svec{\uint\\ \uext},\svec{\vint\\ \vext}}:=
\sint(\uint,\vint)+\sext(\uext,\vext).
\]}

\medskip
As a closed subspace of $\Vint\oplus\Vext$ the space $\Vspace$ equipped with the scalar product
$\lsp\svec{\uint\\\uext}, \svec{\vint\\\vext}\rsp_{\Vspace}:=\lsp \uint,\vint\rsp_{\Vint} +\lsp\uext,\vext\rsp_{\Vext}$
is a Hilbert space. The spaces $\Xspace^{j}_l$ correspond to the infinite directions of the waveguides whereas 
$\Yspace^{j}_l$ correspond to the cross sections. 

\begin{rem}
 $\Vext_l$ is a subset of the tensor product Hilbert space $\Xspace^1_l \otimes \Yspace^1_l$, which is defined via completion under the scalar product
 $$ \lsp u_1 \otimes v_1, u_2 \otimes v_2 \rsp_{\Xspace^1_l \otimes \Yspace^1_l}:= \lsp u_1,u_2\rsp_{\Xspace^1_l } \lsp v_1,v_2\rsp_{\Yspace^1_l },
 \quad u_1,u_2\in \Xspace^1_l,\quad v_1,v_2\in \Yspace^1_l.$$
 Hence, $u\otimes v \in \Vext_l$ is well defined. But $\Vext_l$ is not a tensor product Hilbert space due to the definition of the scalar product
 in \eqref{eq:defi_Vext}.
\end{rem}

\emph{Verification for \eqref{eq:noguidedmodes}:} Assumption~\hyperlink{AssA}{A} is satisfied if we split $u\in H^1_0(\Omega)$ into
$\uint:=u|_{\Oi}$ and $\uext:=u|_{\Oe}$. More precisely, we have for the exterior space
\begin{align*}
 &\Xspace^1:= L^2((0,\infty)),\qquad \Xspace^2:= H^1((0,\infty)), \qquad \Yspace^1:=L^2(\bp),\qquad \Yspace^2:=H^1_0(\bp),\\
 &\Vext := \Xspace^2 \otimes \Yspace^1 \cap \Xspace^1 \otimes \Yspace^2 \sim \{\uext\in H^1(\Oe):\uext|_{\partial\Oe\setminus\bpM}=0\},
\end{align*}
with norms $\|u\|_{\Xspace^2}^2:= \|u\|_{L^2}^2 + \|u'\|_{L^2}^2$ and
\[
\|u\|_{\Yspace^2}^2 := \|u\|_{L^2}^2+\|\nabla u\|_{L^2}^2 
= \lsp u-\Delta_{\bp}u,u\rsp_{L^2} = \|(\Id-\Delta_{\bp})^{1/2}u\|_{L^2}^2,
\] 
such that the norm defined by \eqref{eq:defi_Vext} is given by $\|u\|_{\Vext}^2= 2\|u\|_{L^2(\Oe)}^2+\|\nabla u\|_{L^2(\Oe)}^2$. Moreover,
\begin{align*}
&\Vint = \{\uint\in H^1(\Oi):\uint|_{\partial\Oi\setminus\bpM}=0\},\\ 
&\tracespace = H^{1/2}_0(\bpM) =\mathcal{D}\paren{(\Id-\Delta_{\bpM})^{1/4}}, 
\qquad \tr_-\uint := \uint|_{\bpM}, \qquad \tr_+\uext := \uext|_{\bpM},\\ 
&\sint(\uint,\vint)= \int_{\Oi} \paren{\nabla \uint\cdot \nabla \overline{\vint} - \kappa^2 \uint\overline{\vint}} \,dx,\\
&\sext(\uext,\vext)= \int_{\Oe} \paren{\nabla \uext\cdot \nabla \overline{\vext} - \kappa^2 \uext\overline{\vext}} \,dx.
\end{align*}
In the following we will assume that $\tracespace$ is equipped with the inner product 
$\lsp u,v\rsp_{\tracespace}:= \sum_{n=1}^{\infty}(1+\lambda_n)^{1/2}\lsp u,\varphi_n\rsp\lsp\varphi_n,v\rsp$. 

\medskip
\hypertarget{AssB}{}
{\bf Assumption~B: separation of $\Vext$.} 
\emph{
There exists a complete orthogonal system $\{\varphi_n:n\in\setN\}\subset \bigoplus_{l=1}^L\Yspace^2_l$ 
with the following properties: %Let $l(n)\in\{1,\dots, L\}$ be chosen such that $\varphi_n\in\Yspace^2_{l(n)}$ for all $n\in\setN$. 
\begin{enumerate}
 \item For all  $n\in\setN$ we can choose $l(n)\in\{1,\dots, L\}$ such that for $\varphi_n=(\varphi_n^{(1)},\dots,\varphi_n^{(L)})$
 it holds $\varphi_n^{(j)}= 0$ for $j\neq l(n)$.
  \item The subspaces (not to be confused with $\Vext_l$ of Ass.~\hyperlink{AssA}{A}) 
\[
\Vspace_n:= \Xspace^1_{l(n)}\otimes \Span\{\varphi_n\} \cap \Vext
\]
are orthogonal in $\Vext$ both with respect to the inner product of $\Vext$ 
and with respect to $\sext$, and the spaces $\tr_+(\Vspace_n)$ are orthogonal 
in $\tracespace$. 
\item \label{Ass:BProp3} Finally, $\dim \tr_+(\Vspace_n)<\infty$ for all $n\in\setN$ and
\[
\Vext = \bigcup_{n\in\setN}\Vspace_n \qquad \mbox{and}\qquad 
\tracespace = \bigcup_{n\in\setN} \tr_+(\Vspace_n).
\]
\end{enumerate}
}

\medskip
It follows from the assumption $\Vext = \bigcup_{n\in\setN}\Vspace_n$ that
the spaces 
\begin{equation}\label{def:Xspace_n}
\Xspace_n:=\{u_n\in\Xspace^1_{l(n)}:u_n\otimes \varphi_n\in\Vext\},\quad \lsp u_n,v_n\rsp_{\Xspace_n}:=\lsp u_n\otimes\varphi_n,v_n\otimes\varphi_n\rsp_{\Vext},
\end{equation}
equipped with the Hilbert norms $\|u_n\|_{\Xspace_n}:=\sqrt{\lsp u_n, u_n\rsp_{\Xspace_n}}$
are again Hilbert spaces, and every $\uext\in\Vext$ has a unique representation of the form 
\begin{equation}\label{eq:Uexpansion}
\begin{aligned}
&\uext= \sum_{n\in\setN} u_n\otimes \varphi_n,\qquad
\|\uext\|_{\Vext}^2 = \sum_{n\in\setN}\|u_n\|_{\Xspace_n}^2,\qquad u_n\in\Xspace_n.
\end{aligned}
\end{equation}
We define the sesquilinear forms $\stot_n:\Xspace_n\times \Xspace_n\to \setC$ by 
\[
\stot_n(u,v):=\stot(u\otimes \varphi_n,v\otimes \varphi_n).
\] 
If $\vext = \sum_{n\in\setN}v_n\otimes \varphi_n$, we have
\[
\sext(\uext,\vext) =\sum_{n\in\setN} \stot_n(u_n,v_n)
\]
due to the assumed orthogonality of the spaces $\Vspace_n$ w.r.t.\ $\sext$. 

\emph{Verification for \eqref{eq:noguidedmodes}:} 
Let as in the introduction $\{\varphi_n:n\in\setN\}\subset H_0^1(\bp)=\Yspace^2$ be the complete orthogonal set of eigenfunctions to $-\Delta$, i.e.
$-\Delta\varphi_n = \lambda_n\varphi_n$ with $\lambda_n\geq 0$. Since $\{\varphi_n:n\in\setN\}$ is also a complete orthogonal set in
$ H_0^{1/2}(\bp)\sim \tracespace$ and $L^2(\bp)=\Yspace^1$, the orthogonality assumptions are easy to check, $\dim \tr_+(\Vspace_n)=\dim \Span\{\varphi_n\}=1$, and property
\ref{Ass:BProp3} of Ass.~\hyperlink{AssB}{B}  holds. We have
\begin{align}\label{eq:sn_Xn}
\begin{aligned}
&\|u\|_{\Xspace_n}^2 = \|u'\|_{L^2}^2 + (\lambda_n+2)\|u\|_{L^2}^2,\\
&\stot_n(u,v) = \lsp u',v' \rsp_{L^2} + (\lambda_n-\kappa^2)\lsp u,v\rsp_{L^2}.
\end{aligned}
\end{align}

\medskip
\hypertarget{AssC}{}
{\bf Assumption~C: boundedness and coercivity.} 
\emph{There exists a constant $\nrGM\in \setN$ (later on the number of guided modes),
a stability constant $\Cs>0$, a coercivity constant $\alpha>0$ and rotations $\theta_1,\dots,\theta_\nrGM\in\{z\in\setC:|z|=1\}$ such that
\begin{subequations}
\label{eqs:Ass_sn}
\begin{align}
&|\stot_n(u_n,v_n)|\leq \Cs  \|u_n\|_{\Xspace_n}\|v_n\|_{\Xspace_n},&& n\in\setN  \label{eq:Ass_sn_bd} \\
&\Re\paren{\theta_n \stot_n(u_n,u_n)}\geq \alpha \|u_n\|_{\Xspace_n}^2,&& n=1,\dots,\nrGM \label{eq:Ass_sn_coerc1}\\
&\Re\paren{\stot_n(u_n,u_n)}\geq \alpha \|u_n\|_{\Xspace_n}^2,&& n>\nrGM \label{eq:Ass_sn_coerc2}
\end{align}
\end{subequations}
for all $u_n,v_n\in\Xspace_n$. Moreover, 
there exists a compact linear operator $K:\Vint\to\Vint$ such that 
\begin{equation}\label{eq:compact_pert}
\Re\sint(\uint,\uint) + \Re\lsp K\uint,\uint\rsp_{\Vint}\geq \alpha\|\uint\|^2\qquad \mbox{for all }\uint\in\Vint.
\end{equation}%
}

\medskip
It is essential that the constants $\Cs $ and $\alpha$ do not depend on $n$. 
Due to \eqref{eq:compact_pert}, $\sint$ is coercive up to a compact perturbation.  
In our application (PML or HSM formulation for Helmholtz waveguide problems) $\sext$ is neither coercive nor coercive up to a compact 
perturbation since guided and evanescent modes must be treated differently.
This requires the use of $S$-coercivity in our analysis. To deal with the coupling to the 
interior domain, we have to assume that $\nrGM$ (the number of guided modes) is finite. 

\emph{Verification for \eqref{eq:noguidedmodes}:} Here $\nrGM=0$, and due to \eqref{eq:sn_Xn} 
assumption \eqref{eq:Ass_sn_bd} holds true with $\Cs =1$, \eqref{eq:Ass_sn_coerc1} is empty, and \eqref{eq:Ass_sn_coerc2} 
holds true with $\alpha = (\lambda_1-\kappa^2)/(\lambda_1+2)$, which is positive due to \eqref{eq:kappa_bd}. \\
\eqref{eq:compact_pert} holds true with $K= (\kappa^2+1)J^*J$ where 
$J:\Vint\hookrightarrow L^2(\Omega)$ is the embedding operator, which is compact. 

\medskip
\hypertarget{AssD}{}
{\bf Assumption~D: discrete subspaces.} 
\emph{We consider families of finite dimensional nested subspaces $\Vint_{h}\subset \Vint$ and 
$\Yspace_{h,l}\subset\Yspace^2_l$, which are decreasing in a parameter $h>0$, 
and a family of nested subspaces $\Xspace_{N,l}\subset \Xspace^2_l$, which are increasing in a parameter 
$N\in\setN$ such that $\bigcup_{h>0}\Vint_h\subset \Vint$, 
$\bigcup_{N\in\setN}\Xspace_{N,l}\subset \Xspace^2_l$, and $\bigcup_{h>0}\Yspace_{h,l} \subset\Yspace^2_l$
are dense for $l=1,\dots,L$. Assume that 
\begin{align}\label{eqs:cond_discrete}
&\Vext_{h,N}\subset \Vext\qquad \mbox{and}\qquad 
 \tr_+\paren{\Vext_{h,N}}=\tr_-\paren{\Vint_{h}}
\end{align}
with $\Vext_{h,N}:=\bigoplus_{l=1}^L\Xspace_{N,l}\otimes \Yspace_{h,l}$ and set 
\[
\Vspace_{h,N}:=
\braces{\smat{\uint\\ \uext}\in \Vint_{h} \oplus \Vext_{h,N}: \tr_-\uint =\tr_+\uext}.
\]
Finally, assume there exist operators $\tr_-^{\dagger}\in L(\tracespace,\Vint)$ and 
$\tr_{h,-}^{\dagger}\in L(\tracespace,\Vint_h)$ such that 
$\tr_-\tr_-^{\dagger}=\Id_{\tracespace}$, $\tr_-\tr_{h,-}^{\dagger}\tr_-\uint_h=\tr_-\uint_h$ for all $\uint_h\in\Vint_h$
and
\begin{equation}\label{eq:trace_inverse_conv}
\lim_{h\to 0}\|\tr_-^{\dagger}g-\tr_{h,-}^{\dagger}g\|_{\Vint} =0
\qquad \mbox{for all }g\in\tracespace.
\end{equation}}

\medskip
The conditions \eqref{eqs:cond_discrete} obviously ensure that $\Vspace_{h,N}\subset\Vspace$. 
We emphasize that it is \emph{not} assumed that any of the orthogonal basis functions $\varphi_n$ is contained 
in any of the subspaces $\Yspace_h$. The functions $\varphi_n$ are only used in our analysis, but typically 
not in the numerical algorithms. 

\emph{Verification for \eqref{eq:noguidedmodes}:} We may start with any sequence of finite element 
spaces $\Vint_h\subset \Vint$ such that the best approximations to any $\uint\in\Vint$ in 
$\Vint_h$ converge to $\uint$ as $h\to 0$ and for each $h$ some sub-mesh yields an admissible mesh 
for $\bp$. For $\Vext_{h,N}=\Xspace_N \otimes \Yspace_h$ we set $\Yspace_h:= \tr_-(\Vspace_h)$
and define a non-decreasing mapping $\setN \ni N \mapsto \rho_N>0$ such that $\rho_N \to \infty$ for $N\to \infty$.
Let $\widetilde \Xspace_N$ be any $H^1((0,\rho_N))$-conforming finite element space %with $\tilde v(\rho_N)=0$ for all $\tilde v \in \widetilde \Xspace_N$ 
and
$$\Xspace_N :=\{v \in H^1((0,\infty))~|~v|_{(0,\rho_N)}\in \widetilde \Xspace_N,\quad v|_{[\rho_N,\infty)}\equiv 0 \} \subset \Xspace^2.$$
In order to get nested subspaces, $\Xspace_{N+1}$ has to be constructed such that for $v \in \Xspace_{N+1}$ we have $v|_{(0,\rho_N)}\in \widetilde \Xspace_N$.
This can be done by a suitable refinement of the mesh in $[0,\rho_N]$ ($h$ and/or $p$ refinement) and adding new finite elements for $[\rho_N,\rho_{N+1}]$. 

Then $\Xspace_N\otimes \Yspace_h$ is a finite element space of tensor 
product finite elements. The continuous right inverse $\tr_-^{\dagger}$ can be constructed by considering the boundary value problem
\begin{align*}
&-\Delta v + v = 0 && \mbox{in }\Oi,\\
&v=0 &&\mbox{on }\partial\Oi\setminus \bpM,\\
&v=g &&\mbox{on }\bpM,
\end{align*}
which obviously has a unique weak solution by the Lax-Milgram lemma, and setting $\tr_-^{\dagger}g:=v$. 
$\tr_{h,-}^{\dagger}g$ is the finite element approximation to $\tr_-^{\dagger}g$ in 
$\Vint_h$, and \eqref{eq:trace_inverse_conv} holds true because of the convergence of 
the finite  element method. 

\medskip
Now we are in a position to formulate our main convergence theorem:
\begin{theorem}\label{theo:ConvTheo}
Suppose Assumptions \hyperlink{AssA}{A}, \hyperlink{AssB}{B}, \hyperlink{AssC}{C} and \hyperlink{AssD}{D} hold true and assume that the variational equation 
\begin{equation}\label{eq:abstract_cont_prbl}
\stot\paren{\svec{\uint\\ \uext},\svec{\vint\\ \vext}} = F\paren{\svec{\vint\\ \vext}}
\qquad \mbox{for all }\svec{\vint\\ \vext}\in \Vspace
\end{equation}
has at most one solution for all $F\in\Vspace^*$. Then:
\begin{enumerate}
\item Equation \eqref{eq:abstract_cont_prbl} has a unique solution $u=\svec{\uint\\ \uext}\in\Vspace$ for all 
$F\in\Vspace^*$, and $u$ depends continuously on $F$.
\item There exist constants $h_0,C>0$ such that the discrete variational problems 
\begin{equation}\label{eq:abstract_discr_prbl}
\stot\paren{u_{h,N},v_{h,N}} = F\paren{v_{h,N}}\qquad \mbox{for all }
v_{h,N}\in \Vspace_{h,N}
\end{equation}
have unique solutions for all $h\leq h_0$ and all $N\in \setN$, and 
\begin{equation}\label{eq:error_bound}
\norm{u-u_{h,N}}_{\Vspace} 
\leq C\inf_{w_{h,N}\in\Vspace_{h,N}}\norm{u-w_{h,N}}_{\Vspace}.
\end{equation}
Moreover, the right hand side of \eqref{eq:error_bound} tends to $0$ as $h\to 0$ and $N\to\infty$ for all 
$u\in\Vspace$.
\end{enumerate}
\end{theorem}

Note that Theorem \ref{theo:ConvTheo} involves an assumption $h\leq h_0$, which is already necessary 
for the interior problem, but no assumption $N\geq N_0$. 

\medskip
Let us assume that the sesquilinear form $\stot:=\stot_\mupar$ depends on a parameter $\mupar\in\Lambda$ in a subset $\Lambda \subset \setC$. 
We are looking for solutions
$(\mupar,u) \in \Lambda \times \Vspace \setminus \{0\}$ to the continuous 
generalized eigenvalue problem
 \begin{equation}
 \label{eq:EWP}
  \stot_\mupar(u,v)=0 \qquad \mbox{for all }v \in \Vspace.
 \end{equation}
These eigenpairs will be approximated by solutions  
$(\mupar_{h,N},u_{h,N}) \in \Lambda \times \Vspace_h \setminus \{0\}$ to the discrete eigenvalue problems
\begin{equation}\label{eq:diskrEWP}
  \stot_{\mupar_{h,N }} (u_{h,N},v_{h,N})=0 \qquad \mbox{for all }v_{h,N} \in \Vspace_{h,N}.
\end{equation}
Let $\SpLam\subset\Lambda$ denote the set of eigenvalues $\mupar$ 
and $\SpLam_{h,N}$ the set of discrete eigenvalues $\mupar_{h,N}$.

\hypertarget{AssE}{}
{\bf Assumption~E: eigenproblem setting.} \emph{Let $\Lambda\subset \setC$ be open 
and assume that the sesquilinear form $\stot:=\stot_\mupar$ in Assumption~\hyperlink{AssA}{A}
depends on a parameter $\mupar\in \Lambda$.  Moreover, suppose that each 
$\mupar_0\in\Lambda$ 
has a neighborhood $\widehat\Lambda$ in which the vectors 
$\varphi_n$ in Assumption~\hyperlink{AssB}{B} and the quantities $\Cs ,\alpha,\theta_n,\nrGM$ and $K$ 
in Assumption~\hyperlink{AssC}{C} 
can be chosen independently of $\mupar\in\widehat\Lambda$. 
Finally, assume that $\stot_\mupar$ depends holomorphically on $\mupar\in\Lambda$, 
i.e.\ for the operator $T_\mupar:\Vspace \to \Vspace$ defined 
by $\stot_\mupar (u,v)=\lsp T_\mupar u,v\rsp_\Vspace$, $u,v \in \Vspace$  there exist 
for all $\mupar_0 \in \Lambda$ the derivative 
$T_{\mupar_0}':=\lim_{\mupar \to \mupar_0} \frac{1}{\mupar-\mupar_0} (T_\mupar - T_{\mupar_0})$ 
where the limit exists in the norm of $L(\Vspace)$.}

\emph{Verification for \eqref{eq:noguidedmodes}:} If $\stot_\kappa$ is defined by the left hand side of \eqref{eq:noguidedmodes}, 
then it depends holomorpically on $\kappa$. Clearly $\varphi_n$ is independent of $\kappa$, $\nrGM=0$ does not depend on $\kappa$ and $\theta_n$
is not needed. If $\min\setSp(-\Delta_1)=\lambda_1$ and for all $\kappa_0\in \Lambda$ there holds $\Re(\kappa_0^2)<\lambda_1$, then  
$\alpha=\inf_{\kappa \in \widehat \Lambda} \Re(\lambda_1-\kappa^2)/(\lambda_1+2)$ is independent of $\kappa_0$ as well and positive, if $\widehat \Lambda$
is compact. In the same way $\Cs$ and $K$ can be chosen independently of $\kappa\in \widehat \Lambda$.

\begin{rem}
Up to now we have not defined a complex square root and therefore $\kappa_n$ defined in 
\eqref{eq:defkn} for $\kappa>0$ is not defined for $\kappa\in \setC\setminus\setR$. 
We will do this later in Def.~\ref{def:BranchCutSquareRoot}. At this point, we only consider 
\eqref{eq:noguidedmodes} and do not care whether the eigenvalues are physically meaningful.
\end{rem}

\begin{theorem}
\label{Theo:ConvEWP} 
If Assumptions \hyperlink{AssA}{A}, \hyperlink{AssB}{B}, \hyperlink{AssC}{C}, \hyperlink{AssD}{D} and \hyperlink{AssE}{E} 
hold true and if there exists a $\mupar \in \Lambda$ such that $T_\mupar$ is invertible, then $\SpLam\subset\Lambda$ is discrete without
accumulation points and
\begin{equation}\label{eq:Hausdorff_spectra}
\lim_{h\to 0,N\to\infty}\dist(\SpLam\cap\Lambda',\SpLam_{h,N}\cap\Lambda') = 0
\end{equation}
for all compact subsets $\Lambda'\subset\Lambda$. 
Here the Hausdorff distance of two subsets $A,B\subset \setC$ is denoted by 
$\dist(A,B):=\max\{\sup_{a\in A}\inf_{b\in B}|a-b|,\sup_{b\in B}\inf_{a\in A}|a-b|\}$.
\end{theorem}

Further convergence results including convergence of eigenvectors (or eigenspaces), 
multiplicities of eigenvalues, and rates of convergence are intended 
for future research. 

\section{Proof of Theorems \ref{theo:ConvTheo} and \ref{Theo:ConvEWP}}
The convergence theory is based on $\OpS$-coercivity arguments: We are going to construct an isomorphism $\OpS:\Vspace\to\Vspace$ such that 
the sesquilinear form $\stot\paren{\OpS \cdot,\cdot}$ is coercive up to a compact perturbation. Hence, unique solvability of the continuous problem can be shown
as usual using the Lax-Milgram Lemma combined with Riesz-Fredholm theory. An important difficulty is the fact that the discrete 
spaces $\Vspace_{h,N}$ are not invariant under $\OpS$. In order to overcome this difficulty we will introduce later on a 
family of operators $\OpS_{h}:\Vspace_{h,N}\to \Vspace_{h,N}$ with $\|\OpS-\OpS_h\|_{L(\Vspace_{h,N},\Vspace)}\to 0$ for $h\to 0$.

\subsection{Construction and properties of the operator $\OpS$} 
Using the notation of Ass.~\hyperlink{AssB}{B} we define in the exterior domain the operator
\begin{align}
\OpS^{\rm ext}: \Vext\to \Vext,\qquad 
\OpS^{\rm ext} \uext:= \sum_{n=1}^\nrGM \theta_n\uext_n\otimes \varphi_n
+ \sum_{n=\nrGM+1}^\infty \uext_n\otimes \varphi_n \label{def:OpSext}
\end{align}
where $\uext$ has the expansion \eqref{eq:Uexpansion}. Recall, that the rotations $\theta_n$ and the constant $\nrGM$ were introduced in Ass.~\hyperlink{AssC}{C}.
With the operator $\OpS^{\rm ext}$ we have S-coercivity and boundedness of $\sext$ by Assumption~\hyperlink{AssC}{C}:
\begin{align}\label{eq:coercive_ext}
\begin{aligned}
&\Re \paren{\sext(\OpS^{\rm ext}\uext,\uext)}\geq \alpha \|\uext\|_{\Vext}^2\\
&\abs{\sext(\uext,\vext)}\leq \Cs  \|\uext\|_{\Vext} \|\vext\|_{\Vext}
\end{aligned}
\end{align}
for all $\uext,\vext\in \Vext$. Note, that $\OpS^{\rm ext}$ has a bounded inverse given by
\[
[\OpS^{\rm ext}]^{-1} \uext:= \sum_{n=1}^\nrGM \frac{1}{\theta_n}\uext_n\otimes \varphi_n
+ \sum_{n=\nrGM+1}^\infty \uext_n\otimes \varphi_n.
\]
We need to extend $\OpS^{\rm ext}$ to the whole space $\Vspace$ via the trace space $\tracespace$ defined in Assumption~\hyperlink{AssA}{A}. First we define a bounded 
linear operator $\OpS^{\tracespace}:\tracespace\to\tracespace$ such that $\OpS^{\rm ext}$ and $\OpS^{\tracespace}$ intertwine 
with $\tr_+$:
\begin{equation}
\label{eq:DefSZ}
\OpS^{\tracespace}\tr_+ = \tr_+\OpS^{\rm ext}. 
\end{equation}
As $\tracespace = \bigoplus_{n\in\setN}\tr_+(\Vspace_n)$ by Assumption~\hyperlink{AssB}{B} we have to set 
$\OpS^{\tracespace}\psi_n:=\theta_n\psi_n$ for 
all $\psi_n\in \tr_+(\Vext_n)$ and $n\leq \nrGM$ and $\OpS^{\tracespace}\psi_n:=\psi_n$ if $n>\nrGM$.  As for $\OpS^{\rm ext}$ 
it is easy to see that $\OpS^{\tracespace}$ has a bounded inverse. 

Using $\tr_{-}^{\dagger}:\tracespace \to \Vint$ defined in Ass.~\hyperlink{AssD}{D} we can define $\OpS:\Vspace\to\Vspace$ by 
\begin{equation}
\label{eq:DefOpS}
 \OpS \vect{\uint}{\uext} := 
\vect{\uint + \tr_{-}^{\dagger}(\OpS^{\tracespace}-\Id)\tr_{-} \uint}
{\OpS^{\rm ext}\uext}.
\end{equation}
The image of $\Vspace$ under $\OpS$ is in fact contained in $\Vspace$ since
\begin{align*}
\tr_{-}\uint +\tr_{-}\tr_{-}^{\dagger}(\OpS^{\tracespace}-\Id)\tr_{-}\uint 
= \OpS^{\tracespace}\tr_{-}\uint = \OpS^{\tracespace}\tr_+\uext
= \tr_+\OpS^{\rm ext}\uext\,.
\end{align*}
$\OpS$ has the bounded inverse
\[
 \OpS^{-1} \vect{\uint}{\uext} = 
\vect{\uint + \tr_{-}^{\dagger}([\OpS^{\tracespace}]^{-1}-\Id)\tr_{-} \uint}
{\left[\OpS^{\rm ext}\right]^{-1}\uext},
\]
which is easily verified using the identity 
\[
\OpS^{\tracespace}-\Id + [\OpS^{\tracespace}]^{-1}-\Id+ (\OpS^{\tracespace}-\Id)([\OpS^{\tracespace}]^{-1}-\Id)=0.
\]

\subsection{Proof of Theorem \ref{theo:ConvTheo}, part 1} \label{sec:proofConvTheoPart1}
With the substitution $\tilde{u}:= \OpS^{-1}\svec{\uint\\ \uext}$ 
the variational equation \eqref{eq:abstract_cont_prbl} is equivalent to 
\[
\tilde{\stot}\paren{\tilde{u},v}
= F\paren{v},\qquad
\tilde{\stot}\paren{\tilde{u},v}
:= \stot\paren{\OpS\tilde{u},v}
\]
for all $v=(\vint,\vext)\in\Vspace$. 
Since the homogeneous equation is assumed to be uniquely 
solvable, the  bounded linear operator $A:\Vspace\to \Vspace$ defined by 
$\tilde{\stot}(u,v)= \lsp Au,v\rsp_{\Vspace}$ for all 
$u,v\in\Vspace$ is injective.  
Due to \eqref{eq:compact_pert} and \eqref{eq:coercive_ext} we have  
\begin{equation}\label{eq:Gaarding_abstr}
\begin{aligned}
&\Re\tilde{\stot}\paren{\svec{\uint\\ \uext},\svec{\uint\\ \uext}} + 
\Re \lsp (K-\tilde{K})\uint,\uint\rsp_{\Vint} 
\geq \alpha \norm{\svec{\uint\\ \uext}}_{\Vspace}^2,\\
&\tilde{K}:=\tr_-^{\dagger}(\OpS^{\tracespace}-\Id)\tr_-\,
\end{aligned}
\end{equation}
for all $\svec{\uint\\ \uext}\in\Vspace$. 
It follows from the Lax-Milgram lemma that $A + \svec{K-\tilde{K} & 0 \\ 0 &0}$ is boundedly invertible. 
Since $\dim (\OpS^{\tracespace}-\Id)(\tracespace)<\infty$, the operator $\tilde{K}$ is compact. 
Together with the injectivity of $A$ it follows from Riesz-Fredholm theory that $A$ has a bounded inverse. 
This implies the first assertion. 

\subsection{Proof of Theorem \ref{theo:ConvTheo}, part 2} 
We first show sufficient conditions for discrete inf-sup stability for 
general $S$-coercive problems: 
\begin{lemm}\label{lemm:Sh_lower_bound}
Let $\Vspace$ be any complex Hilbert space and $\stot:\Vspace \times \Vspace \to \setC$ a bounded 
sesquilinear form and $\OpS :\Vspace \to \Vspace$ a bounded linear operator.
Moreover, let $\Vspace_h \subset \Vspace$ be a series of closed subspaces,  
$P_h:\Vspace\to \Vspace_h$ the orthogonal projections. Then 
\[
\inf_{\substack{v_h \in \Vspace_h \\ v_h\neq 0}} \sup_{\substack{u_h \in \Vspace_h\\ u_h\neq 0}} 
\frac{\left|\stot(u_h,v_h)\right|}{\|u_h\|_\Vspace \|v_h\|_\Vspace}
   \geq \frac{1}{\|S\|}\paren{\inf_{\substack{v_h \in \Vspace_h \\ v_h\neq 0}} 
   \sup_{\substack{u_h \in \Vspace_h\\ u_h\neq 0}} \frac{\left|\stot(\OpS u_h,v_h)\right|}{\|u_h\|_\Vspace \|v_h\|_\Vspace}    - \|\stot\|\,\|(\Id-P_h)SP_h\|_{L(\Vspace)}}.
\]
In particular, if $\stot$ is $\OpS$-coercive, i.e.\ there exists a constant
$\tilde{\alpha}>0$, such that $\stot(\OpS u, u)\geq \tilde{\alpha} \|u\|_\Vspace^2$ and if
\begin{equation}\label{eq:cond_discrete_infsup_stability}
\lim_{h\to 0} \|(\Id-P_h)SP_h\|_{L(\Vspace)}= 0,
\end{equation}
then $\stot$ is discretely inf-sup stable for sufficiently small $h>0$ with constants independent of $h$ . 
\end{lemm}
\begin{proof}
For all $v_h \in \Vspace_h$ we have
\begin{equation*}
 \begin{aligned}
   \sup_{\substack{u_h \in \Vspace_h\\ u_h\neq 0}} \frac{\left|\stot(u_h,v_h)\right|}{ \|u_h\|_\Vspace}
& \geq  \sup_{\substack{u_h \in \Vspace_h\\ P_h S u_h\neq 0}}  \frac{\left|\stot(\OpS u_h,v_h) - \stot((P_h-I)\OpS P_h u_h, v_h)\right|}{\|P_h S u_h \|_\Vspace} \\
&\geq \frac{1}{\|\OpS\|}\paren{\sup_{\substack{u_h \in \Vspace_h\\ u_h\neq 0}} \frac{\left|\stot(\OpS u_h,v_h)\right|}{\|u_h\|_\Vspace} - \|\stot\|\,\|(P_h-\Id)\OpS P_h \|_{L(\Vspace)} \|v_h\|_{\Vspace}} . 
 \end{aligned}
\end{equation*}
The proposition follows by dividing this inequality by $\|v_h\|_{\Vspace}$ and 
taking the infimum over all $v_h \in \Vspace_h$.
\end{proof}

Now let us verify condition \eqref{eq:cond_discrete_infsup_stability}:
\begin{lemm}\label{lemm:Sstab_conv}
Suppose Assumptions~\hyperlink{AssA}{A}-\hyperlink{AssD}{D} hold true, and let $P_{h,N}: \Vspace\to \Vspace_{h,N}$ denote 
the orthogonal projections. Then 
\[
\lim_{h\to 0}\sup_{N\in\setN} \|(\Id-P_{h,N})\OpS P_{h,N}\|_{L(\Vspace)} = 0.
\]
\end{lemm}
\begin{proof}
Recall, that the discrete spaces $\Vspace_{h,N}$ and the corresponding quantities like the discrete trace operators $\tr_{h,-}^\dagger:\tracespace \to \Vint_h$
as well as the $\Vext_{h,N}=\bigoplus_{l=1}^L\Xspace_{N,l}\otimes \Yspace_{h,l}$ with $\Xspace_{N,l}\subset \Xspace_l^2$
and $\Yspace_{h,l}\subset \Yspace_l^2$ were introduced in Ass.~\hyperlink{AssD}{D}.

Since $P_{h,N}:\Vspace\to \Vspace_{h,N}$ denotes the orthogonal projection, it holds
\begin{equation}\label{eq:var_estim_proj}
\|(\Id-P_{h,N})\OpS P_{h,N}\|_{L(\Vspace)} 
= \sup_{\substack{u\in \Vspace_{h,N}\\ u\neq 0}} 
\frac{\|(\Id-P_{h,N})Su\|_{\Vspace}}{\|u\|_{\Vspace}}
= \sup_{\substack{u\in \Vspace_{h,N}\\ u\neq 0}} \inf_{v\in \Vspace_{h,N}}
\frac{\|Su-v\|_{\Vspace}}{\|u\|_{\Vspace}}.
\end{equation}
To estimate the right hand side of this equation choose $u =\svec{\uint\\ \uext}\in \Vspace_{h,N}$. 
In the unique expansion $\uext=\sum_{n\in\setN}u_n\otimes \varphi_n$ 
(see \eqref{eq:Uexpansion}) all $u_n$ belong to $\Xspace_N$. 
By definition of $\OpS$ in \eqref{def:OpSext} we have
\begin{equation*}
 \OpS^{\rm ext} \uext = \uext + \sum_{n=1}^\nrGM (\theta_n -1 ) u_n\otimes \varphi_n.
\end{equation*}
We set $\OpS^{\rm ext}_h \uext:= \uext+ \sum_{n=1}^\nrGM (\theta_n -1 ) u_n\otimes P^{\Yspace}_h\varphi_n$
with the orthogonal projection $P^{\Yspace}_h:\bigoplus_l\Yspace^2_l \to \bigoplus_l\Yspace_{h,l}$. 
As $\OpS^{\rm ext}_h \uext\in \Vext_{h,N}$, we can set $\vext:=\OpS^{\rm ext}_h\uext$ later. 

Due to Ass.~\hyperlink{AssB}{B} there exists for $n\in \setN$ a $l(n)\in \{1,\dots,L\}$ such that for $\varphi_n=(\varphi_n^{(1)},\dots,\varphi_n^{(L)})$
it holds $\varphi_n^{(j)}= 0$ for $j\neq l(n)$. Using the definition of $\|\cdot\|_{\Vext}$ in \eqref{eq:defi_Vext} we have 
\begin{align*}
&\|\OpS^{\rm ext} \uext-\OpS^{\rm ext}_h\uext\|_{\Vext}
= \norm{\sum_{n=1}^\nrGM  (\theta_n -1 ) u_n \otimes (\varphi_n-P^{\Yspace}_h\varphi_n)}_{\Vext}\\
&\leq 2\sum_{n=1}^\nrGM \norm{u_n\otimes (\varphi_n-P^{\Yspace}_h\varphi_n)}_{\Vext}\\
&= 2\sum_{n=1}^\nrGM \!\!
\paren{\!\|u_n\|_{\Xspace^2_{l(n)}}^2 \!\!\|\varphi_n^{(l(n))}\!\!\!-\!P^{\Yspace_{l(n)}^2}_h\varphi_n^{(l(n))}\|_{\Yspace^1_{l(n)}}^2  \!\!\!\!
+ \|u_n\|_{\Xspace^1_{l(n)}}^2 \!\! \|\varphi_n^{(l(n))}\!\!\!-\!P^{\Yspace_{l(n)}^2}_h\varphi_n^{(l(n))}\|_{\Yspace^2_{l(n)}}^2\!}^{1/2}\\
&\leq 2C_{\nrGM,h}
\sum_{n=1}^\nrGM\paren{ \|u_n\|_{\Xspace^2_{l(n)}}^2 \|\varphi_n^{(l(n))}\|_{\Yspace^1_{l(n)}}^2  
+ \|u_n\|_{\Xspace^1_{l(n)}}^2 \|\varphi_n^{(l(n))}\|_{\Yspace^2_{l(n)}}^2 }^{1/2} \\
&=2C_{\nrGM,h} \sum_{n=1}^\nrGM \|u_n\otimes \varphi_n  \|_{\Vext}%\\
%&\quad
\leq 2C_{\nrGM,h} \sqrt{M}\paren{\sum_{n=1}^\nrGM \|u_n\otimes \varphi_n  \|_{\Vext}^2}^{1/2}\leq 2C_{\nrGM,h} \sqrt{M}\|u\|_{\Vext}
\end{align*}
with
\begin{equation*}
 C_{\nrGM,h}:=\max_{n=1\dots\nrGM}\max \left\{
\frac{\|\varphi_n^{(l(n))}-P^{\Yspace_{l(n)}^2}_h\varphi_n^{(l(n))}\|_{\Yspace^1_{l(n)}}}{\|\varphi_n^{(l(n))}\|_{\Yspace^1_{l(n)}}} ,  
\frac{\|\varphi_n^{(l(n))}-P^{\Yspace_{l(n)}^2}_h\varphi_n^{(l(n))}\|_{\Yspace^2_{l(n)}}}{\|\varphi_n^{(l(n))}\|_{\Yspace^2_{l(n)}}} \right\} .
\end{equation*}
Due to the density $\bigcup_h\Yspace_{h,l}\subset\Yspace^2_l$, the finiteness of 
$\nrGM$ and the continuity of the embeddings $\Yspace^2_l\hookrightarrow \Yspace^1_l$, we have
$\lim_{h\to 0} C_{\nrGM,h} =0$, i.e.\
\begin{equation}\label{eq:Sext_conv}
\lim_{h\to 0}\sup_{N\in\setN}\|\OpS^{\rm ext} -\OpS^{\rm ext}_h\|_{L(\Vext_{h,N},\Vext)} = 0.
\end{equation}
We define $\OpS_{h}:\Vspace_{h,N}\to \Vspace_{h,N}$ by 
\[
\OpS_{h}\begin{pmatrix}\uint \\ \uext\end{pmatrix}
:=\begin{pmatrix} \uint + \tr_{h,-}^{\dagger}(\tr_+\OpS^{\rm ext}_h\uext-\tr_-\uint) \\
\OpS^{\rm ext}_h\uext\end{pmatrix}.
\]
Then using $\tr_-\uint = \tr_+\uext$ we have
\begin{align*}
\left[(\OpS-\OpS_{h})\smat{\uint\\ \uext}\right]^{\rm int}
&= \paren{\tr_{-}^{\dagger}\tr_+\OpS^{\rm ext} - \tr_{h,-}^{\dagger}\tr_+\OpS^{\rm ext}_h}\uext
 + \paren{\tr_{h,-}^{\dagger}-\tr_-^{\dagger}}\tr_-\uint\\
&= \paren{\tr_-^{\dagger}-\tr_{h,-}^{\dagger}}\tr_+(\OpS^{\rm ext}-\Id)\uext
+ \tr_{h,-}^{\dagger}\tr_+(\OpS^{\rm ext}-\OpS^{\rm ext}_h)\uext
\end{align*}
Since the range of $\tr_+(\OpS^{\rm ext}-\Id)$ is finite dimensional and $\tr_-^{\dagger}-\tr_{h,-}^{\dagger}$ 
converges point wise to $0$ \eqref{eq:trace_inverse_conv}, we have 
$\lim_{h\to 0}\|(\tr_-^{\dagger}-\tr_{h,-}^{\dagger})\tr_+(\OpS^{\rm ext}-\Id)\|_{L(\Vext,\Vint)}= 0$. 
Moreover, by the uniform boundedness principle $\sup_{h>0}\|\tr_{h,-}^{\dagger}\|_{L(\tracespace,\Vint)}<\infty$. 
Together with \eqref{eq:Sext_conv} this implies 
\begin{equation}\label{eq:Sint_conv}
\lim_{h\to 0}\sup_{N\in\setN}\norm{\left[\OpS-\OpS_{h}\right]^{\rm int}}_{L(\Vspace_{h,N},\Vint)} = 0.
\end{equation} 
Setting $v:=S_hu$ in \eqref{eq:var_estim_proj} and combining \eqref{eq:Sext_conv} and \eqref{eq:Sint_conv} we obtain
\[
\sup_{N\in\setN}\|(\Id-P_{h,N})\OpS P_{h,N}\|_{L(\Vspace)} 
\leq \sup_{N\in\setN}\|\OpS-\OpS_h\|_{L(\Vspace_{h,N},\Vspace)}
\stackrel{h\to 0}{\longrightarrow} 0.
\]
\end{proof}

\begin{lemm}\label{lemm:density}
Under Assumptions~\hyperlink{AssA}{A}-\hyperlink{AssD}{D} $\bigcup_{h>0,N\in\setN}\Vspace_{h,N}\subset\Vspace$ is dense. 
\end{lemm}

\begin{proof}
Assume that $\lsp w,u_{h,N}\rsp=0$ for all $u\in \Vspace_{h,N}$ and all $h,N$ 
for some $w\in\Vspace$. In particular
\[
0 = \lsp w^{\rm int},\tr_{h,-}^{\dagger}\tr_+(v_N\otimes \psi_h)\rsp_{\Vint}
+ \lsp w^{\rm ext}, v_N\otimes \psi_h\rsp_{\Vext}
\]
for all $v_N\in \Xspace_N$ and $\psi_h\in \Yspace_h$. Due to the form of the inner product 
of $\Vext$, the assumptions on $\Xspace_N$ and $\Yspace_h$ and the point wise convergence of 
$\tr_{h,-}^{\dagger}$, we have 
\begin{equation}\label{eq:aux_density}
0= \lsp w^{\rm int},\tr_{-}^{\dagger}\tr_+(\uext)\rsp_{\Vint}
+ \lsp w^{\rm ext}, \uext\rsp_{\Vext}
\end{equation}
first for all $\uext$ of the form $\uext=v\otimes\psi$ with $v\in \Xspace^2$ and 
$\psi\in\Yspace^2$ and then by density of $\Xspace^2\subset\Xspace^1$ and 
$\Yspace^2\subset\Yspace^1$ for all $\uext\in\Vext$.

For a given $u:=(\uint, \uext)^\top \in \Vspace$ we obtain with \eqref{eq:aux_density}, $\tr_-\uint=\tr_+\uext$
and the density assumption on $\Vint_h$
$$\lsp w,u\rsp_\Vspace=\lsp w^{\rm int} , \uint -\tr_-^\dagger \tr_- \uint\rsp_{\Vint}=0.$$
This shows that $w=0$. Hence the orthogonal complement of $\bigcup_{h,N}\Vspace_{h,N}$ 
is $\{0\}$, i.e.\ this space is dense in $\Vspace$. 
\end{proof}

We can now complete the proof of the second part of Theorem \ref{theo:ConvTheo} as follows:
Due to \eqref{eq:Gaarding_abstr} and Lemmas~\ref{lemm:Sh_lower_bound} and \ref{lemm:Sstab_conv} the discrete inf-sup constants of the variational problems 
\begin{equation}\label{eq:aux_problems}
\stot\paren{u_{h,N},v_{h,N}} + 
 \lsp (K-\tilde{K})\uint_h,\vint_h\rsp_{\Vint} = F(v_{h,N}),
\qquad v_{h,N}=\svec{\vint_h\\ \vext_{h,N}}\in \Vspace_{h,N}
\end{equation}
are uniformly bounded away from $0$ for $h\leq h_0$. Therefore, these variational 
equations have unique solutions $u_{h,N}\in \Vspace_{h,N}$ for all $h\leq h_0$, 
and together with the density lemma~\ref{lemm:density} it follows that the 
Galerkin method \eqref{eq:aux_problems} converges,  
and the error bound \eqref{eq:error_bound} holds true for this  
modified problem (see e.g.\  \cite[Theorems 13.6]{Kress}). 
Since $K-\tilde{K}$ is compact the Galerkin method \eqref{eq:abstract_discr_prbl} 
for the original problem \eqref{eq:abstract_cont_prbl} converges as well with error bound
\eqref{eq:error_bound} (see e.g.\ \cite[Theorems 13.6 and 13.7]{Kress}).

\subsection{Proof of Theorem \ref{Theo:ConvEWP}}\label{Sec:Res}
For the following we need in addition to Ass.~\hyperlink{AssA}{A}-\hyperlink{AssD}{D} the Ass.~\hyperlink{AssE}{E} for the eigenvalue setting. Recall, that $\SpLam$ denotes the set of eigenvalues $\kappa$ of
$\stot_\kappa(u,v)=0$, $v \in  \Vspace$, with eigenfunction $u\in \Vspace\setminus\{0\}$. 
Moreover, if there exists a $\kappa\in \Lambda$ such that the operator $T_\mupar:\Vspace \to \Vspace$ defined 
by $\stot_\mupar (u,v)=\lsp T_\mupar u,v\rsp_\Vspace$, $u,v \in \Vspace$, is invertible, than $\SpLam$ is discrete without accumulation points by analytic
Fredholm theory (see e.g. \cite[Part III, Cor.~XI.8.4]{Gohbergetal:90}). Note, that we have shown in Sec.~\ref{sec:proofConvTheoPart1}, that $T_\mupar$ is 
a Fredholm operator for all $\mupar \in \Lambda$.

As opposed to some other eigenvalue convergence results 
(see e.g.\ \cite[Chapter 11]{Hackbusch}) 
some complications arise since we do not have a compact embedding assumption 
in the exterior domain. (Recall that e.g.\ $H^1(\Oe)\hookrightarrow L^2(\Oe)$ 
is not compact due to the unboundedness of $\Oe$.) We could use as in \cite[Sec.~4]{SteinbachUnger:12} the abstract framework of \cite{Karma:96}.
Nevertheless, in order to be self-consistent we present here the proofs in our framework.

Let us define
\begin{equation*}
 \beta(\mupar):=\inf_{\substack{u \in \Vspace \\ u\neq 0}} \sup_{\substack{v \in \Vspace\\ v\neq 0}} 
\frac{\left|\stot_\mupar(u,v)\right|}{\|u\|_\Vspace \|v\|_\Vspace},\qquad 
 \beta_{h,N}(\mupar):=\inf_{\substack{u \in \Vspace_{h,N} \\ u\neq 0}} \sup_{\substack{v \in \Vspace_{h,N}\\ v\neq 0}} 
\frac{\left|\stot_\mupar(u,v)\right|}{\|u\|_\Vspace \|v\|_\Vspace}. 
\end{equation*} 
As a consequence of Theorem \ref{theo:ConvTheo} the operators 
$T_{\mupar}$ have a bounded inverse for all $\mupar\in\Lambda\setminus\SpLam$, 
and by a Neumann series argument the mapping $\mupar\mapsto T_{\mupar}^{-1}$ is 
holomorphic on $\Lambda\setminus\SpLam$. 
As $\beta(\mupar) = \|T_{\mupar}^{-1}\|^{-1} $ and 
$\|\stot_{\mupar}\|=\|T_{\mupar}\|_{L(\Vspace)}$, we have
\begin{align}\label{eq:defiC}
&\beta \mbox{ is continuous on } \Lambda\setminus \SpLam\\
&C_{\Lambda'}:=\sup\{\|\stot_{\mupar}\|~|~\mupar\in\Lambda'\} \mbox{ is finite}
\end{align}
for all compact $\Lambda'\subset\Lambda$. 

\begin{lemm}\label{lemm:beta_ieq}
Under the assumptions of Theorem \ref{Theo:ConvEWP} suppose that 
$\inf\{\beta(\mupar):\mupar\in\widehat\Lambda\}>0$ for some 
compact subset $\widehat\Lambda\subset\Lambda$ as in Assumption~\hyperlink{AssE}{E}. 
Then there exist constants $\rho, h_0,N_0>0$ such that
\[
\beta_{h,N}(\mupar)\geq \rho\qquad \mbox{for all } h\leq h_0, N\geq N_0,
\mupar\in\widehat\Lambda.
\]
\end{lemm}

\begin{proof}
Note that by Assumption~\hyperlink{AssE}{E} the operator $\OpS$ is independent of $\mupar$.
From Lemmas \ref{lemm:Sh_lower_bound} and \ref{lemm:Sstab_conv} we deduce  
that there exists $\eta(h,N)$ independent of $\mupar \in \widehat\Lambda$ 
with $\eta(h,N)\to 0$ for $h \to 0$ and $N \to \infty$ such that
\begin{equation}
 \beta_{h,N}(\mupar)\geq \frac{1}{\|\OpS\|} \inf_{\substack{u \in \Vspace_{h,N} \\ u\neq 0}} \sup_{\substack{v \in \Vspace_{h,N}\\ v\neq 0}} 
\frac{\left|\stot_\mupar(\OpS u,v)\right|}{\|u\|_\Vspace \|v\|_\Vspace} -\eta(h,N).
\end{equation}
Therefore, the proof is done if we can show the assertion for $\tilde \stot_\mupar:=\stot_\mupar(\OpS \bullet,\bullet)$ and
\begin{equation}
  \tilde \beta_{h,N}(\mupar):= \inf_{\substack{u \in \Vspace_{h,N} \\ u\neq 0}} \sup_{\substack{v \in \Vspace_{h,N}\\ v\neq 0}} 
\frac{\left|\tilde \stot_\mupar(u,v)\right|}{\|u\|_\Vspace \|v\|_\Vspace}.
\end{equation}
Equivalently, if we define $\tilde T(\mupar):\Vspace \to \Vspace$ by $\tilde \stot_\mupar(u,v)=\lsp \tilde T(\mupar) u,v\rsp$ for all $u,v \in \Vspace$, and
$\tilde T_{h,N}(\mupar):=P_{h,N}\tilde  T(\mupar):\Vspace_{h,N} \to \Vspace_{h,N}$, 
we have to show due to $\tilde \beta_{h,N}(\mupar) = \|\tilde  T_{h,N}(\mupar)^{-1}\|^{-1}$, 
that there exist $h_0,N_0,\rho>0$ independent of $\mupar\in\widehat\Lambda$ such
that$ \|\tilde  T_{h,N}(\mupar)^{-1}\| \leq 1/\rho$ for all $h\leq h_0$, $N\geq N_0$. 
If $\widehat\Lambda$ is a singleton, the assertion follows from 
\cite[Theorem 13.7(2)]{Kress}. For compact $\widehat\Lambda$ we can argue similarly 
keeping track of dependencies on $\mupar \in \widehat\Lambda$. 

Using $\hat K:=\svec{K-\tilde{K} & 0 \\ 0 &0}$ as in 
Sec.~\ref{sec:proofConvTheoPart1}, $A(\mupar):=\tilde  T(\mupar)+\hat K$ and 
$A_{h,N}(\mupar):=P_{h,N} A(\mupar)$ we can factorize
\begin{equation}\label{eq:factorizationT}
 \tilde  T_{h,N}(\mupar)=P_{h,N} \tilde  T(\mupar) = A_{h,N}(\mupar) \left(\Id - A_{h,N}(\mupar)^{-1} P_{h,N} \hat K \right).
\end{equation}
By Assumptions \hyperlink{AssA}{A}-\hyperlink{AssE}{E} we have
\begin{equation*}
 \begin{aligned}
 &\|A(\mupar)\| \leq C_{\widehat\Lambda}\|\OpS\|   +\|\hat K\|,
 \qquad \|A(\mupar)^{-1}\|\leq \alpha^{-1},\\
 &\|A_{h,N}(\mupar)^{-1}\|\leq \alpha^{-1}, \qquad
 \|\tilde T(\mupar)^{-1}\|=\tilde \beta(\mupar)^{-1} 
\leq \|\OpS^{-1}\|/\rho,\\
 &\|\left(\Id - A(\mupar)^{-1} \hat K\right)^{-1}\|\leq \|A(\mupar)\| \|\tilde T(\mupar)^{-1}\|  \leq 
 \rho^{-1}\|\OpS\| \|\OpS^{-1}\|  (C_{\widehat\Lambda} +\|\hat K\|) 
 \end{aligned}
\end{equation*}
for all $\mupar \in \widehat\Lambda$. By Galerkin orthogonality and coercivity we have
\begin{equation*}
\| A_{h,N}(\mupar)^{-1} P_{h,N} f - A(\mupar)^{-1} f \|  \leq \frac{\|A(\mupar)\|}{\alpha} \inf_{v \in \Vspace_{h,N}} \| A(\mupar)^{-1} f -v\|_{\Vspace},\quad \text{for all }f \in \Vspace.
\end{equation*}
Therefore by density of $\bigcup_{h,N}\Vspace_{h,N} \subset \Vspace$, compactness of $\widehat\Lambda$, and continuity of $\mupar\mapsto A(\mupar)^{-1}$ we have
\begin{equation}\label{eq:pointwise_conv_A}
\lim_{h\to 0,N\to\infty}\sup_{\mupar\in\widehat\Lambda}
\| A_{h,N}(\mupar)^{-1} P_{h,N} f - A(\mupar)^{-1} f \|_{\Vspace} =0
\qquad \mbox{for all }f\in\Vspace.
\end{equation}
We will show that this implies 
\begin{equation}\label{eq:norm_conv_A}
\lim_{h\to 0,N\to\infty}\sup_{\mupar\in\widehat\Lambda}
\| (A(\mupar)^{-1}  - A_{h,N}(\mupar)^{-1} P_{h,N})\hat K\|_{L(\Vspace)} = 0.
\end{equation}
In fact, for given $\epsilon>0$ the relatively compact set 
$U:=\{ \hat K f~|~\|f\|\leq 1\}$ can be covered by a finite number of balls 
$B_r(f_m)$, $m=1,\dots M(\epsilon)$ with radius 
$r:=\epsilon/3\sup_{\mupar\in\widehat\Lambda}\|A(\mupar)\|$. Due to 
\eqref{eq:pointwise_conv_A} there exist $h_0,N_0>0$ such that 
$\| A_{h,N}(\mupar)^{-1} P_{h,N} f_j - A(\mupar)^{-1} f_j\|\leq \epsilon/3$ for 
all $j=1,\dots,M(\epsilon)$, $h\leq h_0$, $N\geq N_0$ and $\mupar\in\widehat\Lambda$. 
Since all $f\in U$ are contained in some ball $B_r(f_j)$, we have
\begin{align*}
&\|A_{h,N}(\mupar)^{-1}P_{h,N}f-A(\mupar)^{-1}f\|
\leq \|A_{h,N}(\mupar)^{-1}P_{h,N}(f-f_j)\|\\
&+ \|A_{h,N}(\mupar)^{-1}P_{h,N}f_j-A(\mupar)^{-1}f_j\|
+ \|A(\mupar)^{-1}(f_j-f)\|
\leq 2 (C_{\widehat\Lambda}\|\OpS\|+\|\hat{K}\|)r+\frac{\epsilon}{3} = \epsilon
\end{align*}
completing the proof of \eqref{eq:norm_conv_A}. Hence by
a Neumann series argument (see \cite[Theorem 10.1]{Kress}) we have
\begin{equation*}
 \|\left(\Id - A_{h,N}(\mupar)^{-1} P_{h,N} \hat K\right)^{-1}\| \leq 
 \frac{\|\left(\Id - A(\mupar)^{-1} \hat K\right)^{-1}\|}{1- \|\left(\Id - A(\mupar)^{-1} \hat K\right)^{-1} \left(  A(\mupar)^{-1} \hat K - A_{h,N}(\mupar)^{-1} P_{h,N} \hat K\right)\|},
\end{equation*}
if the denominator is positive. By \eqref{eq:norm_conv_A} there are 
$h_0,N_0>0$ such that the denominator is $\geq \frac{1}{2}$ for all 
$h\leq h_0$ and $N\geq N_0$. 
In view of \eqref{eq:factorizationT} this implies uniform boundedness of 
$\|\widetilde T_{h,N}(\mupar)^{-1}\|$ in $h,N$, and $\mupar$. 
\end{proof}

\begin{prop}\label{prop:EW}
Under the assumptions of Theorem \ref{Theo:ConvEWP} the following holds true:
\begin{enumerate}
\item If there exists a sequence $(\mupar_{h,N}) \subset \Lambda'$ of discrete eigenvalues to \eqref{eq:diskrEWP} converging to $\mupar_0 \in \Lambda'$
as $h\to 0$ and $N\to\infty$, then $\mupar_0\in\SpLam$.
\item For each $\mupar_0\in \SpLam$ there exists a sequence 
$(\mupar_{h,N})_{h,N} \subset \Lambda'$ of discrete eigenvalue to \eqref{eq:diskrEWP} converging to $\mupar_0$.
\item For all $\mupar_0 \in \Lambda'\setminus\SpLam$ there exist constants 
$h_0,N_0,\epsilon>0$ such that the set $\{\mupar\in\Lambda':|\mupar-\mupar_0|<\epsilon\}$ 
contains no discrete eigenvalues for $h\leq h_0$ and $N\geq N_0$. 
\end{enumerate}
\end{prop}
\begin{proof}
\emph{Part 1:} 
Let $(\mupar_{h,N}) \subset \Lambda$ be a sequence of discrete eigenvalues 
converging to $\mupar_0 \in \Lambda$ and assume that $\mupar_0\notin \SpLam$. 
Then $\beta(\mupar_0)=\|T_{\mupar_0}^{-1}\|^{-1}>0$, and by continuity of $\beta$ 
at $\mupar=\mupar_0$ there exists $\epsilon>0$ such that 
$\inf\{\beta(\mupar)>0~|~\mupar\in\Lambda,|\mupar -\mupar_0|\leq \epsilon\}>0$. 
W.l.o.g.\ $B_{\epsilon}(\mupar_0):=\{\mupar\in\setC ~|~|\mupar-\mupar_0|< \epsilon\}$ 
is contained in some $\widehat\Lambda$ from 
Assumption~\hyperlink{AssE}{E}. Then due to Lemma \ref{lemm:beta_ieq} there exist $h_0,N_0>0$ such that 
$\beta_{h,N}(\mupar_{h,N})>0$ for all $h\leq h_0$
and $N\geq N_0$. This implies that the unique solution to \eqref{eq:diskrEWP} 
for such $h$ is $u_{h,N}=0$, contradicting the assumption that 
$\mupar_{h,N}$ is a discrete eigenvalue.

\emph{Part 2:} 
If $\mupar_0\in \SpLam$, then $\beta(\mupar_0)=0$ 
and due to discreteness of $\SpLam$ and holomorphy of 
$\mupar\mapsto T_{\mupar}^{-1}$ on $\Lambda\setminus \SpLam$ 
there exists $\epsilon>0$ such that $\beta(\mupar)>0$ for all 
$\mupar\in \overline{B_{\epsilon}(\mupar_0)}\setminus\{\mupar_0\}$. 
Again, we may assume that the independence properties of Assumption~\hyperlink{AssE}{E}
hold in $\overline{B_{\epsilon}(\mupar_0)}$.
By continuity of $\beta$ and compactness of
$\partial B_\epsilon(\mupar_0)$ the number 
$\rho^\epsilon:=\min\{\beta(\mupar)~|~\mupar \in \partial B_\epsilon(\mupar_0)\}$ 
is strictly positive. By Lemma \ref{lemm:beta_ieq} there exists $\rho>0$ 
such that $\beta_{h,N}(\mupar)\geq \rho$ 
for all $\mupar \in  \partial B_\epsilon(\mupar_0)$, 
$h\leq h_0$ and $N\geq N_0$. Let $u_0\in \Vspace$ %, $\|\u_0\|_{\Vspace}=1$ 
be an eigenvector corresponding to $\mupar_0$, i.e.\ 
$\stot_{\mupar_0}(u_0,v)=0$ for all $v\in\Vspace$. Then with 
$C_{\Lambda}:=\sup\{\|\stot_{\mupar}\|~|~\mupar\in\overline{\Lambda}\}$
\[
\beta_{h,N}(\mupar_0)
\leq \sup_{\substack{v \in \Vspace_{h,N}\\ v\neq 0}}
\frac{|\stot_{\mupar_0}(P_{h,N}u_0,v)|}{\|P_{h,N}u_0\|_{\Vspace}\|v\|_{\Vspace}}
= \sup_{\substack{v \in \Vspace_{h,N}\\ v\neq 0}}
\frac{|\stot_{\mupar_0}(P_{h,N}u_0-u_0,v)|}{\|P_{h,N}u_0\|_{\Vspace}\|v\|_{\Vspace}}
\leq C_{\Lambda}\frac{\|P_{h,N}u_0-u_0\|_{\Vspace}}{\|P_{h,N}u_0\|_{\Vspace}}.
\]
The right hand side converges to $0$, and hence for sufficiently small $h$ and 
large $N$ we have 
\begin{equation}\label{eq:aux_EVP_discr}
\beta_{h,N}(\mupar) \geq \rho>   \beta_{h,N}(\mupar_0)
\qquad \mbox{for all }\mupar \in \partial B_\epsilon(\mupar_0).
\end{equation}
Assume that for some such $h$ there exist no discrete eigenvalues 
in $\overline{B_{\epsilon}(\mupar_0)}$. Let $\tilde{T}_{\mupar}$ be a matrix 
representing $\stot$ on $\Vspace_{h,N}$ with respect to some fixed basis. 
Then $\tilde{T}_{\mupar}$ is invertible for all $\mupar\in B_{\epsilon}(\mupar_0)$ 
and since $\tilde{T}_{\mupar}$ depends holomorphically on $\mupar$, so does 
$\tilde{T}_{\mupar}^{-1}$. Moreover, $\beta_{h,N}(\mupar) = \|\tilde{T}_{\mupar}^{-1}\|^{-1}$. 
It follows from Cauchy's integral formula
$
\tilde{T}_{\mupar_0}^{-1}
=\frac{1}{2\pi i}\int_{\partial B_{\epsilon}(\mupar_0)}
\frac{d\mupar}{\mupar-\mupar_0}\tilde{T}_{\mupar}^{-1}
$
that $\|\tilde{T}_{\mupar_0}^{-1}\|\leq \sup\{\|\tilde{T}_{\mupar}^{-1}\|:
\mupar\in \partial B_{\epsilon}(\mupar_0)\}$. This contradicts \eqref{eq:aux_EVP_discr}.

\emph{Part 3:} Suppose the assertion is false for some 
$\mupar_0\in\Lambda\setminus \SpLam$. Then there exists a sequence 
of discrete eigenvalues $(\mupar_{h,N})$ converging to $\mupar_0$ as $h\to 0$ and 
$N\to\infty$, 
and with the help of part 1 we obtain the contradiction $\mupar_0\in\SpLam$. 
\end{proof}

With the help of Proposition \ref{prop:EW} the proof of Theorem \ref{Theo:ConvEWP} 
is a straightforward compactness argument: Part 2 implies that 
$\sup_{\mupar\in\SpLam\cap\Lambda'}\inf_{\mupar_{h,N}\in\SpLam_{h,N}\cap\Lambda'}|\mupar-\mupar_{h,N}|\to 0$. 
Given $\delta>0$ sufficiently small we can use compactness
of $\Lambda'$ to obtain a finite covering of 
$\Lambda'\setminus \bigcup_{\mupar\in\SpLam\cap\Lambda'}B_{\delta}(\mupar)$ by 
balls described in part 3. Since none of these balls contains a 
discrete eigenvalue in the limit $h\to 0$ and $N\to\infty$, it follows that 
$\sup_{\mupar_{h,N}\in\SpLam_{h,N}\cap\Lambda'}\inf_{\mupar\in\SpLam\cap\Lambda'}
|\mupar-\mupar_{h,N}|\leq \delta$. As $\delta>0$ was arbitrary, the limit is $0$. 

\section{Complex scaling/ PML}
\label{sec:PML}
In this section we first apply Theorem~\ref{theo:ConvTheo} to a Perfectly Matched Layer (PML) discretization
of the diffraction problem \eqref{eqs:scat_prbl} in Sec.~\ref{sec:setting}. In the literature there exist
already some convergence results for such problems  (see e.g.\ \cite{BNiBonLeg:2004}).
However, in our approach the truncation error is treated as an approximation error and not as an error on 
the continuous level. Therefore, the techniques used in \cite{LassasSomersalo:98,CW:03,BNiBonLeg:2004,PC2,KimPasciak:09,kalvin:11}  to handle 
this modeling error are not needed.

Moreover, since the PML method is better known than the Hardy space method presented in the next section, 
this section may help to follow the framework of the Hardy space method.

We will be particularly interested in complex frequencies $\kappa \in \setC$ with positive real part representing 
the angular frequency and non positive imaginary part representing a damping in time. 
Since the radiation condition Def.~\ref{defi:modal_rad_cond} 
is only defined for positive frequencies $\kappa$, we have to define a proper holomorphic extension. Last we 
formulate the variational framework and prove the Assumptions \hyperlink{AssA}{A}-\hyperlink{AssE}{E} of Sec.~\ref{sec:ConvTheo}. Theorem~\ref{theo:ConvTheo} and Theorem \ref{Theo:ConvEWP}
yield convergence for discrete solutions to the diffraction problem as well as to the corresponding resonance problem.

\subsection{complex scaling radiation condition}\label{sec:scaling_RC}
In this and the following 
subsection we consider for simplicity the case of a single waveguide 
$\Oe:=(0,\infty)\times \bp$ with left boundary $ \bpM:=\{0\}\times \bp$, 
but without an interior domain:
\begin{subequations}
\label{eq:WGproblem}
 \begin{align}
&- \Delta \us - \kappa^2 \us = 0 && \mbox{in } \Oe, \label{eq:Helmholtz} \\
& \Bdv \us = 0 && \mbox{on }\partial\Oe\setminus\bpM, \label{eq:HelmholtzBC}\\
&  \frac{ \partial \us}{\partial x} = g_{\rm in} && \mbox{on } \bpM, \label{eq:WGproblemDV}\\
&  \us \text{ satisfies a radiation condition.} \label{eq:abstractrc}
 \end{align}
\end{subequations}
If we use the modal radiation condition (see Definition \ref{defi:modal_rad_cond}), then plugging \eqref{eq:WGproblemDV} 
into \eqref{eq:expansion_u} yields $c_n= \lsp g_{\rm in},\varphi_n\rsp_{L^2(\bp)}/(i \kappa_n)$, and we obtain
the unique solution 
\begin{equation}
\label{eq:seriesrepr}
\us(x,y)=\sum_{n=1}^\infty \frac{\lsp g_{\rm in},\varphi_n\rsp_{L^2(\bp)}}{i \kappa_n} e^{i \kappa_n x} \varphi_n (y),\qquad (x,y) \in \Oe.
\end{equation}
\begin{defi}[complex scaling radiation condition]
\label{Def:RCcomplexScaling}
 Let $\sigma\in \setC$ with $\Re(\sigma)>0$ and $\Im(\sigma)>0$ be the complex scaling parameter. %, and suppose 
A function $u \in H^2_{\rm loc}(\Oe\cup\bpM)$ satisfies the complex scaling radiation condition with 
parameter $\sigma$ if the mapping $(0,\infty)\to L^2(\bp)$, $x\mapsto u(x,\cdot)$ has a holomorphic extension 
$\mathcal{S}_{\sigma}\to L^2(\bp)$ to an open set $\mathcal{S}_{\sigma}\subset \setC$ containing  
$\{z \in \setC\setminus\{0\}~|\arg(z)\in [0,\arg(\sigma)]\}$ and if the function
\begin{equation}
 u_{\sigma}(x,y):=u(\sigma x,y),\qquad (x,y) \in \Oe,
\end{equation}
belongs to $H^2(\Oe)$.
\end{defi}
Complex scaling was used in quantum physics since the 1970s (see e.g.\ \cite{HislopSigal:96,moiseyev:98}) and reintroduced 
by B\'erenger \cite{Berenger:94} under the name Perfectly Matched Layer (PML). For time-depending problems, the complex scaling parameter 
is typically chosen
frequency dependent. Since for resonance problems this would lead to nonlinear eigenvalue problems, we avoid the 
incorporation of the frequency into the
complex scaling. Moreover, due to the waveguide structure we may have several wavenumbers and it is not clear a priori, 
which of these should be used in the complex scaling.
\begin{lemm}\label{lemm:equivalence_complex_scal_rc}
Let $\sigma\in \setC$ with $\Re(\sigma)>0$ and $\Im(\sigma)>0$ be any complex scaling parameter, let $\kappa>0$, and 
assume \eqref{eq:noWoodAnomaly}.
Then  a solution $\us\in H^1_{\rm loc}(\Oe\cup\bpM)$ to  \eqref{eq:Helmholtz}, \eqref{eq:HelmholtzBC}, and \eqref{eq:WGproblemDV} 
satisfies the modal radiation condition (see Def.~\ref{defi:modal_rad_cond}) if and only if it satisfies the complex 
scaling radiation condition with parameter $\sigma$ (see Def.~\ref{Def:RCcomplexScaling}).
In this case $u_{\sigma}$ satisfies 
\begin{subequations}
\label{eq:PMLproblem}
 \begin{align}
& - \frac{1}{\sigma^2} \partial_x^2  u_{\sigma} - \Delta_{\bp}  u_{\sigma} - \kappa^2 u_{\sigma} = 0 && \mbox{in } (0,\infty)\times \bp,\\
& \Bdv u_{\sigma} = 0 && \mbox{on }(0,\infty) \times \partial \bp, \\
&  \frac{ \partial u_{\sigma}}{\partial x} = \sigma g_{\rm in} && \mbox{on } \{0\}\times \bp
 \end{align}  
\end{subequations}
and is given explicitly by
 \begin{equation}\label{eq:PMLseries}
   u_{\sigma}(x,y)=\sum_{n=1}^\infty \frac{\lsp g_{\rm in},\varphi_n\rsp_{L^2(\bp)}}{i \kappa_n} e^{i \sigma \kappa_n x} \varphi_n (y),\qquad (x,y)\in \Oe.
 \end{equation}
Vice versa, any solution $u_{\sigma}\in H^2(\Oe\cup\Gamma)$ 
to \eqref{eq:PMLproblem} corresponds to the holomorphic 
extension of a solution to \eqref{eq:WGproblem}.
\end{lemm}
\begin{proof}
First assume that $\us$ satisfies Def~\ref{defi:modal_rad_cond}. To show that the right hand side of 
\eqref{eq:seriesrepr} is holomorphic in $x\in \setC\setminus\{0\}$ if $\arg(x) \in [0,\arg(\sigma)]$ it 
suffices to show that the series and its formal complex derivative are absolutely convergent locally uniformly in $x$ in 
the sense that for each $x$ there exist $\epsilon,C>0$ such that 
$\sum_{n=1}^\infty  |\lsp g_{\rm in},\varphi_n\rsp|^2\kappa_n^{-2}|e^{i\sigma\kappa_n\tilde{x}}|^2\leq  C$ 
and  $\sum_{n=1}^\infty|\lsp g_{\rm in},\varphi_n\rsp|^2|e^{i\sigma\kappa_n\tilde{x}}|^2\leq C$ for all 
$\tilde{x}\in\setC$ with $|x-\tilde{x}|<\epsilon$.  Note that $\Re(i \sigma \kappa_n)<0$ for all $n \in \setN$ and 
$\lim_{n\to\infty}\Re(i \sigma \kappa_n)/\sqrt{\lambda_n} =-\Re\sigma$. Hence the uniform bounds follow 
from the Weyl asymptotics of the eigenvalues $\lambda_n$ and 
$\|g_{\rm in}\|_{H^{1/2}}=\sum_{n=1}^\infty (1+\lambda_n)^{1/2}|\lsp g_{\rm in},\varphi_n\rsp|^2<\infty$, i.e.\  
the right hand side of \eqref{eq:seriesrepr} defines the required holomorphic extension. 
Moreover, it is easy to see that $u_{\sigma}$ satisfies \eqref{eq:PMLproblem} 
and \eqref{eq:PMLseries} and belongs to $H^2(\Oe)$. 

Vice versa, assume that $\us$ satisfies the complex scaling radiation condition. 
Since $\us$ solves \eqref{eq:Helmholtz} and \eqref{eq:HelmholtzBC}, the 
series representation \eqref{eq:expansion_u} holds true. Since $x\mapsto \us(
x,\cdot)$ has a $L^2(\bp)$-valued holomorphic extension, the mappings 
$x\mapsto \lsp\us(x,\cdot),\varphi_n\rsp$ are also holomorphic. Therefore, they 
are given by 
$x\mapsto c_n\exp(i\kappa_n x)+d_n\exp(-i\kappa_nx)$, not only for $x\in (0,
\infty)$, but also for 
$x\in {\mathcal S}_{\sigma}$, i.e.\ the holomorphic extension of $\us$ is 
given by the right hand side of 
\eqref{eq:expansion_u} with $x\in  {\mathcal S}_{\sigma}$. As 
\begin{equation*}
u_{\sigma}(x,y)=\sum_{n=1}^\infty \paren{c_ne^{i\sigma \kappa_nx}+d_n e^{-i 
\sigma \kappa_nx}}\varphi_n(y)\qquad \mbox{in } \Oe,
\end{equation*}
the assumption $u_{\sigma} \in H^2(\Oe)$ implies $d_n=0$ for all $n\in \setN$
, i.e.\ $\us$ satisfies the modal radiation condition. \\
Given a solution $u_{\sigma}\in H^2(\Oe\cup\Gamma)$ to \eqref{eq:PMLproblem} 
we can conclude that it is of the form \eqref{eq:PMLseries}, and hence 
corresponds to a holomorphic extension of a solution to \eqref{eq:WGproblem}. 
\end{proof}
Note that the holomorphic extension in Def.~\ref{Def:RCcomplexScaling} does 
not appear explicitly in numerical computations 
since such computations are based on \eqref{eq:PMLproblem}.

\subsection{complex scaling radiation condition for complex frequencies}
For complex frequencies $\kappa$ the choice of the branch cut of the square root function is not 
canonical, and different choices may lead to different modal radiation conditions. 
Similarly, different choices of $\sigma$ may lead to different complex scaling radiation conditions: 
A  solution $u_\sigma \in H^2(\Oe)$ to \eqref{eq:PMLproblem} with complex $\kappa$ is given by \eqref{eq:PMLseries},
if $\kappa_n=\sqrt{\kappa^2-\lambda_n}$ is defined such that $\Re(i\sigma \kappa_n)<0$. Hence, we are led to the following definition.
\begin{defi}\label{def:BranchCutSquareRoot}
For $\sigma=|\sigma|\exp(i\arg(\sigma)) \in \setC$ with $\arg:\setC \to [-\pi,\pi)$ and $\arg(\sigma)\in (0,\frac{\pi}{2})$ we define
\begin{equation}\label{eq:defsqrt}
 \sqrt{z}^\sigma:=\sqrt{|z|}e^{i \frac{\varphi}{2}} \qquad \text{for } z = |z| e^{i \varphi} \text{ with } \varphi \in [-2 \arg(\sigma),2 \pi - 2 \arg(\sigma) ).
\end{equation}
If $\lambda_1\leq \lambda_2\leq \dots$ denote the eigenvalues of $-\Delta_\Gamma$, we define $\kappa_n^\sigma : \setC \to \setC$ for $n \in \setN$ by
$\kappa_n^\sigma=\kappa_n^\sigma(\kappa):=\sqrt{\kappa^2 -\lambda_n}^\sigma$.
\end{defi}
By definition we have $\Re(i\sigma \kappa_n^\sigma)\leq0$ for all $n\in \setN$. We can define similar to Def~\ref{defi:modal_rad_cond} a complex modal radiation condition:
A function $u$ of the form
\begin{equation}
\label{eq:seriesPML}
  u(x,y)=\sum_{n=1}^\infty \paren{c_n e^{i \kappa_n^\sigma x} + d_n e^{-i \kappa_n^\sigma x} }\varphi_n (y),\qquad (x,y)\in \Oe,
\end{equation}
satisfies the modal radiation condition if all coefficients $d_n$ vanish. As in Lemma \ref{lemm:equivalence_complex_scal_rc} equivalence of this modal radiation condition
to the complex scaling radiation condition Def.~\ref{Def:RCcomplexScaling} can be shown if $\Re(i\sigma \kappa_n^\sigma)\neq 0$ for all $n\in \setN$.

Nevertheless, if we would use this definition without any restrictions, we would
get different solutions $\us$ (see \eqref{eq:seriesrepr}) to \eqref{eq:WGproblem} for different PML parameters $\sigma$. 
In other words, different $\sigma$ yield different
radiation conditions. To avoid this undesirable effect, we
define admissible regions for $\kappa$, such that $\kappa_n^{\sigma_1}=\kappa_n^{\sigma_2}$ for $\sigma_1,\sigma_2\in \setC$. 
\begin{defi}\label{def:admissSet}
Let $\lambda_1\leq \lambda_2\leq \dots$ denote the eigenvalues of $-\Delta_\Gamma$ and $\kappa_n^\sigma$ as
defined in Def.\ref{def:BranchCutSquareRoot}. The \emph{admissible set} $\Lambda^\sigma_{\Delta_{\Gamma}} \subset \setC$ 
is the set of all $\kappa\in\setC$ with $\Re\kappa>0$ and $\Im\kappa\leq 0$ such that
 \begin{enumerate}
  \item $\kappa_n^\sigma$ is holomorphic at $\kappa$ and
  \item $\kappa_n^\sigma$ is continuous along the path $\{ \Re(\kappa) - t i \in \setC ~| ~  t \in (0,-\Im(\kappa) ) \} $.
 \end{enumerate}
for all $n \in \setN$.
\end{defi}
For $\kappa \in \Lambda^{\sigma_1}_{\Delta_{\Gamma}} \cap 
\Lambda^{\sigma_2}_{\Delta_{\Gamma}}$ 
the modal radiation conditions with parameters $\sigma_1$ and $\sigma_2$ coincide, since for all $n\in \setN$
$\sqrt{\Re(\kappa)^2-\lambda_n}^{\sigma_1}=\sqrt{\Re(\kappa)^2-\lambda_n}^{\sigma_2}$  and none of the paths
$\left\{ \left(\Re(\kappa) - t i\right)^2-\lambda_n \in \setC ~| ~  t \in (0,-\Im(\kappa) ) \right\} $
has passed the branch cuts of $\sqrt{\cdot}^{\sigma_1}$ and  $\sqrt{\cdot}^{\sigma_2}$. % 

In Fig.~\ref{Fig:domainkappa} two admissible sets are given for a two-dimensional waveguide $\setR_+ \times [0,1]$. 
For $\sigma=1+i$ the branch cut of the square root is the negative imaginary axis, and therefore only in absolute values small imaginary parts of
$\kappa$ are allowed if $\Re(\kappa)$ is a little bit larger than on $\sqrt{\lambda_n}$ (see Fig.~\ref{Fig:domainkappa:1}). 
For $\sigma=1+1.5i$ the branch cut of the square root is in the third quadrant and therefore $\kappa$ with $\Re(\kappa)$ 
a little bit smaller than one $\sqrt{\lambda_n}$ are more problematic (see Fig.~\ref{Fig:domainkappa:2}).

\begin{figure}
  \capstart
 \begin{center}
  \subfigure[$\sigma = 1 +i$ \label{Fig:domainkappa:1}]{\includegraphics[width=0.4\textwidth]{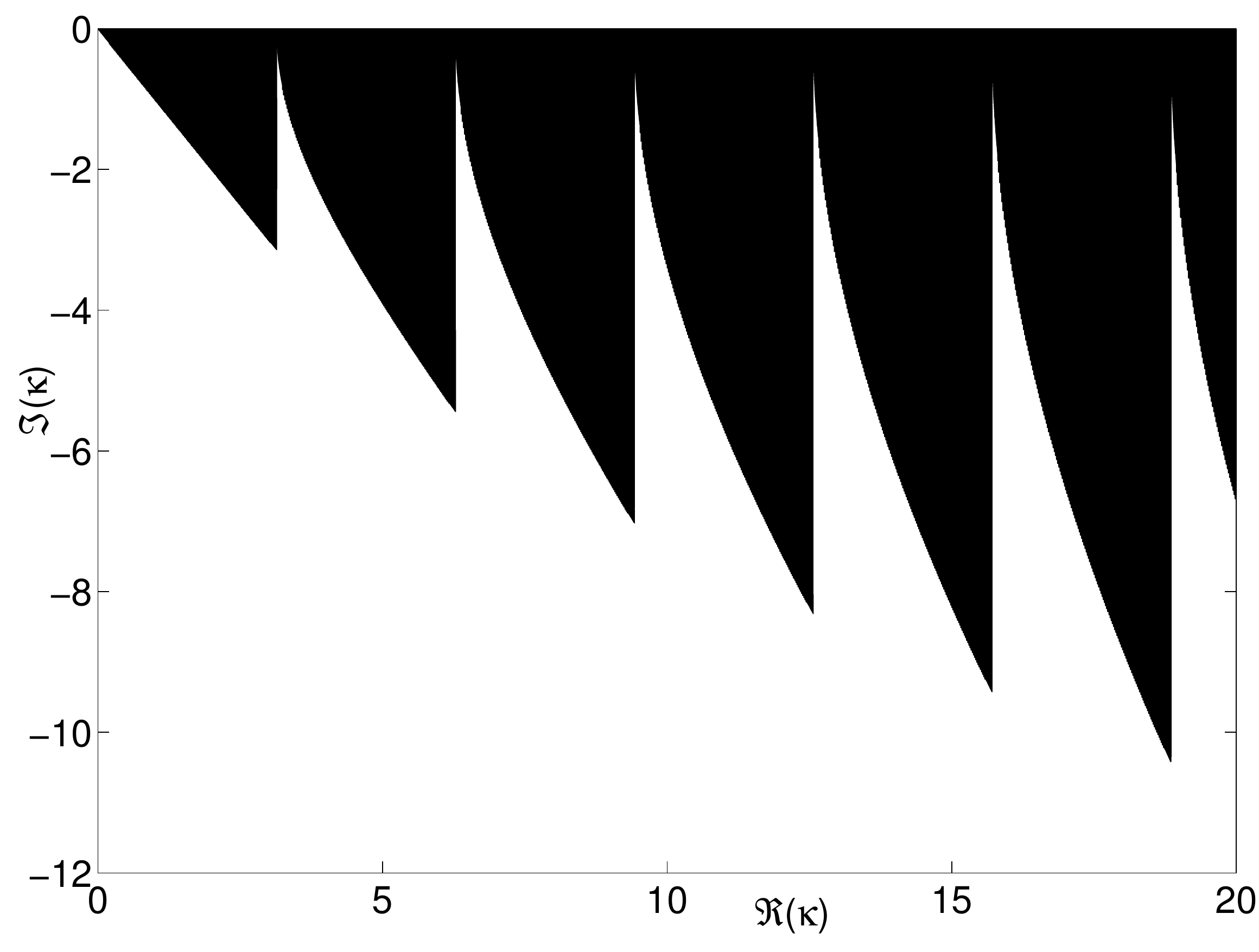}} \hspace{0.05\textwidth}
  \subfigure[$\sigma = 1 +1.5 i$ \label{Fig:domainkappa:2}]{\includegraphics[width=0.4\textwidth]{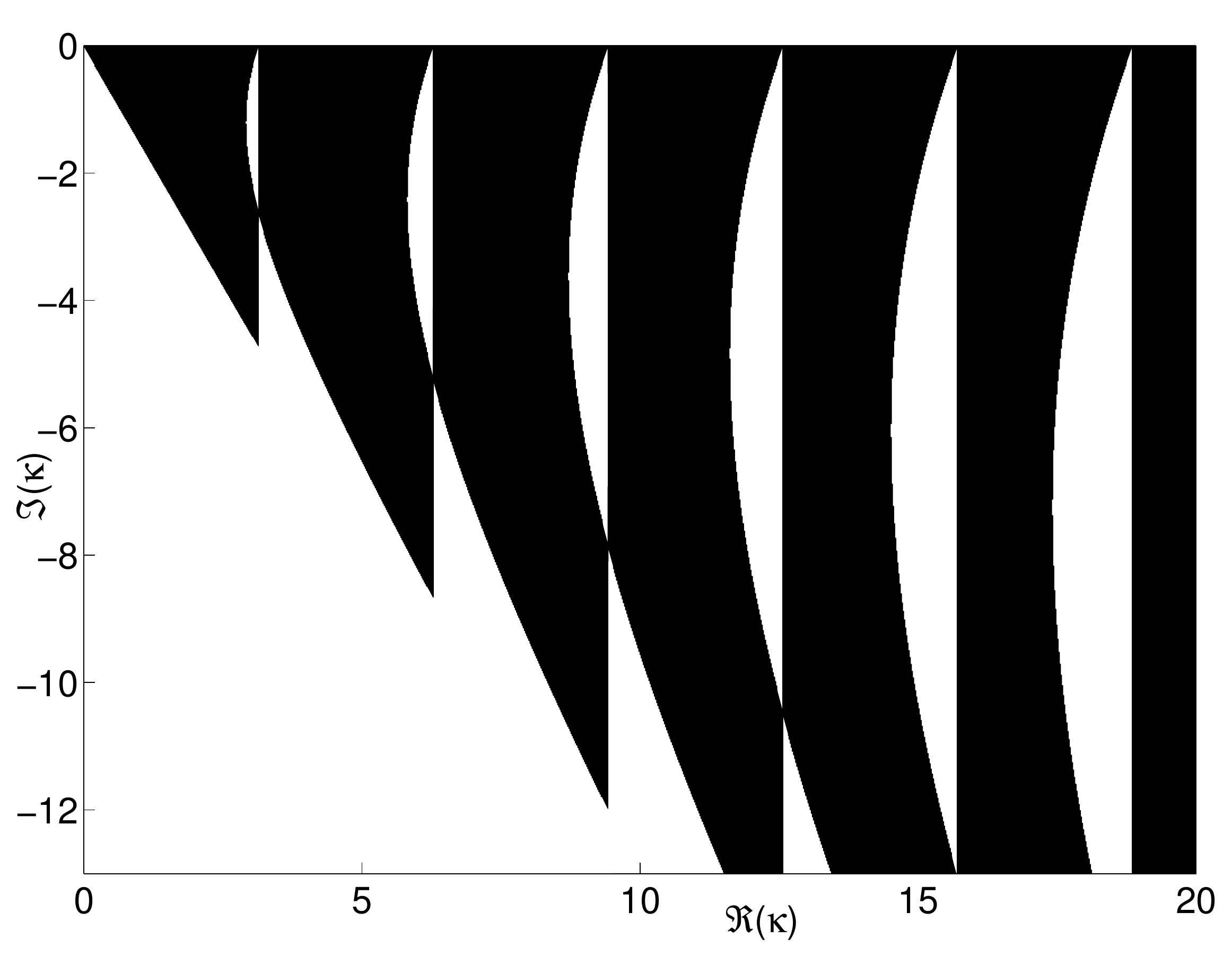}}
  \caption{admissible sets $\Lambda^\sigma_{\Delta_\Gamma}$ for two different $\sigma$ and $\lambda_n=(n-1)^2 \pi^2 $, $n \in \setN$}   
  \label{Fig:domainkappa}
 \end{center}
\end{figure}

Note, that $\Lambda^\sigma_{\Delta_{\Gamma}}$ is the union of the disjoint 
sets 
\begin{equation}\label{eq:components_Lambda}
\begin{aligned}
\left(\Lambda^\sigma_{\Delta_{\Gamma}}\right)^n:=\Big\{\kappa \in \setC~|~&\Re(\kappa)>0,\Im(\kappa)\leq 0,\sqrt{\lambda_n}<\Re(\kappa)<\sqrt{\lambda_{n+1}},\\
&\arg(\kappa^2-\lambda_{n+1}) <-2\arg(\sigma) < \arg(\kappa^2-\lambda_{n})\Big\}.
\end{aligned}
\end{equation}

\subsection{convergence of the PML method}
\label{sec:PML_assumpt}
In the case of several waveguides $W_l=\eta_l((0,\infty)\times \bp_l)$ for $l=1,\dots,L$ (see Sec.~\ref{sec:setting}), we
use the complex scaling vector $\sigma=(\sigma_1,\dots,\sigma_L) \in \setC^L$ with $\Re(\sigma_l),\Im(\sigma_l)>0$ and define
for a solution $u$ to \eqref{eqs:scat_prbl} $\uint:=u|_{\Oi}$,
\begin{equation*}
 u_{l}^{(\sigma_l)}(x,y):= u|_{W_l}\circ \eta_l( \sigma_l x,y),\qquad (x,y)\in (0,\infty)\times \bp_l,\quad l=1,\dots,L,
\end{equation*}
and 
$\uext_{\sigma}:=\left(u_{1}^{(\sigma_1)},\dots, u_{L}^{(\sigma_L)}  \right)^\top.$ 
The admissible set will be
\begin{equation}\label{eq:defi_Lambda}
\Lambda:= \bigcap_{l=1}^L\Lambda^{\sigma^l}_{\Delta_l}.%, \qquad 
\end{equation}
Let us formally state our definition of resonances:
\begin{defi}
$\kappa\in\Lambda$ (for some scaling parameters $\sigma_l$) 
is called a \emph{resonance} if there 
exists a resonance function $u\in H^1_{\rm loc}(\Omega)\setminus\{0\}$ 
satisfying $-\Delta u=\kappa^2 u$ in $\Omega$, $\Bdv u= 0$ on $\partial \Omega$ 
and the complex scaling radiation condition with parameter $\sigma_l$ in 
each waveguide $W_l$.
\end{defi}
We will check point by point the assumptions of Sec.~\ref{sec:ConvTheo} for a complex scaled version of \eqref{eqs:scat_prbl}. 
For notational simplicity we again discuss only the case of Dirichlet boundary conditions, i.e.\ $\Bdv u := u|_{\partial\Omega}$.

\medskip
\emph{Assumption~\hyperlink{AssA}{A}: Exterior and interior spaces.} $\Vint$,$\tr_-$ and $\tracespace$ are defined as in Section~\ref{sec:ConvTheo} after Ass.~\hyperlink{AssA}{A}
with $\bpM=\bigcup_{l=1}^L {\bpM}_l$.
We define $\Vext=\bigoplus_{l=1}^L \Vext_l$ with 
$$\Vext_l:=\{\uext_l \in H^1((0,\infty)\times \bp_l):~\uext_l|_{(0,\infty)\times \partial \bp_l}=0\}.$$
The spaces $\Xspace^1_l$, $\Xspace^2_l$, $\Yspace^1_l$ and $\Yspace^2_l$ are defined as in Sec.~\ref{sec:ConvTheo}. 
The trace operator $\tr_+:\Vext \to \tracespace$ is defined for $\uext=(\uext_1,\dots,\uext_L)^\top \in \Vext$ 
point wise: For $y\in \bpM$ we choose $l\in \{1,\dots,L\}$ such that  $y=\eta_l(0,\tilde y) \in {\bpM}_l$ with $\tilde y\in  \tilde {\bpM}_l$
and define $\left(\tr_+ \uext\right)(y) := \uext_l(0,\tilde y)$.

Finally, we define the bounded sesquilinear forms
\begin{equation*}
\begin{aligned}
  \aint(\uint,\vint) &:=\int_{\Oi} \nabla \uint\cdot \nabla \overline{\vint} \,dx, \qquad
  \bint(\uint,\vint) :=\int_{\Oi} \uint\overline{\vint} \,dx\\
  \aext_l(\uext_l,\vext_l) &:=\int_{0}^\infty \int_{\bp_l} \paren{\frac{1}{\sigma_l}\, \partial_x \uext_l \, 
     \partial_x \overline{\vext_l} 
  + \sigma_l \,\nabla_y \uext_l\cdot \nabla_y \overline{\vext_l}} \,dy\, dx,\\
  \bext_l(\uext_l,\vext_l) &:=\int_{0}^\infty \int_{\bp_l} \sigma_l \, \uext_l \overline{\vext_l} \,dy\, dx,\\
\end{aligned}
\end{equation*} 
and set $\sint:= \aint-\kappa^2\bint$, $\sext_l:=\aext_l-\kappa^2 \bext_l$ 
and $\sext(\uext,\vext):= \sum_{l=1}^L\sext_l(\uext_l,\vext_l)$
for $\uext=(\uext_1,\dots,\uext_L)^\top$ and $\vext=(\vext_1,\dots,\vext_L)^\top$. 

Using these definitions we arrive at the PML variational formulation:
If $\kappa \in \Lambda$ then $u$ is a solution to \eqref{eqs:scat_prbl} with the complex scaling radiation condition with parameter $\sigma_l$ in 
each waveguide $W_l$ if and only if $(\uint,\uext_{\sigma})^\top\in \Vspace$ solves
\begin{equation}
\label{eq:PMLformulation}
 \stot \paren{\svec{\uint\\ \uext_{\sigma}},\svec{\vint\\ \vext}} =  F\left(\svec{\vint\\ \vext}\right),\qquad \svec{\vint\\ \vext}\in \Vspace,
\end{equation}
with
\begin{equation*}
  F\left(\svec{\vint\\ \vext}\right):=\int_{\Oi}f\overline{\vint}\,dx 
+ \sum_{l=1}^L \int_{\bpM_l}\diffq{\ui}{\nu}\overline{\vint}\,ds + \sum_{l=1}^L \sext_l\left( \left(E_{+,l} \ui|_{\bpM_l}\right) \circ \eta_l,\vext_l\right).
\end{equation*}
$E_{+,l}:\tracespace_l \to \Vext_l$ can be any bounded extension operator with bounded
support $\{x \in W_l~|~(E_{+,l} f)(x)\neq 0, f \in \tracespace_l\}$ in $W_l$.

Moreover, with the help of the generalization of 
Lemma \ref{lemm:equivalence_complex_scal_rc} to complex $\kappa$ we can show  
that $\kappa\in\Lambda$ is a resonance if and only if there exists 
$u\in\Vspace\setminus\{0\}$ such that 
\[
\stot_{\kappa}(u,v) = 0\qquad \mbox{for all }v\in\Vspace. 
\]

\medskip
\emph{Assumption~\hyperlink{AssB}{B}: separation of $\Vext$.} In order to simplify the presentation, we only consider the case of one waveguide 
$W_1=(0,\infty) \times \bp$ in the following 
and omit the lower index $1$ for $l=1$. As in Sec.~\ref{sec:ConvTheo} we use the orthogonal set of eigenfunctions $\{\varphi_n:n\in\setN\}\subset \Yspace^2$ to $-\Delta$, i.e.
$-\Delta\varphi_n = \lambda_n\varphi_n$ with $\lambda_n\geq 0$. The orthogonality assumptions are trivial. The norms of $\Xspace_n$ and the separated sesquilinear forms are given by
\begin{equation}\label{eq:SepSesquiForm}
\|u\|_{\Xspace_n}^2 = \|u'\|_{L^2}^2 + (\lambda_n+2)\|u\|_{L^2}^2, \quad \stot_n(u,v) = \frac{1}{\sigma} \lsp u',v' \rsp_{L^2} + \sigma (\lambda_n-\kappa^2)\lsp u,v\rsp_{L^2}.
\end{equation}

\medskip
\emph{Assumption~\hyperlink{AssC}{C}: boundedness and coercivity.} 
$\stot_n$ is bounded by
\begin{equation}\label{eq:Cs_n}
  |\stot_n(u,v) | \leq \max\paren{\frac{1}{|\sigma|}+ |\sigma \kappa^2|,|\sigma|} \| u\|_{\Xspace_n} \|v \|_{\Xspace_n}
\end{equation}
with a constant independent of $n$. For the coercivity we consider each of the disjoints sets of $\Lambda^\sigma_{\Delta_{\Gamma}}$ defined in \eqref{eq:components_Lambda} separately: For $n_0\in \setN$ and $\kappa\in \left(\Lambda^\sigma_{\Delta_{\Gamma}}\right)^{n_0}$ it holds
\begin{equation}\label{eq:SetsCoerc}
-\pi\leq \arg(\kappa^2-\lambda_{n_0+1}) <-2\arg(\sigma) < \arg(\kappa^2-\lambda_{n_0}) \leq 0.
\end{equation}
Note, that $n \mapsto \arg(\kappa^2-\lambda_{n})\in [-\pi,0]$ is monotonically decreasing  since $\lambda_n \to \infty$ for $n\to \infty$. 
We distinguish two cases which for $\kappa>0$ correspond exactly to the cases of propagating modes 
($\kappa^2>\lambda_n$) and evanescent modes ($\kappa^2<\lambda_n$): $n\leq n_0$ and $n>n_0$.
\begin{enumerate}
 \item For $n=1,\dots,n_0$ the right half of \eqref{eq:SetsCoerc} leads to
 \begin{equation}\label{eq:SetsCoerc2}
  -\arg(\sigma)< \arg(\kappa^2-\lambda_{n})+\arg(\sigma)=  \arg(\sigma(\kappa^2-\lambda_{n}))\leq \arg(\sigma),
 \end{equation}
since $\arg(\sigma)\in (0,\frac{\pi}{2})$ and $\arg(\kappa^2-\lambda_{n})\in (-\pi,0]$. We define the rotations of \eqref{eq:Ass_sn_coerc1} by
$$\theta_n(\kappa):=\exp\left(i\left(\frac{\pi+\arg(\sigma)-\arg(\sigma(\kappa^2-\lambda_n))}{2} \right) \right)$$
and compute
\begin{equation*}
 \begin{aligned}
  \frac{\theta_n}{\sigma}&=\frac{1}{|\sigma|} \exp\left(i\left(\frac{\pi-\arg(\sigma)-\arg(\sigma(\kappa^2-\lambda_n))}{2} \right) \right),\\
  -\theta_n \sigma (\kappa^2-\lambda_n) &= -\exp\left(i\left(\frac{\pi+\arg(\sigma)+\arg(\sigma(\kappa^2-\lambda_n))}{2} \right) \right).
 \end{aligned}
\end{equation*}
Using \eqref{eq:SetsCoerc2} we get 
$$\alpha_n(\kappa):=\min\left\{\Re\left( \frac{\theta_n(\kappa)}{\sigma}\right),\frac{\Re\left( \theta_n(\kappa) \sigma (\lambda_n-\kappa^2) \right)}{2+\lambda_n} \right\}>0 $$
and  \eqref{eq:Ass_sn_coerc1} is shown for the separated sesquilinear forms $\stot_n$ defined in \eqref{eq:SepSesquiForm}.
\item For $n=n_0+1,\dots$ we take the left half of \eqref{eq:SetsCoerc}, use $0>\arg(\sigma(\kappa^2-\lambda_{n}))= \arg(\sigma(\lambda_{n}-\kappa^2))-\pi$
and deduce
 \begin{equation}\label{eq:SetsCoerc3}
\arg(\sigma)\leq\arg(\sigma(\lambda_{n}-\kappa^2))<\pi-\arg(\sigma).  
 \end{equation}
Since $\left(\Lambda^\sigma_{\Delta_{\Gamma}}\right)^{n_0}$ is bounded (see Fig.~\ref{Fig:domainkappa}) and $\lambda_j\to \infty$ for $j \to \infty$ 
there exists a constant  $\nrGM^{n_0}\in \setN$ defined by
$$\nrGM^{n_0}:=\min\left\{j \in \setN ~|~ \lambda_{j+1} > 2 \frac{\Re(\sigma \kappa^2)}{\Re(\sigma)} \text{ for all }\kappa\in \left(\Lambda^\sigma_{\Delta_{\Gamma}}\right)^{n_0}\right\}.$$
For $n>\nrGM^{n_0}$ there holds $\Re(\sigma (\lambda_n - \kappa^2))> \frac{\Re(\sigma)}{2}\lambda_n$
and \eqref{eq:Ass_sn_coerc2} holds true with
$$\alpha_n:=\min\left\{ \frac{1}{\Re(\sigma)},\frac{\Re(\sigma)\lambda_{n}}{4+2\lambda_n}\right\}>0.$$

For $n=n_0+1,\dots,\nrGM^{n_0}$ we define similar to the first case 
$$\theta_n(\kappa):=\exp\left(i\left(\frac{\arg(\sigma)-\arg(\sigma(\lambda_n-\kappa^2))}{2} \right) \right)$$
and use \eqref{eq:SetsCoerc3} for
$$\alpha_n(\kappa):=\min\left\{\Re\left( \frac{\theta_n(\kappa)}{\sigma}\right),\frac{\Re\left( \theta_n(\kappa) \sigma (\lambda_n-\kappa^2) \right)}{2+\lambda_n} \right\}>0.$$
\end{enumerate}
Since $\alpha_n \to \min\{\frac{1}{\Re(\sigma)}, \frac{\Re(\sigma)}{2}\} $ for $n\to \infty$, the constant $\alpha(\kappa):=\inf\{\alpha_n(\kappa)~|~n\in \setN\}$ 
in  \eqref{eq:Ass_sn_coerc1} and \eqref{eq:Ass_sn_coerc2} is strictly positive.

\medskip
\emph{Assumption \hyperlink{AssD}{D}: discrete subspaces.} The discrete subspaces are chosen exactly the same way as in Section \ref{sec:ConvTheo}.

\medskip
\emph{Assumption \hyperlink{AssE}{E}: eigenvalue setting.} 
Most properties stated in this assumption are obvious, but we have to argue 
that $\Cs ,\theta_n, \alpha$ and $\nrGM$ can be chosen independent of 
$\kappa$ in a neighborhood $\widehat\Lambda$ of each $\kappa_0\in\Lambda$. If $\widehat\Lambda \subset \left(\Lambda^\sigma_{\Delta_{\Gamma}}\right)^{n}$ for one $n \in \setN$,
then $\nrGM$ is independent of $\kappa \in \widehat\Lambda$. Due to \eqref{eq:SetsCoerc2} and \eqref{eq:SetsCoerc3} $\theta_n$ depends continuously on $\kappa$. The same
holds true for $\Cs$ and $\alpha$. Therefore, they can be chosen independent of $\kappa \in \widehat\Lambda$ if $\widehat\Lambda$ is compact.

\medskip
Since all assumptions are satisfied, Theorem~\ref{theo:ConvTheo} is applicable and yields the following: 
\begin{theorem}[PML for diffraction problems]
If $\kappa \in \Lambda$ with $\Lambda$ defined in \eqref{eq:defi_Lambda} is 
not a resonance, then equation \eqref{eq:PMLformulation} is uniquely 
 solvable with solution $(\uint, \uext)^\top \in \Vspace$ for all right hand sides $F \in \Vspace^*$, and there exists a constant $h_0>0$ such that the discrete variational problems
 \begin{equation}
\label{eq:PMLformulation_disc}
 \atot \paren{\svec{\uint_h\\ \uext_{h,N}},\svec{\vint_h\\ \vext_{h,N}}} - \kappa^2 \btot \paren{\svec{\uint_h\\ \uext_{h,N}},\svec{\vint_h\\ \vext_{h,N}}} = F\paren{\svec{\vint_h\\ \vext_{h,N}}}, 
 \quad \svec{\vint_h\\ \vext_{h,N}}\in \Vspace_{h,N}
\end{equation}
have a unique solution $(\uint_h, \uext_{h,N})^\top \in \Vspace_{h,N}$ for all $h\leq h_0$ and all $N\in \setN$. Moreover, there exists a constant $C>0$ independent of $h$ and $N$ such that
\begin{equation*}
\norm{\svec{\uint\\ \uext}-\svec{\uint_h\\ \uext_{h,N}}}_{\Vspace} 
\leq C\inf_{(w^{\rm int}_h, w^{\rm ext}_{h,N})^\top \in\Vspace_{h,N}}\norm{\svec{\uint\\ \uext}-\svec{w^{\rm int}_h\\ w^{\rm ext}_{h,N}}}_{\Vspace}.
\end{equation*}
\end{theorem}
Part of the approximation error is the error due to truncation of the infinite PML. 
In each waveguide $W_l$, $l=1,\dots,L$, we approximate (cf. \eqref{eq:PMLseries}) 
\begin{equation*}
 \tilde u_l (x,y):=\sum_{n=1}^\infty c_n e^{i \sigma_l \sqrt{\kappa^2-\lambda_n^l}^{\sigma_l} x} \varphi_n^l(y),\quad (x,y)\in (0,\infty)\times \bp_l,
\end{equation*}
by $0$ for all $x\geq \rho_N, y \in \bp_l$ ($\rho_N$ being the length of the PML defined in Sec.~\ref{sec:ConvTheo} after Ass.~\hyperlink{AssD}{D}). 
Hence, suppressing the indices $l$ the truncation error can be estimated by 
\begin{equation}
 \|\tilde u \|_{H^1((\rho,\infty)\times \bp)}^2 \leq \sum_{n=1}^\infty |c_n|^2 \left(|\kappa_n^\sigma \sigma|^2 + 
   \frac{1 + \lambda_n}{-2 \Re(i \kappa_n^\sigma \sigma)} \right) e^{2 \Re(i \kappa_n^\sigma \sigma) \rho_N}
\end{equation}
with $\kappa_n^\sigma $ defined in Def.~\ref{def:BranchCutSquareRoot}. 
Due to $\Re(i \kappa_n^\sigma \sigma)<0$ the truncation error is exponentially decreasing 
with increasing $\rho_N$. Nevertheless, the error becomes large, 
if $\Re(i \sqrt{\kappa^2-\lambda_n}^\sigma \sigma) \approx 0$ for some $n$, which is the case for $\kappa^2 \approx\lambda_n$ as well as near the branch cuts of the square root. 

\medskip
Theorem \ref{Theo:ConvEWP} yields the following:
\begin{theorem}[PML for resonance problems]
For all compact $\Lambda'\subset\Lambda$ we have
\[
\lim_{h\to 0,N\to\infty}\dist(\SpLam\cap\Lambda',\SpLam_{h,N}\cap\Lambda') = 0.
\]
\end{theorem}

\begin{proof}
What remains to be shown is that 
there exists a $\kappa$ such that \eqref{eq:PMLformulation} is uniquely solvable. $\stot_\kappa$ depends holomorphically on $\kappa$ for
$\kappa \in \Lambda \cup \{z\in \setC: \arg(z)\in (0,\frac{\pi}{2})\}$. 
Since Ass.~\hyperlink{AssC}{C} can be shown for $\kappa \in\{z\in \setC: \arg(z)\in (0,\frac{\pi}{2})\}$ similar to $\kappa \in \Lambda$,
we can use Theorem \ref{Theo:ConvEWP} for $\Lambda \cup \{z\in \setC: \arg(z)\in (0,\frac{\pi}{2})\}$.
Since for $\kappa$ with $0<\Re(\kappa)<\Im(\kappa)$ the real parts of all the coefficients in $\stot_\kappa$ are positive, i.e.
$$\min \left\{1, \Re\left(-\kappa^2\right), \Re\left(\frac{1}{\sigma}\right),\min_{n\in \setN} 
\Re\left(\frac{\sigma(\lambda_n-\kappa^2)}{2+\lambda_n} \right)  \right\}>0,$$
\eqref{eq:PMLformulation} is for such $\kappa$ uniquely solvable by the Lax-Milgram Lemma and the proof is complete.
\end{proof}

\section{Hardy space method}
As in the previous section we first introduce another equivalent formulation of the radiation condition 
called the pole condition. Based on the pole condition 
we formulate the Hardy space variational problem and use Theorem~\ref{theo:ConvTheo} to show an exponential convergence with respect to the number of degrees of freedom in radial direction. We end this section with the description 
of a suitable choice of the approximating subspace which avoids deterioration of convergence 
for frequencies close to Wood anomalies. 

\subsection{pole condition}
For the discussion of the pole condition we again consider only one waveguide as in \S \ref{sec:scaling_RC}. 
Let $\us(x,y)=\sum_{n=1}^\infty c_n e^{i\kappa_n x}\varphi_n(y)$ be a solution to 
\eqref{eq:Helmholtz} and \eqref{eq:HelmholtzBC} with $\kappa>0$ satisfying the modal radiation condition.  
Then the Laplace transform $\hat u (s,y):=\LT (\us(\bullet,y))(s)$ of $\us$ 
in the infinite direction $x$ is given by
\begin{equation*}
 \hat u (s,y) = \sum_{n=1}^\infty \frac{c_n}{s - i\kappa_n} \varphi_n(y),\qquad \Re(s)>0,\quad y\in \bp.
\end{equation*}
It has a meromorphic extension to $\setC$ with poles at $\{i \kappa_n,~n\in \setN\}$. 
In contrast, the Laplace transform of $e^{-i\kappa_n x}$ has a pole at $-i\kappa_n$. 
Since for real $\kappa$ the numbers $i\kappa_n$ lie on the positive imaginary axis and 
the negative real axis, formally $\us$ satisfies the modal radiation condition if 
and only if $\hat u$ has no poles in a complex half plane $\{\kappa_0 s: s\in\setC,\;\Im s <0\}$
for some $\kappa_0\in\setC$ with $\Re\kappa_0>0$, $\Im\kappa_0>0$, which will be a parameter of the 
method.

We define the M\"obius mapping $m_{\kappa_0}:\setC\setminus\{1\} \to \setC$, $m_{\kappa_0}(z):=i \kappa_0 \frac{z+1}{z-1}$ 
and a corresponding M\"obius transform $\MT:L^2(\kappa_0 \setR) \to L^2(S^1)$ from $\kappa_0 \setR:=\{ \kappa_0 s ~|~s\in \setR\}$ to the complex unit sphere 
$S^1:=\{z \in \setC ~|~|z|=1\}$ via
\begin{equation*}
 (\MT f)(z) := \frac{(f \circ m_{\kappa_0})(z)}{z-1}, \quad z\in S^1\setminus\{1\},\qquad f \in L^2(\kappa_0 \setR).
\end{equation*}
Due to the scaling $(z-1)^{-1}$ the M\"obius transform $\MT$ is unitary up to a constant. 
Applying $\MT$ to the Laplace transformed function $\hat u$
we get
\begin{equation}\label{eq:Un}
\left(\MT \hat u \right)(z)=\sum_{n=1}^\infty \frac{c_n \varphi_n(y)}{i(\kappa_0-\kappa_n)z+i(\kappa_n+\kappa_0)},\qquad z \in S^1,\quad y\in \bp.
\end{equation}

The Hardy space $H^+(S^1)$ is defined as the set of all functions $f\in L^2(S^1)$ for which 
there exists a holomorphic function $v : \{z \in \setC ~|~|z|<1\} \to \setC$ such that 
$\lim_{r \nearrow 1} \int_0^{2\pi}|v(re^{it})-f(e^{it})|^2\,dt =0$.  Equipped with the 
$L^2$-inner product, $H^+(S^1)$ is a Hilbert space (see e.g.\ \cite{Duren:70}).

\begin{defi}[pole condition]\label{def:PC}
Let $\kappa_0\in\setC$ with $\Re\kappa_0>0$ and $\Im\kappa_0>0$. A function 
$u\in H^2_{\rm loc}(\Oe\cup\bpM)$ satisfies the \emph{pole condition with parameter $\kappa_0$} if 
\[
\int_0^{\infty} e^{-s_0x}\|u(x,\cdot)\|_{L^2(\bp)}\,dx<\infty
\]
for some $s_0>0$ and the Laplace transform $(\LT u)(s):=\int_0^{\infty}e^{-sx}u(x,\cdot)\,dx$ 
(with values in $L^2(\bp)$) has a holomorphic extension from $\{s\in\setC:\Re s> s_0\}$ 
to the half-plane $\{\kappa_0 s: s\in\setC,\;\Im s <0\}$ with $L^2$-boundary values on 
$\kappa_0\setR$ such that 
\[
\MT\LT u \in H^+(S^1)\otimes L^2(\bp).
\]
\end{defi}
\begin{lemm}
\label{Lem:EquivPC}
 Let $\kappa \in \setC$ with $\Re(\kappa)>0$ and $\Im(\kappa)\leq 0$ %satisfying \eqref{eq:noWoodAnomaly} 
 and let $\us\in H^1_{\rm loc}(\Oe\cup\bpM)$ be a solution to \eqref{eq:Helmholtz} and \eqref{eq:HelmholtzBC}
with expansion \eqref{eq:seriesPML}  using the definition of $\kappa_n^\sigma$ of Definition \ref{def:BranchCutSquareRoot} with 
$\sigma:=i/\kappa_0$. Moreover, let $\kappa$ belong to the admissible set $\Lambda^\sigma_{\Delta_{\Gamma}}$ defined in Def.~\ref{def:admissSet}.
Then the following statements are equivalent: 
\begin{enumerate}
 \item (modal radiation condition) All coefficients $d_n$ in \eqref{eq:seriesPML} vanish. 
 \item $\us$ satisfies the pole condition with parameter $\kappa_0$. 
\end{enumerate}
\end{lemm}

\begin{proof}
By definition of $\kappa_n^\sigma$ and $\kappa \in \Lambda^\sigma_{\Delta_{\Gamma}}$ there holds
\begin{equation}\label{eq:prop_kn}
\Re\left(\kappa_n^{\sigma} / \kappa_0 \right)>0\qquad \mbox{and}\qquad 
\left|\frac{\kappa_n^{\sigma}+\kappa_0}{\kappa_n^{\sigma}-\kappa_0}\right|> 1\qquad 
\mbox{for all }n\in\setN. 
\end{equation}
First assume that $\us$ satisfies the modal radiation condition. Then $\MT\LT\us$ is well 
defined and satisfies \eqref{eq:Un} with $\kappa_n=\kappa_n^{\sigma}$. 
Therefore, each term in the series \eqref{eq:Un} belongs to $H^+(S^1)\otimes L^2(\bp)$. 
Moreover, the series converges in $L^2(S^1)\otimes L^2(\bp)$ since 
$\|\frac{1}{i(\kappa_0-\kappa_n^{\sigma})z+i(\kappa_n^{\sigma}+\kappa_0)}\|_{L^2(S^1)}
=\mathcal{O}\paren{\frac{1}{|\kappa_n^{\sigma}|}}=\mathcal{O}\paren{\lambda_n^{-1/2}}$ 
(see \cite[proof of Lemma A.3]{HohageNannen:09}) and 
$\sum_{n=1}^\infty (1+\lambda_n)^{1/2}|c_n|^2 = \|u|_{\bpM}\|_{H^{1/2}(\bpM)}^2<\infty$.

Vice versa assume that $\us$ satisfies the pole condition. Then 
\[
\tfrac{c_n }{i(\kappa_0-\kappa_n^{\sigma})z+i(\kappa_n^{\sigma}+\kappa_0)}
+\tfrac{d_n }{i(\kappa_0+\kappa_n^{\sigma})z-i(\kappa_n^{\sigma}-\kappa_0)}
=\MT\LT\lsp \us,\varphi_n\rsp_{L^2(\bp)}\in H^+(S^1)
\]  
for all $n\in\setN$. 
Since $z\mapsto (i(\kappa_0+\kappa_n^{\sigma})z-i(\kappa_n^{\sigma}-\kappa_0))^{-1}$ has a pole 
at $\frac{\kappa_n^{\sigma}-\kappa_0}{\kappa_n^{\sigma}+\kappa_0}\in \{z\in\setC:|z|<1\}$, it
follows that $d_n=0$. 
\end{proof}

Note that $s_0>0$ in Def.~\ref{def:PC} is needed for frequencies $\kappa$ with $\Im\kappa<0$ since by definition of
$\kappa_n^\sigma$ propagating modes become exponentially 
increasing in this case. However, the pole condition is independent of the choice of $s_0$. 

\subsection{Hardy space variational formulation for one waveguide}
For the details of the Hardy space method in one dimension we refer to \cite[sec.~2]{HohageNannen:09}. The role of the damping 
parameter $\sigma$ is replaced in the HSM by the parameter $\kappa_0 \in \setC$ of the M\"obius transform, 
which satisfies $\Re(\kappa_0)>0$ and $\Im(\kappa_0)>0$. 

For simplicity we introduce the linear, injective and bounded operators $\OpT_\pm:\setC \oplus H^+(S^1) \to H^+(S^1)$ by 
\begin{equation}
 \left(\OpT_\pm \svec{f_0 \\ F} \right)(z):=\frac{1}{2}\left(f_0 + (z\pm 1) F(z) \right),\quad z \in S^1,\qquad \svec{f_0 \\ F} \in \setC \oplus H^+(S^1)
\end{equation}
and recall the equations (2.9) and (2.14) from \cite[sec.~2]{HohageNannen:09}: For suitable $f:[0,\infty)\to \setC$ and $f_0:=f(0)$ there exists a $F \in H^+(S^1)$ such that
\begin{eqnarray}
\label{eq:tpmform}
  \MT \LT f = \frac{1}{i \kappa_0} \OpT_-\svec{f_0 \\ F}\quad\text{and} \quad  \MT \LT \partial_x f=  \OpT_+\svec{f_0 \\ F}.
\end{eqnarray}
Due to the boundedness of $\OpT_\pm$ and the parallelogram identity, there exist constants $C_1,C_2>0$ such that
\begin{equation}
\label{eq:EquiTpm}
  C_1 \left\| \svec{f_0\\ F} \right\|_{\setC \oplus L^2(S^1)}^2 \leq \left\|\OpT_+ \svec{f_0\\ F} \right\|_{L^2(S^1)}^2 + \left\|\OpT_- \svec{f_0\\ F} \right\|_{L^2(S^1)}^2
  \leq C_2 \left\| \svec{f_0\\ F} \right\|_{\setC \oplus L^2(S^1)}^2.
\end{equation}
Similar to \cite[Lemma A.3]{HohageNannen:09}, the space $\Vext:=\Xspace^2 \otimes \Yspace^1 \cap \Xspace^1 \otimes \Yspace^2$ with
\begin{subequations}
 \begin{eqnarray}
  \Xspace^2&:=&
	\setC \oplus H^+(S^1),\quad \lsp \svec{f_0\\F}, \svec{g_0\\G} \rsp_{\Xspace^2} :=f_0 \overline{g_0} + \lsp F,G \rsp_{L^2(S^1)},\\
  \Xspace^1&:=&\text{completion of } \setC \oplus H^+(S^1) \text{ w.r.t. } \notag\\
  &&\lsp \svec{f_0\\F}, \svec{g_0\\G} \rsp_{\Xspace^1} := \lsp \OpT_- \svec{f_0 \\ F},\OpT_- \svec{g_0 \\ G} \rsp_{L^2(S^1)},\\
  \Yspace^1&:=&L^2(\bp),\qquad \Yspace^2:=H^1(\bp)  
 \end{eqnarray}
and
\begin{equation}
 \lsp \svec{f_0\\F}, \svec{g_0\\G} \rsp_{\Vext} := \lsp \svec{f_0\\F}, \svec{g_0\\G} \rsp_{\Xspace^2\otimes \Yspace^1} +  \lsp \svec{f_0\\F}, \svec{g_0\\G} \rsp_{\Xspace^1\otimes \Yspace^2}
\end{equation}
\end{subequations}
is a Hilbert space and fulfills the requirements of the Hardy space method. Note, that
$$\Vext\subset \left(\setC \oplus H^+(S^1)\right) \otimes L^2(\bp)\sim L^2(\bp) \oplus \left(H^+(S^1) \otimes   L^2(\bp)\right).$$
We will denote elements of $\Vext$ in the second form, i.e. $\svec{v_0\\V}\in \Vext$ with $v_0\in L^2(\bp)$ and $V \in H^+(S^1) \otimes   L^2(\bp)$.
Recall from \cite[Lemma A.1]{HohageNannen:09} the identity 
\begin{equation}
\label{eq:BasicId}
 \int_0^\infty f(x)\, g(x)\,dx = \frac{ - i \kappa_0}{ \pi} \int_{S^1} (\MT \LT f)(z) (\MT \LT g)(\overline{z})\,|dz|,
\end{equation}
which is applicable for $u(\bullet,y)$, $v(\bullet,y)$ as well as $\partial_x u(\bullet,y)$ and $\partial_x v(\bullet,y)$ and all $y \in \bp$. 
Using the involution $\OpC :H^+(S^1) \to H^+(S^1)$ defined by 
$(\OpC F)(z) := \overline{F(\overline{z})}$ for $z \in S^1$ and $F\in H^+(S^1)$
as in \cite{HohageNannen:09} we get
\begin{equation*}
 %(\OpC \vML)(z) := \overline{\vML(\overline{z})},\qquad z \in S^1, \vML \in H^+(S^1)
 \int_{S^1} (\MT \LT f)(z) (\OpC\MT \LT g)(\overline{z})\,|dz| = \lsp \MT \LT f , \MT \LT g\rsp_{L^2(S^1)}.
\end{equation*}
Hence, the  exterior Hardy space sesquilinear forms for one waveguide are
\begin{equation}
\label{eq:HSMsesqui}
\begin{aligned}
  \aext\left(\svec{u_{0}\\\uML},\svec{v_{0}\\\vML}\right) :=&  \frac{ - i \kappa_0}{\pi} \lsp (\OpT_+ \otimes \Id_{\bp}) (u_0, \uML) , (\OpT_+ \otimes \Id_{\bp}) (v_0, \vML) \rsp_{L^2(S^1) \otimes L^2(\bp)}\\
  &+ \frac{ - i \kappa_0}{\pi} \frac{1}{(i \kappa_0)^2} \lsp (\OpT_- \otimes \nabla_{\bp}) (u_0, \uML) , (\OpT_- \otimes \nabla_{\bp}) (v_0, \vML) \rsp_{H^1(S^1) \otimes L^2_{\rm tan}(\bp)},\\
  \bext\left(\svec{u_{0}\\\uML},\svec{v_{0}\\\vML}\right) :=& \frac{ - i \kappa_0}{\pi} \frac{1}{(i \kappa_0)^2} \lsp (\OpT_- \otimes \Id_{\bp}) (u_0, \uML) , 
  (\OpT_- \otimes \Id_{\bp}) (v_0, \vML) \rsp_{H^1(S^1) \otimes L^2(\bp)}.
\end{aligned}
\end{equation} 
$L^2_{\rm tan}(\bp)$ denotes the space of square integrable tangential vector fields on $\bp$. 
For a single waveguide, the Hardy space variational formulation is to find the solution $(u_0,\uML) \in \Vext$ of 
\begin{equation}
 \aext\left(\svec{u_{0}\\\uML},\svec{v_{0}\\\vML}\right) - \kappa^2 \bext\left(\svec{u_{0}\\\uML},\svec{v_{0}\\\vML}\right) = F\left(\svec{v_{0}\\\vML}\right),\qquad \svec{v_{0}\\\vML} \in \Vext,
\end{equation}
for one $F \in {\Vext}^*$.

\subsection{convergence of the Hardy space method}
\label{seq:HSM_assumpt}
Similar to Sec.~\ref{sec:PML_assumpt} we check the assumptions point by point. For simplicity, we again use 
Dirichlet boundary condition and only one single waveguide.

\medskip
\emph{Assumption~\hyperlink{AssA}{A}: Exterior and interior spaces} and \emph{Assumption \hyperlink{AssB}{B}:  separation of $\Vext$.} Most of these assumptions hold true as in the PML case since we use the same interior space,
 the same spaces $\Yspace_1$ and $\Yspace_2$ and the same orthogonal system $\{\varphi_n: n\in \setN\}\subset \Yspace_2$ as in the PML case. The assumptions on $\Vext$ hold true by construction. 
 The boundedness and surjectivity of the trace operator $\tr_+:\Vext \to \tracespace:=H^{1/2}(\bp)$ 
 defined by
 $$ \tr_+ \svec{v_0\\V}:=v_0,\qquad \svec{v_0\\V}\in \Vext,
 $$
 can be proven similar to \cite[Lemma A.3]{HohageNannen:09}. The modal exterior sesquilinear forms defined in \eqref{eq:sn_Xn} are
 \begin{equation}
 \begin {aligned}
  \stot_n \left( \svec{u_{0}\\\uML},\svec{v_{0}\\\vML}  \right) :=&  \frac{ - i \kappa_0}{\pi} \lsp \OpT_+  \svec{u_{0}\\\uML} , \OpT_+  \svec{v_{0}\\\vML} \rsp_{L^2(S^1)} \\
  &+  (\lambda_n-\kappa^2) \frac{ i }{\kappa_0 \pi} \lsp \OpT_-  \svec{u_{0}\\\uML} , \OpT_-  \svec{v_{0}\\\vML} \rsp_{L^2(S^1)}
 \end {aligned}  
 \end{equation}
 and the modal spaces are $\Xspace_n:=\setC \oplus H^+(S^1)$ with 
 \begin{equation}
 \lsp\svec{u_{0}\\\uML},\svec{v_{0}\\\vML}\rsp_{\Xspace_n}:=u_0 \overline{v_0} + \lsp \uML,\vML\rsp_{L^2(S^1)}+(1+ \lambda_n)\lsp \OpT_- \svec{u_{0}\\\uML}, \OpT_-\svec{v_{0}\\\vML} \rsp_{L^2(S^1)}.
 \end{equation}

 \medskip
 \emph{Assumption~\hyperlink{AssC}{C}: boundedness and coercivity.} Continuity of $\stot_n$ independent of $n$ follows with the continuity of $\OpT_\pm$. For the coercivity, 
 we have due to \eqref{eq:EquiTpm} the same situation as in the PML case with $\sigma:=i /\kappa_0$. Hence, $\nrGM$ and $\theta_1,\dots,\theta_n$ are exactly
 the same as for the PML case and the coercivity constant $\alpha$ differs only by a positive constant (independent of $n$) from the constant in the PML case.
 
 \medskip
\emph{Assumption~\hyperlink{AssD}{D}: discrete subspaces.} We choose the interior finite element space $\Vint_h$ as 
in section \ref{sec:setting} and $\Yspace_h:=\tr_-\Vint_h$. Since the
trigonometric monomials are an orthogonal basis of $H^+(S^1)$, we choose 
$\Xspace_N:=\setC \oplus \Span \{z^0,\dots,z^{N-2}\} \subset \setC \oplus H^+(S^1)$.

\medskip
\emph{Assumption~\hyperlink{AssE}{E}: eigenproblem setting.}
The argument is again very similar to that for PML.

\medskip
Since all assumptions to Theorem~\ref{theo:ConvTheo} are fulfilled, we have the following convergence theorem.

\begin{theorem}[HSM for diffraction problems]
\label{Theo:ConvHSMWG}
Let $\kappa_0^l \in \setC$ with $\Re(\kappa_0^l),\Im(\kappa_0^l)>0$ for $l=1,\dots,L$ and $\kappa \in \setC$ with $\Re(\kappa)>0$, $\Im(\kappa)\leq 0$, 
 $\kappa^2 \not \in \bigcup_{l=1}^L \setSp(-\Delta_l)$ and 
$\Lambda:=\bigcap_{l=1}^L \Lambda^{i/\kappa_0^l}_{\Delta_l}$. 
If $\kappa \in \Lambda$ is not a resonance, then equation 
\eqref{eq:PMLformulation} with the exterior Hardy space sesquilinear forms of 
\eqref{eq:HSMsesqui}
 is uniquely solvable with solution $(\uint, \uext)^\top \in \Vspace$, and  
 there exists a constant $h_0>0$ such that the discrete variational problems
 \begin{equation}
\label{eq:HSMformulation_disc}
 \atot \paren{\svec{\uint_h\\ \uext_{h,N}},\svec{\vint_h\\ \vext_{h,N}}} - \kappa^2 \btot \paren{\svec{\uint_h\\ \uext_{h,N}},\svec{\vint_h\\ \vext_{h,N}}} = F\paren{\svec{\vint_h\\ \vext_{h,N}}}, 
 \quad \svec{\vint_h\\ \vext_{h,N}}\in \Vspace_{h,N}
\end{equation}
have a unique solution $(\uint_h, \uext_{h,N})^\top \in \Vspace_{h,N}$ for all $h\leq h_0$ and all $N\in \setN$. Moreover, there exists a constant $C>0$ independent of $h$ and $N$ such that
\begin{equation*}
\norm{\svec{\uint\\ \uext}-\svec{\uint_h\\ \uext_{h,N}}}_{\Vspace} 
\leq C\inf_{(w^{\rm int}_h, w^{\rm ext}_{h,N})^\top \in\Vspace_{h,N}}\norm{\svec{\uint\\ \uext}-\svec{w^{\rm int}_h\\ w^{\rm ext}_{h,N}}}_{\Vspace}.
\end{equation*}
\end{theorem}

\begin{theorem}[HSM for resonance problems]
Let $\Lambda$ be defined as in Theorem \ref{Theo:ConvHSMWG}. Then 
for all compact $\Lambda'\subset\Lambda$ we have
\[
\lim_{h\to 0,N\to\infty}\dist(\SpLam\cap\Lambda',\SpLam_{h,N}\cap\Lambda') = 0.
\]
\end{theorem}

There is no truncation error in the Hardy space method. Nevertheless, if $\kappa^2 \approx \lambda_n$ for some $n$ the approximation error can be large: For a single waveguide
 $\uML$ is given by (see \eqref{eq:Un} together with \eqref{eq:tpmform})
 \begin{equation}\label{eq:defi_zetan}
  \uML(z,y)= \sum_{n=1}^\infty \frac{\tilde c_n }{1/\zeta_n- z} \varphi_n(y)=\sum_{n=1}^\infty \tilde c_n \varphi_n(y) \sum_{j=0}^\infty \zeta_n^{j+1} z^j  \text{ with } \zeta_n:=\frac{\kappa_n-\kappa_0}{\kappa_n+\kappa_0}.
 \end{equation}
W.l.o.g. we assume $(\tilde c_n)_n$ to be exponentially decaying. This is always the case, if there exists a positive distance $a$ of $\bpM$ to a source of the scattered wave
due to the exponentially decaying evanescent modes $e^{i \kappa_n a}$. With \eqref{eq:prop_kn} we have $|\zeta_n|<1$ and we 
can estimate the square of the exterior approximation error for each mode 
 $(u_{0,n},\uML_n)^\top \in \Xspace_n$ by
  \begin{equation}
 \label{eq:errorestimateHSmmodal}
  \begin{aligned}
    &(1+\lambda_n) \inf_{v_0 \in \setC} |u_{0,n} -v_0|^2 + (3+2 \lambda_n) \inf_{\vML \in \{z^0, z^1 ,\dots, z^N\}} \|\uML_n -\vML\|_{L^2(S^1)}^2 \\
     &\quad = (3+2 \lambda_n) |\tilde c_n|^2 \sum_{j=N+1}^\infty |\zeta_n|^{2(j+1)} %\\
     %&\quad 
     = \frac{(3+2 \lambda_n) |\tilde c_n|^2  |\zeta_n|^{2(N+2)}}{1-|\zeta_n|^2}.
  \end{aligned}
 \end{equation}
 For fixed $n$ we see an exponential decay with increasing number of degrees of freedom $N+2$. For fixed $N$, exponential convergence in $n$
 follows with exponentially decreasing $(\tilde c_n)_n$ and $\lambda_n \in \calO(n^2)$, since for $n \to \infty$
 \begin{equation*}
 \frac{ (3+2 \lambda_n)  |\zeta_n|^{2(N+2)}}{1-|\zeta_n|^2 }\leq \tilde C\frac{ 3+2 \lambda_n }{|1-\zeta_n|^2 } = 
 \frac{\tilde C}{4 |\kappa_0|^2} (3+2 \lambda_n)|\sqrt{\kappa^2-\lambda_n} + \kappa_0|^2.
 \end{equation*}
 Similar to the PML error, the error becomes large, if $|\zeta_n|\approx 1$ for one $n \in \setN$, which is the 
 case for $\kappa^2\approx\lambda_n$ (i.e. $\zeta_n\approx -1$) and near the branch cuts of the 
 square root for $\kappa_n$.

\subsection{Modified Hardy space method}
\label{Sec:Mod_HSM}
Nevertheless, for diffraction problems with given frequency and given wavenumbers we are able to modify the HSM slightly to get rid of the problem for $\kappa^2\approx \lambda_n$, i.e.\
$|\kappa_n|$ small and $|\zeta_n| \approx 1$ in \eqref{eq:defi_zetan}. 
The problem arises since the approximation of the mode 
\begin{equation*}
 b_{\zeta} (z):=\frac{1}{1-{\zeta} z } = \sum_{j=0}^\infty \zeta^j z^j
\end{equation*}
with the monomials $z^0, \dots, z^N$ is bad for $|\zeta|\approx 1$. Hence, if $\kappa_n$ and therefore $\zeta_n$ is known 
and if one of the $|\zeta_n|$ is near to $1$, it seems reasonable to include this critical mode to the basis
\begin{equation*}
\tilde \Xspace_N:=\setC \oplus \Span \{z^0, z^1 ,\dots, z^N, \frac{1}{1-\zeta z }\}   \subset \setC\oplus H^+(S^1).
\end{equation*}
Note that $\zeta_n \to 1$ for $n \to \infty$, but since $(\tilde c_n)_n$ decreases exponentially, $|\zeta_n|\approx 1$ is only a problem, if this happens for small $n$.

The discrete operators $\tilde \OpT_\pm^N : \tilde \Xspace_N \to \Span \{z^0, z^1 ,\dots, z^N,z^{N+1}, \frac{1}{1-\zeta z }\}$ and
the usual operators  $\OpT_\pm^N :\Xspace_N \to \Span \{z^0, z^1 ,\dots, z^N,z^{N+1}\}$ are described by
the matrices 
\begin{equation*}
\tilde \OpT_\pm^N:=\left(\begin{array} {ccccc}
1 & \pm 1 & 0 & 0  & -\frac{1}{d}\\
0 & \ddots & \ddots& 0 & 0\\
0 & 0&   1&\pm1 &0\\
0&0&0&1&0\\
0&0&0&0&\frac{1}{d} \pm1
 \end{array} \right),\quad 
 \OpT_\pm^N:=\left(\begin{array} {cccc}
1 & \pm 1 & 0 & 0  \\
0 & \ddots & \ddots& 0 \\
0 & 0&   1&\pm1 \\
0&0&0&1\\
 \end{array} \right).
\end{equation*}
We define the bilinear form in \eqref{eq:BasicId} by
\begin{equation*}
  q_{\kappa_0}(\uML,\vML) := \frac{ - i \kappa_0}{ \pi} \int_{S^1}\uML(z)\,\vML(\overline{z})\,|dz|,\qquad \uML,\vML\in H^+(S^1).
\end{equation*}
The monomials are orthogonal to each other, and therefore $q_{\kappa_0}(z^j, j^k)=-2 i \kappa_0 \delta_{j,k}$. For $b_{\zeta}$ we compute
\begin{equation*}
 q_{\kappa_0}(b_{\zeta}, z^j) = -2 i \kappa_0 \zeta^j,\qquad q_{\kappa_0}(b_{\zeta}, b_{\zeta})
=\frac{-2 i \kappa_0}{1-\zeta^2}.
\end{equation*}
If we define the matrices $M_{\kappa_0}$, $\tilde M_{\kappa_0}$, $S_{\kappa_0}$ and $ \tilde S_{\kappa_0}$ by 
\begin{equation*}
 \left(M_{\kappa_0} \right)_{jk}:= \frac{1}{(i \kappa_0)^2} q_{\kappa_0}(\OpT_-^N b_l, \OpT_-^N b_k ), \qquad \left(S_{\kappa_0} \right)_{jk}:= q_{\kappa_0}(\OpT_+^N b_l, \OpT_+^N b_k )
\end{equation*}
for $b_j,b_k \in \Xspace_N$  and $b_j,b_k \in \tilde \Xspace_N$ respectively, we get for the usual Hardy space method
\begin{equation*}
 M_{\kappa_0} := \frac{2 i}{\kappa_0} \left(\OpT_-^N \right)^T \OpT_-^N , \quad  S_{\kappa_0} := -2 i \kappa_0 \left(\OpT_+^N \right)^T \OpT_+^N 
\end{equation*}
and for the modified one
\begin{equation*}
 \tilde M_{\kappa_0} := \frac{2 i}{\kappa_0} \left(\tilde \OpT_-^N \right)^T D \, \tilde \OpT_-^N , \quad  \tilde S_{\kappa_0} := -2 i \kappa_0 \left(\tilde \OpT_+^N \right)^T D \, \tilde\OpT_+^N 
\end{equation*}
with
\begin{equation*}
D:= \left(\begin{array} {ccccc}
1 & 0 & \cdots & 0 & \zeta^0\\
0 & \ddots & & & \zeta^{1}\\
\vdots & & \ddots & &\vdots\\
0 & & &1  &\zeta^{N}\\
\zeta^{0} &\zeta^{1} &\cdots&\zeta^{N}& \frac{1}{1-\zeta^{2}}
 \end{array} \right).
\end{equation*}
This modification of the Hardy space method is covered by our theory if $|\zeta|\neq 1$. It improves the approximation error a lot, if $|\zeta|\approx 1$. The condition of the 
 system matrix will become large if $|\zeta|$ is not in the neighborhood of $1$ since then the extra basis function is well approximated by the other basis functions.% Sloppily speaking,

 \begin{rem}
 There exist strategies to improve the PML in the case of small effective damping as well. In \cite{ZschiedrichKloseSchaedleetal:2006} 
 an adaptive procedure to chose the thickness of the damping layer is presented, which was amongst others used for the simulation of a 3d plasmonic waveguide
 \cite{BurgerZschiedrichPomplunetal:2010}. Moreover, at least for positive $\kappa$ the mesh in the damping layer should be coarser with increasing $x$, since 
 typically the highly oscillating waves ($\Re(\kappa_n)$ large) needing a fine mesh are damped out quickly.
\end{rem}

\section{Numerical Results}
There exist several numerical studies for diffraction and resonance problems for two-dimensional waveguide problems using PML and Hardy space methods
\cite{HeinKochNannen:10,HeinKochNannen:12}. Here, we confine ourselves to one detailed convergence study for a 3d diffraction problem and one numerical computation of a resonance problem.
All the computations were done in the finite element code Netgen/Ngsolve \cite{netgen} using openMP parallelization with the direct solver PARDISO \cite{Pardiso} or MPI parallelization
with the sparse direct solver MUMPS.

\subsection{Scattering problem}
We consider a single tube $\Omega:=\setR \times \bp$ with the unit disk  
$\bp:=B_1(0)\subset \setR^2$ as cross section and homogeneous
Neumann boundary conditions $\Bdv u =\diffq{u}{\nu}=0$ on 
$\partial\Omega$. The interior domain was chosen as 
$\Oi:=(0,1)\times B_1(0)$, and the two components of the 
exterior domain $\Oe=W_1\cup W_2$ are $W_1:=(1,\infty)\times B_1(0)$ 
and $W_2:=(-\infty,0)\times B_1(0)$.
The eigenfunctions of $-\Delta:\{\varphi\in H^2(\bp):\diffq{\phi}{\nu}=0
\mbox{ on }\partial \bp\} \to L^2(\bp)$ 
are 
\begin{equation*}
 \varphi_{\pm m,n}(r\cos\theta,r\sin\theta)
=J_{m}\left(\mu_{m,n} r \right) e^{\pm i m \theta},\qquad m\in \setN_0, n\in \setN.
\end{equation*}
$J_{m}$ are the Bessel functions and $\mu_{m,n}$ the $n$th root of $J_{m}'$. 
The first corresponding eigenvalues are $\lambda_1=\mu_{0,1}^2=0$,
\begin{equation*}
 \begin{aligned}
  &\lambda_2=\lambda_3=\mu_{1,1}^2 \approx 1.84118378134^2, \qquad &&
\lambda_4=\lambda_5=\mu_{2,1}^2 \approx 3.05423692823^2,\\
  &\lambda_6:=\mu_{0,2}^2 \approx 3.83170597021^2,\qquad && 
\lambda_7=\lambda_8:=\mu_{3,1}^2 \approx 4.20118894121^2.
 \end{aligned}
\end{equation*}
For all the computations we have chosen as incoming wave a superposition of 
$5$ waveguide modes (see Fig.~\ref{Fig:ModesWG} for the real part of the modes) 
using the eigenfunctions to the first
$5$ eigenvalues neglecting the multiplicities:
\begin{equation*}
 u_{\rm inc}(x,y):=\sum_{n \in \{1,2,4,6,7\}} 
e^{i\sqrt{\kappa^2-\lambda_n} x}\varphi_n(y),\qquad x\in\setR,y\in\bp.
\end{equation*}
The interior domain for all computations in this subsections is $\Oi=(0,1)\times B_1(0)$ and discretized by 17750 tetrahedrons with maximal mesh size $h=0.1$. 

\begin{figure}
  \capstart
 \begin{center}
  \subfigure{\includegraphics[width=0.19\textwidth]{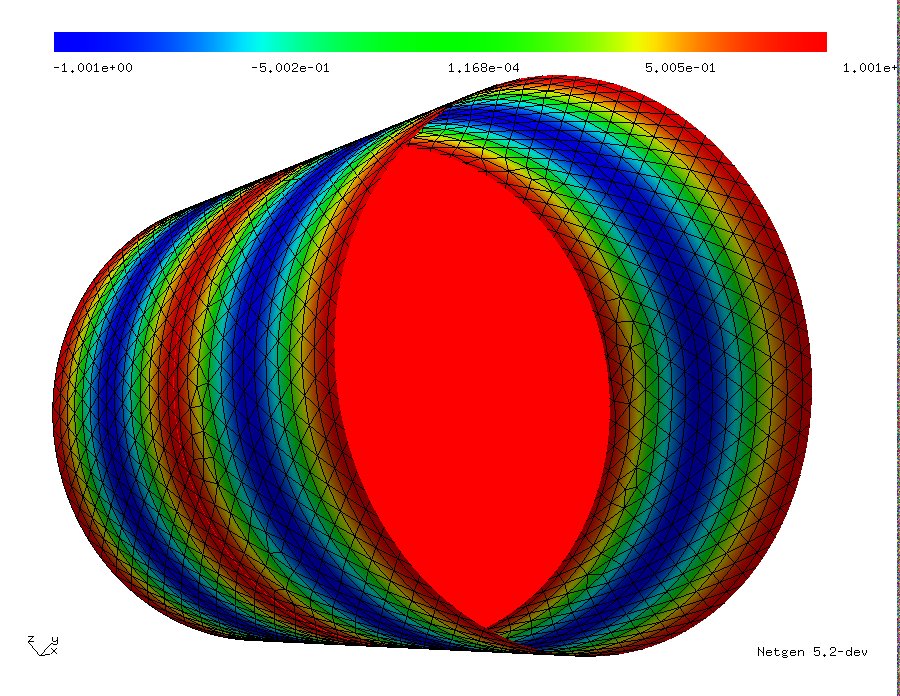}} 
  \subfigure{\includegraphics[width=0.19\textwidth]{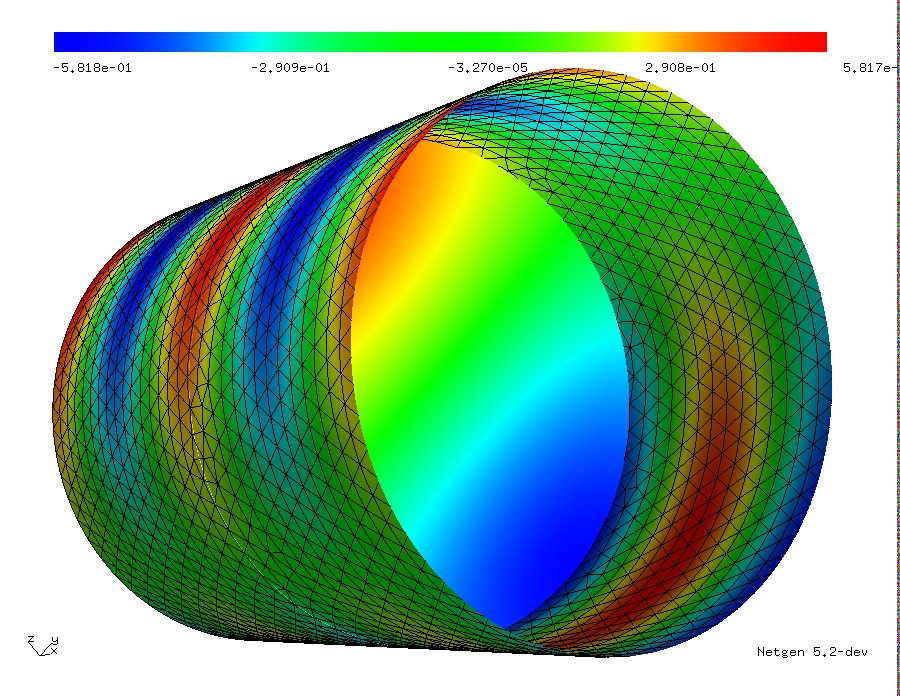}} 
  \subfigure{\includegraphics[width=0.19\textwidth]{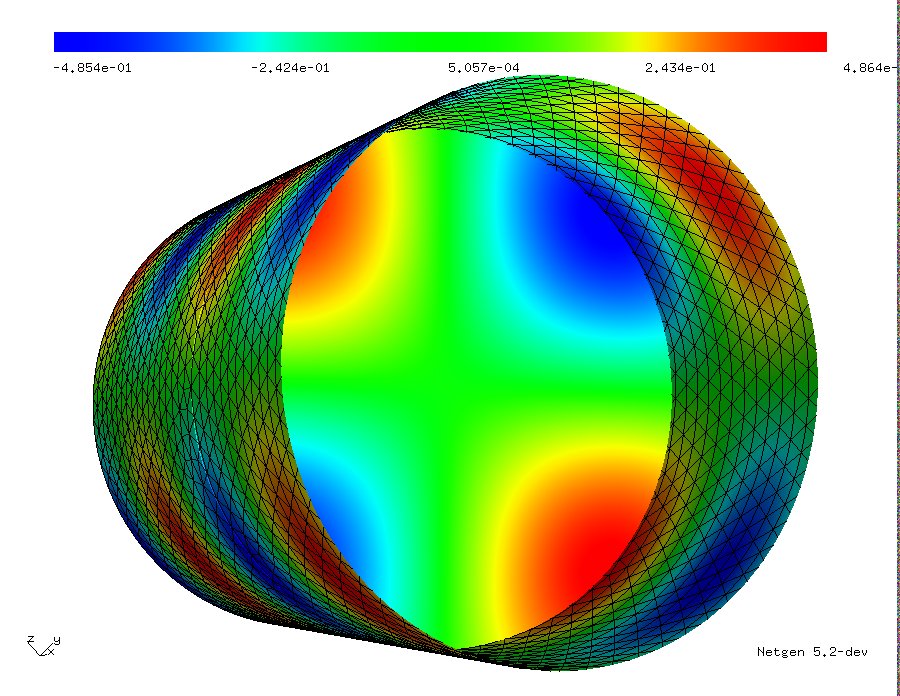}} 
  \subfigure{\includegraphics[width=0.19\textwidth]{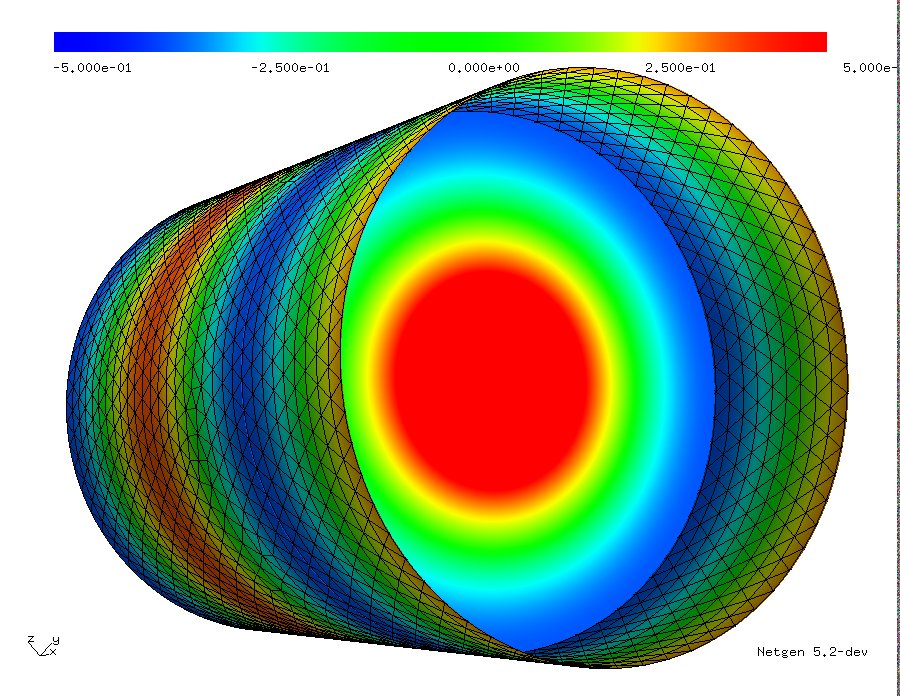}}
  \subfigure{\includegraphics[width=0.19\textwidth]{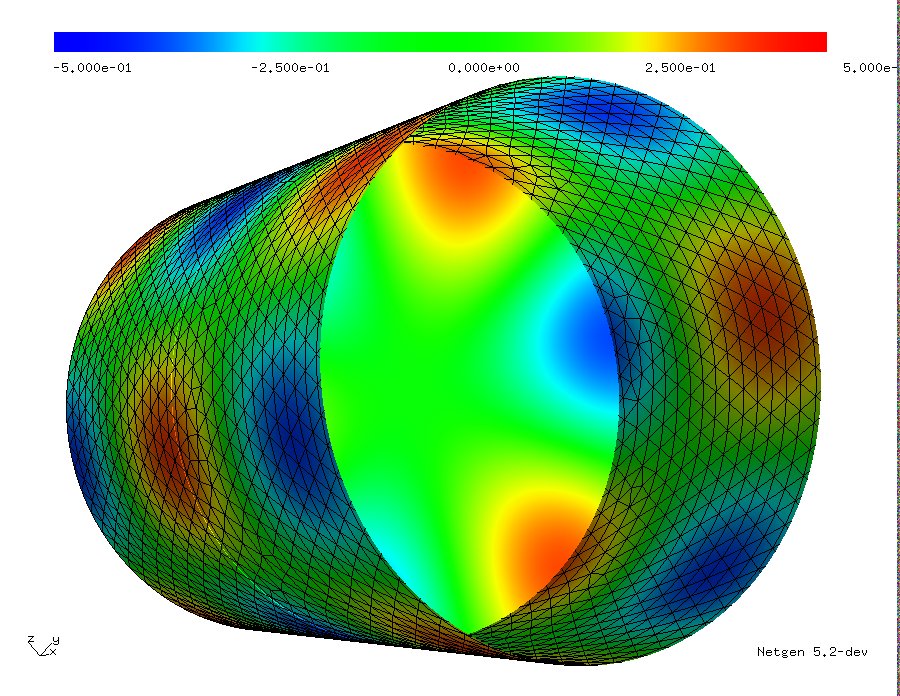}}
  \caption{$5$ waveguide modes to the surface eigenvalues $\lambda_1$, $\lambda_2$, $\lambda_4$, $\lambda_6$ and $\lambda_7$ for $\kappa=5$ and $\bp=B_1(0)$}
  \label{Fig:ModesWG}
 \end{center}
\end{figure}

First, we analyze the dependence of the error of the Hardy space method 
on the complex parameter $\kappa_0$. Neglecting the compact perturbation 
arguments in the proof of Theorem~\ref{theo:ConvTheo} the theoretical 
error bound of \eqref{eq:errorestimateHSmmodal} is 
\begin{equation}\label{eq:eta_explicit}
 \eta(\kappa_0,\kappa,N):=\frac{C(\kappa_0,\kappa)}{\alpha(\kappa_0,\kappa)} \sqrt{\sum_{n \in \{1,2,4,6,7\}} 
 \frac{(3+2\lambda_n)|e^{i\kappa_n(\kappa)}|^2|d_n(\kappa_0,\kappa)|^{2(N+2)}}{1-|d_n(\kappa_0,\kappa)|}},\quad 
\end{equation}
with $\alpha(\kappa_0,\kappa):= \min\{\alpha_1(\kappa_0,\kappa),\dots,\alpha_5(\kappa_0,\kappa)\}$ and
\begin{eqnarray*}
 C(\kappa_0,\kappa)&:=&\frac{4}{\pi} \max\left\{ |\kappa_0|+\left|\frac{\kappa^2}{\kappa_0}\right| ,\frac{1}{|\kappa_0|} \right\},\\
 \alpha_n(\kappa_0,\kappa)&:=&\frac{1}{\pi}  \begin{cases}
           \min\left\{\Re(\kappa_0),\Re\left( \frac{\lambda_n^2-\kappa^2}{\kappa_0(1+\lambda_n)} \right) \right\},\quad &\Re\left( \frac{\lambda_n^2-\kappa^2}{\kappa_0(1+\lambda_n)} \right)>0,\\
           \min\left\{\Im(\kappa_0),\Im\left( \frac{\lambda_n^2-\kappa^2}{\kappa_0(1+\lambda_n)} \right) \right\},\quad &\Re\left( \frac{\lambda_n^2-\kappa^2}{\kappa_0(1+\lambda_n)} \right)\leq0
           \end{cases},\\
 d_n(\kappa_0,\kappa)&:=&\frac{\kappa_n(\kappa)-\kappa_0}{\kappa_n(\kappa)+\kappa_0}.
\end{eqnarray*}
We tested three different frequencies with a sufficiently fine finite element
discretization such that the error of the Hardy space method was dominating.
The results shown in Fig.~\ref{Fig:Bestk0_neu} demonstrate that the bound
\eqref{eq:eta_explicit} represents the dependence of the error on the parameter 
$\kappa_0$ qualitatively correctly and would provide a good guidance for 
the choice of $\kappa_0$ in this example. 

Let us discuss that last two frequencies in Fig.~\ref{Fig:Bestk0_neu}
which are very special. $\kappa=4.2$ is in the neighborhood of $\sqrt{\lambda_7}$, and hence the error of the standard HSM is large even with $20$ degrees of freedom in radial direction (cf.\ Fig.~\ref{Fig:Err_k}). Moreover, it can  
clearly be seen in this case that 
the optimal parameter $\kappa_0$ depends on the wavenumbers $\kappa_n$. 
The ''problematic'' wavenumber $\kappa_5\approx 0.1i$ is very small and therefore
the optimal parameter $\kappa_0$ of the standard HSM would be very small. The modified HSM of Sec.~\ref{Sec:Mod_HSM} resolves this problem completely (see Fig.~\ref{Fig:Err_k}). 

The test for $\kappa=5$ in Fig.~\ref{Fig:Bestk0_neu} is also special, since for $\kappa>\sqrt{\lambda_7}$ all $5$ used modes are guided and no evanescent mode has to be resolved by the Hardy space method. 
Hence, the optimal $\kappa_0$ would be almost real and only very few degrees of freedoms in radial direction are necessary ($N=4$ for a polynomial order $p=6$).
This case would happen in a practical computation 
if the distance of the artificial boundary $\{1\}\times \bp$ to a source or a scatterer is large since then the evanescent modes are already
damped out at $\{1\}\times \bp$. 

\begin{figure}
  \capstart
 \begin{center}
 \subfigure[$\eta(\kappa_0,3.5,6)$]{\includegraphics[width=0.32\textwidth]{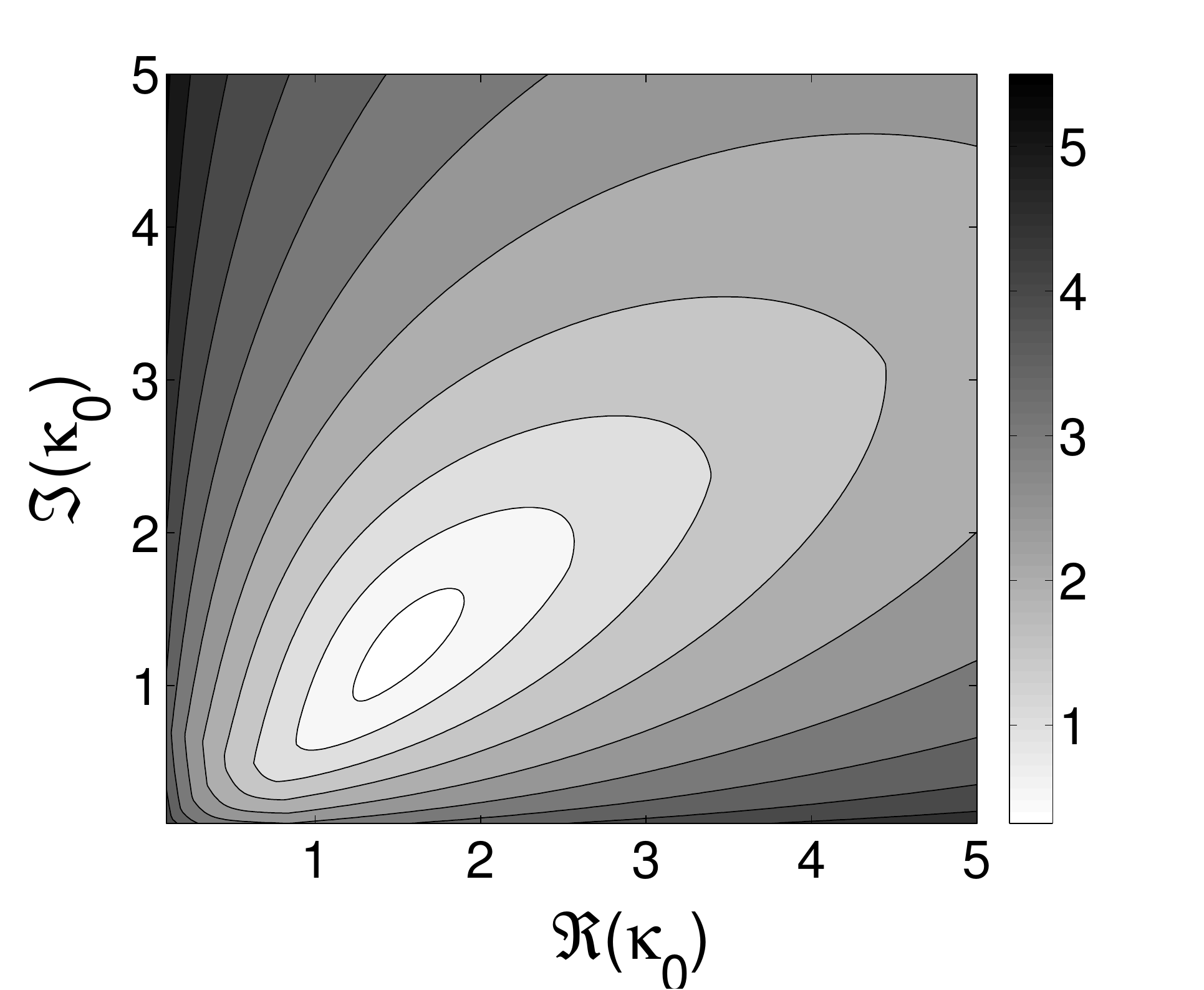}} \hfill
 \subfigure[$\eta(\kappa_0,4.2,20)$]{\includegraphics[width=0.32\textwidth]{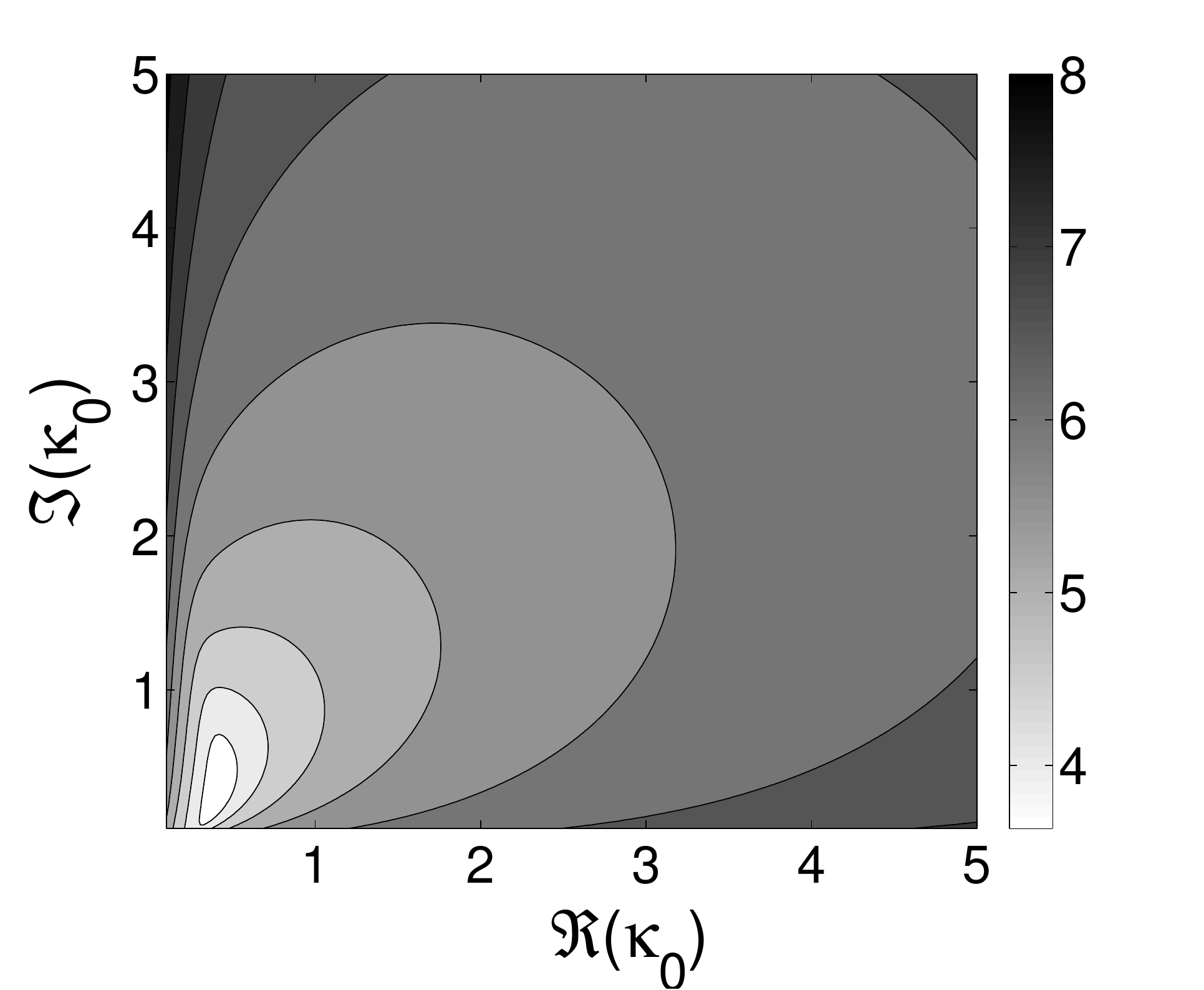}} \hfill
 \subfigure[$\eta(\kappa_0,5,6)$]{\includegraphics[width=0.32\textwidth]{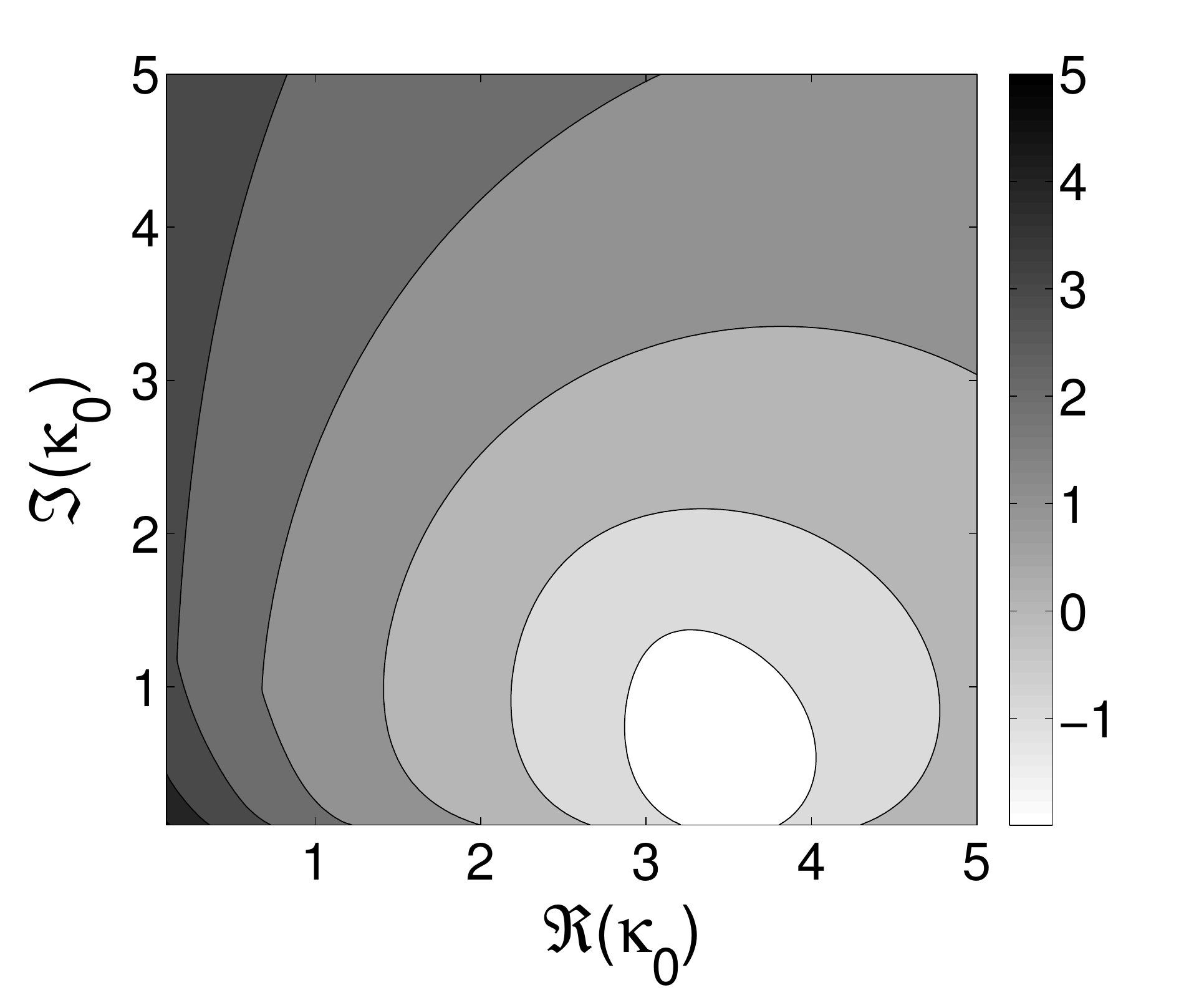}}

 \subfigure[$\kappa=3.5$, $N=6$, $p=5$]{\includegraphics[width=0.32\textwidth]{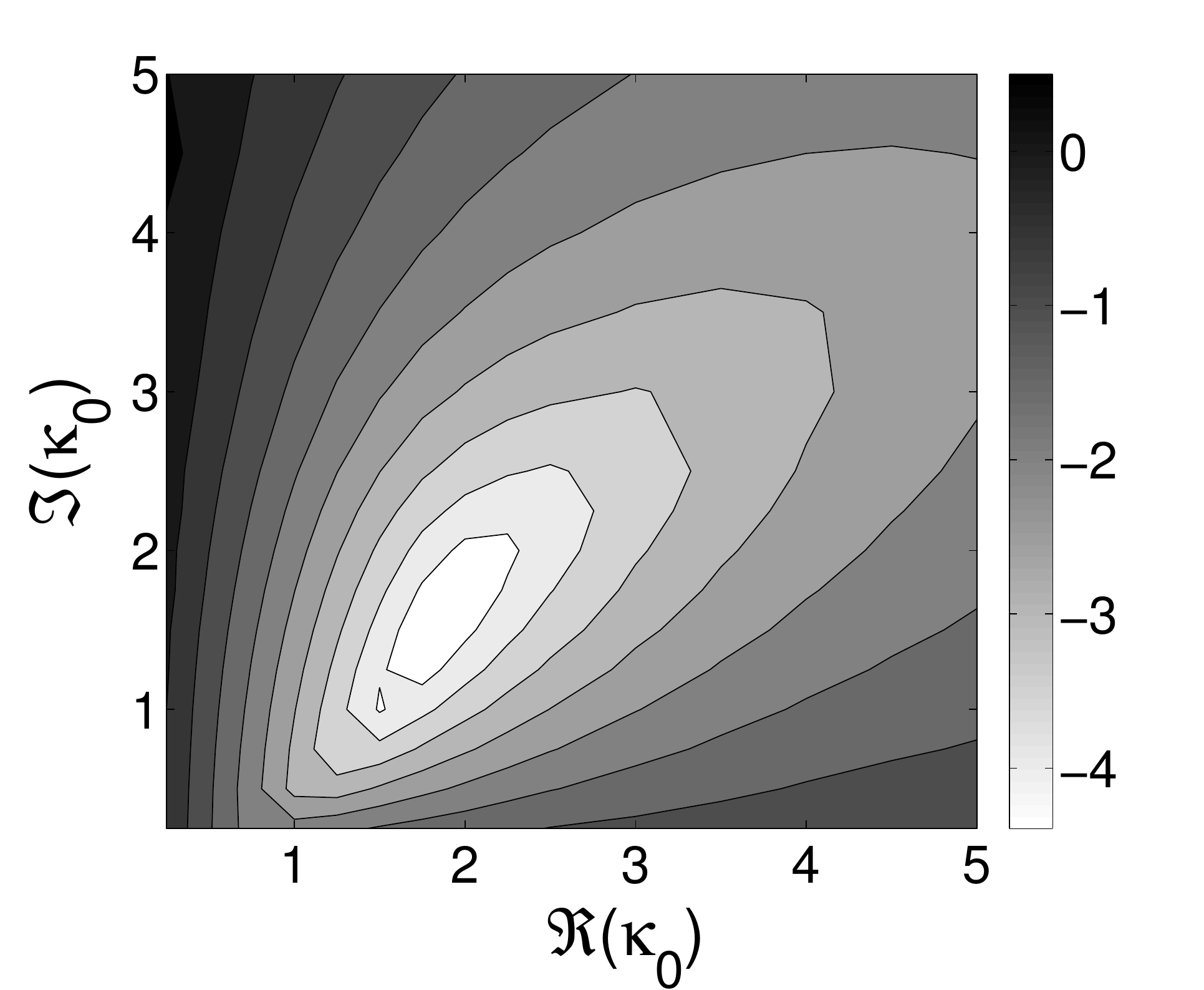}} \hfill
\subfigure[$\kappa=4.2$, $N=20$, $p=4$\label{Fig:Bestk0_neu_num3}]{\includegraphics[width=0.32\textwidth]{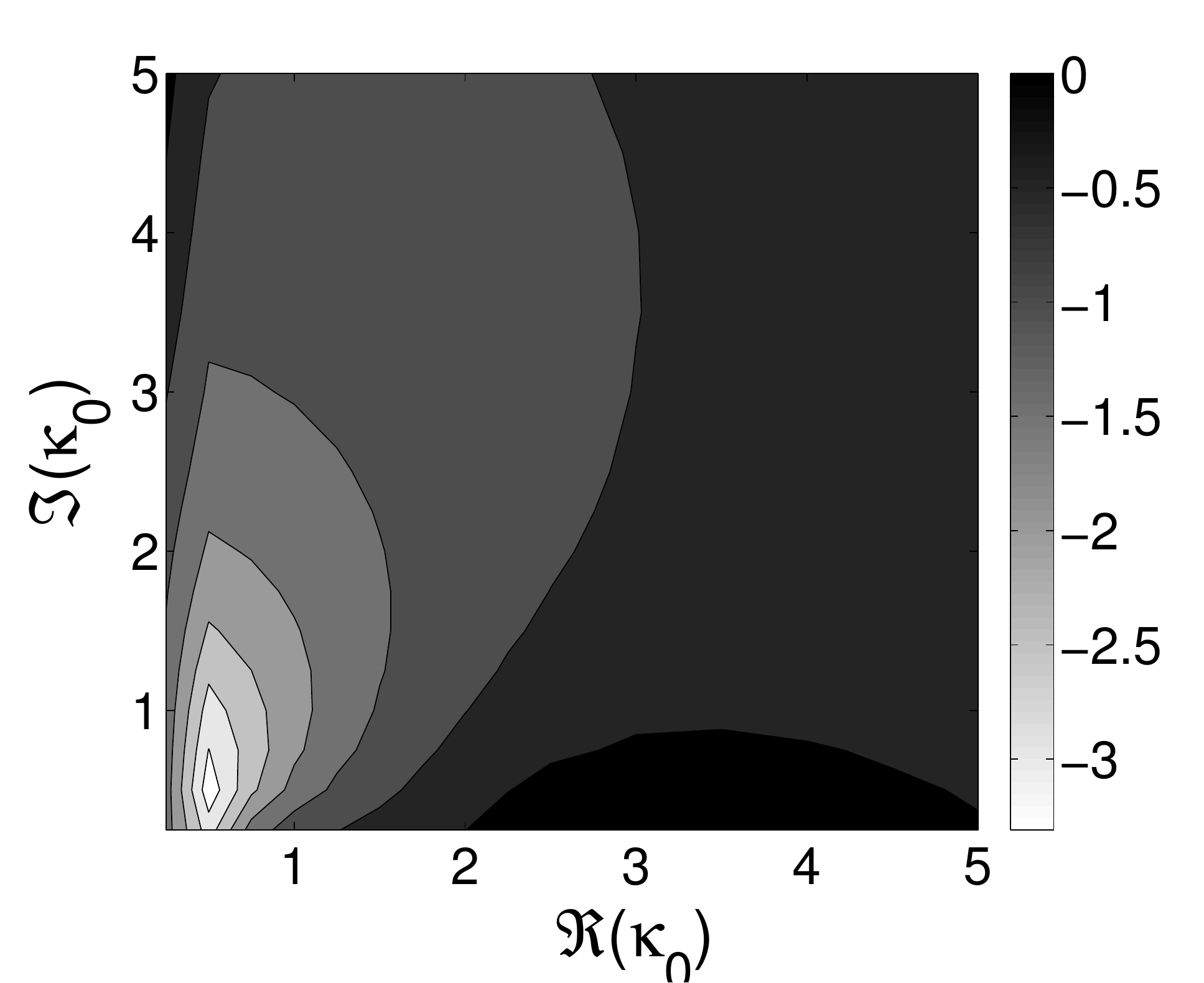}} \hfill
\subfigure[$\kappa=5$, $N=4$, $p=6$\label{Fig:Bestk0_neu_num4}]{\includegraphics[width=0.32\textwidth]{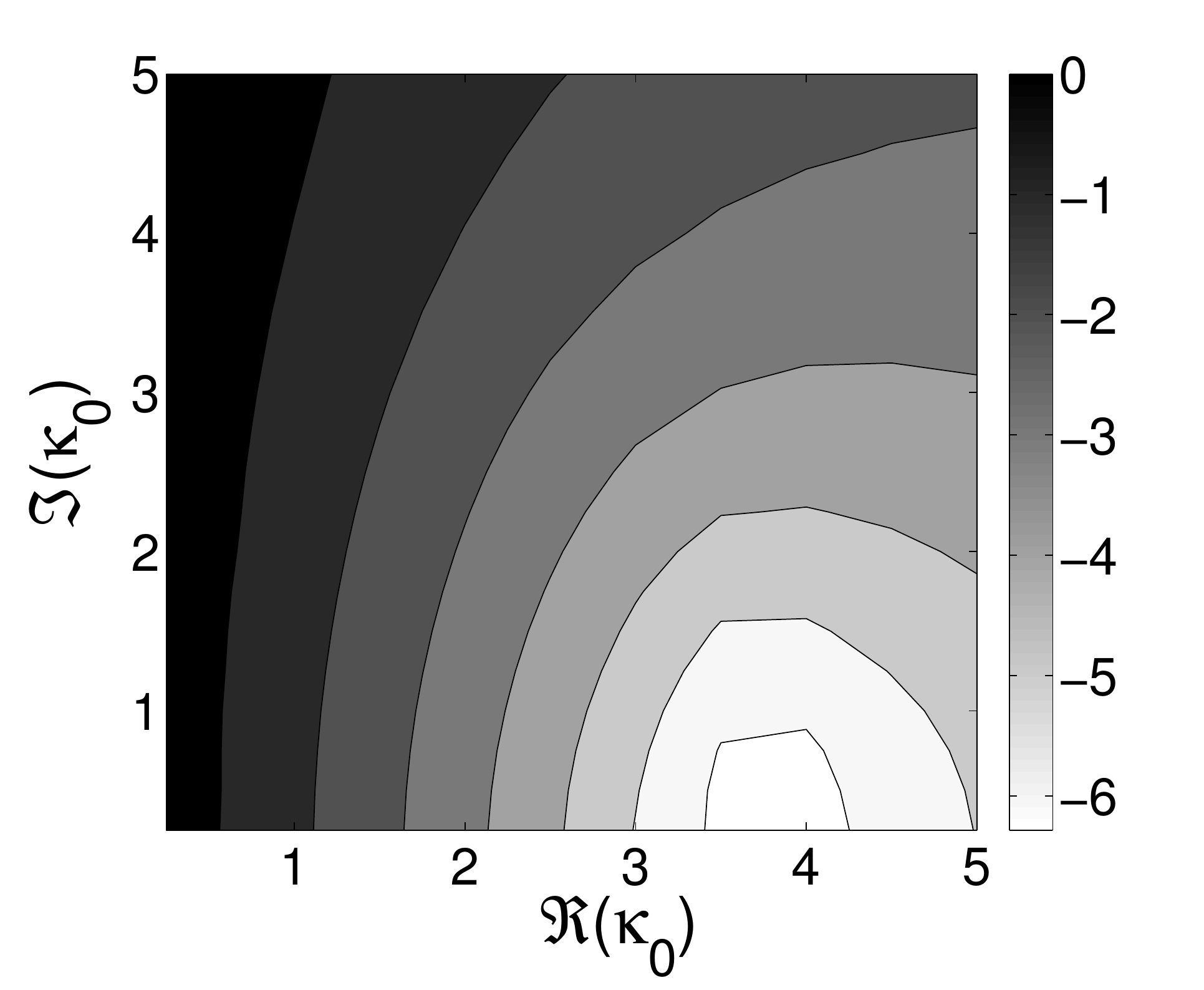}}
 \caption{Study of the dependence of the error of the HSM on the complex 
parameter $\kappa_0$. The upper panels show $\log_{10}\eta$ where $\eta$ is 
approximate error bound in \eqref{eq:eta_explicit}. The lower panel shows 
$\log_{10}\|\uint_{h}-\uint\|_{H^1(\Oi)}$.}
  \label{Fig:Bestk0_neu}
 \end{center}
\end{figure}

Second, we have fixed the parameter $\kappa_0=2+2i$ and computed the relative $H^1(\Oi)$-error for different finite element polynomial orders
and different numbers $N$ of degrees of freedom for the Hardy space method  
(see Fig.~\ref{Fig:Conv_dimH_order}) In the left panel the exponential convergence
of the Hardy space method can be seen. For the most expensive computation with $N=14$ and $p=6$, we have used the MPI parallel sparse direct solver MUMPS with 30 cores and in total 
1 million unknowns. Approx. 30\% of these unknowns were needed for the Hardy space method. The wall time for this computation was approximately 39 minutes, 
37 of them spent for the MUMPS factorization.

\begin{figure}
  \capstart
 \begin{center}
  \subfigure[fixed $\kappa=3.5$, varying $N$ and $p$ \label{Fig:Conv_dimH_order}]{\includegraphics[width=0.46\textwidth]{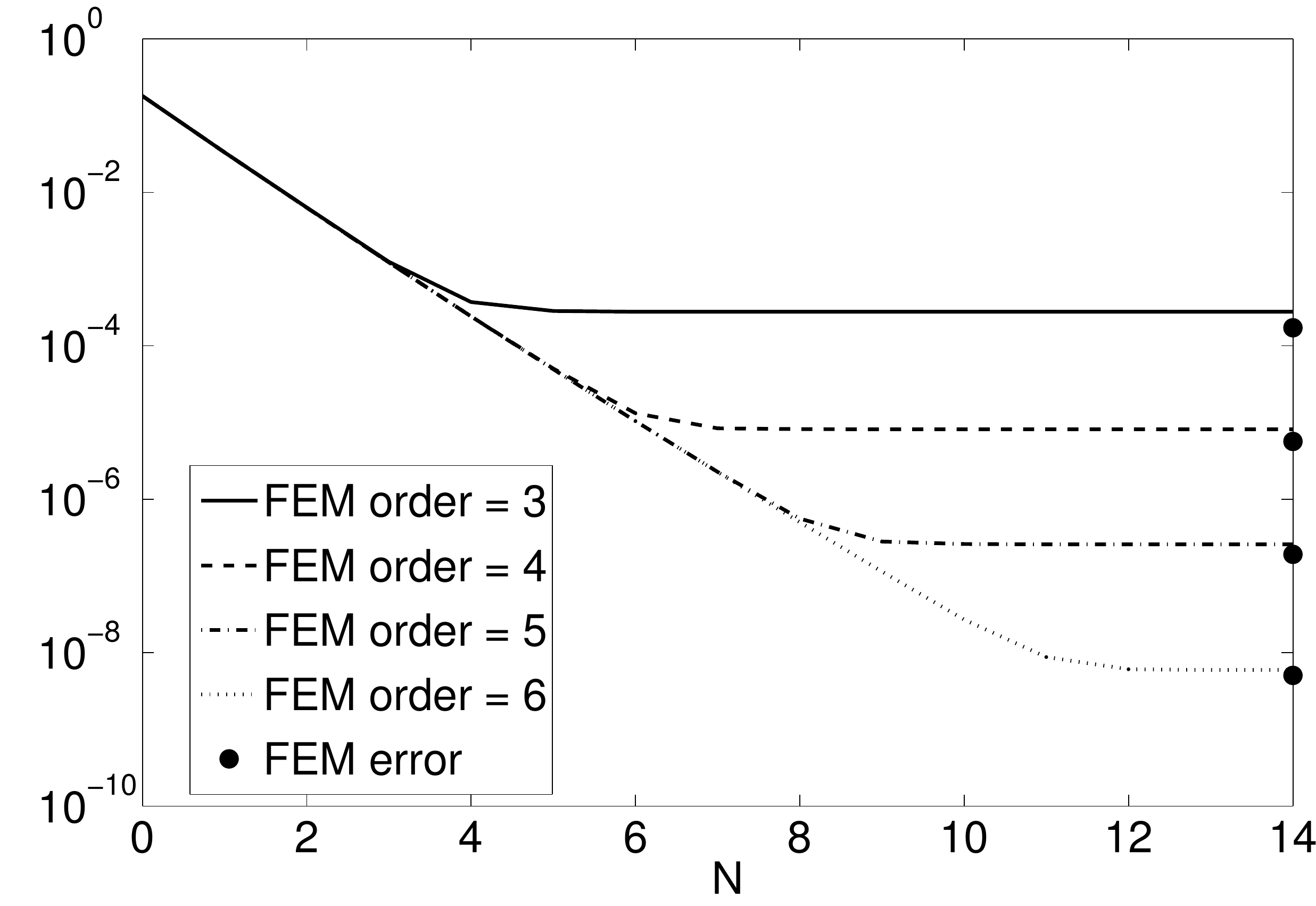}} \hfill
  \subfigure[fixed $N=10$, $p=5$, varying $\kappa$ \label{Fig:Err_k}]{\includegraphics[width=0.52\textwidth]{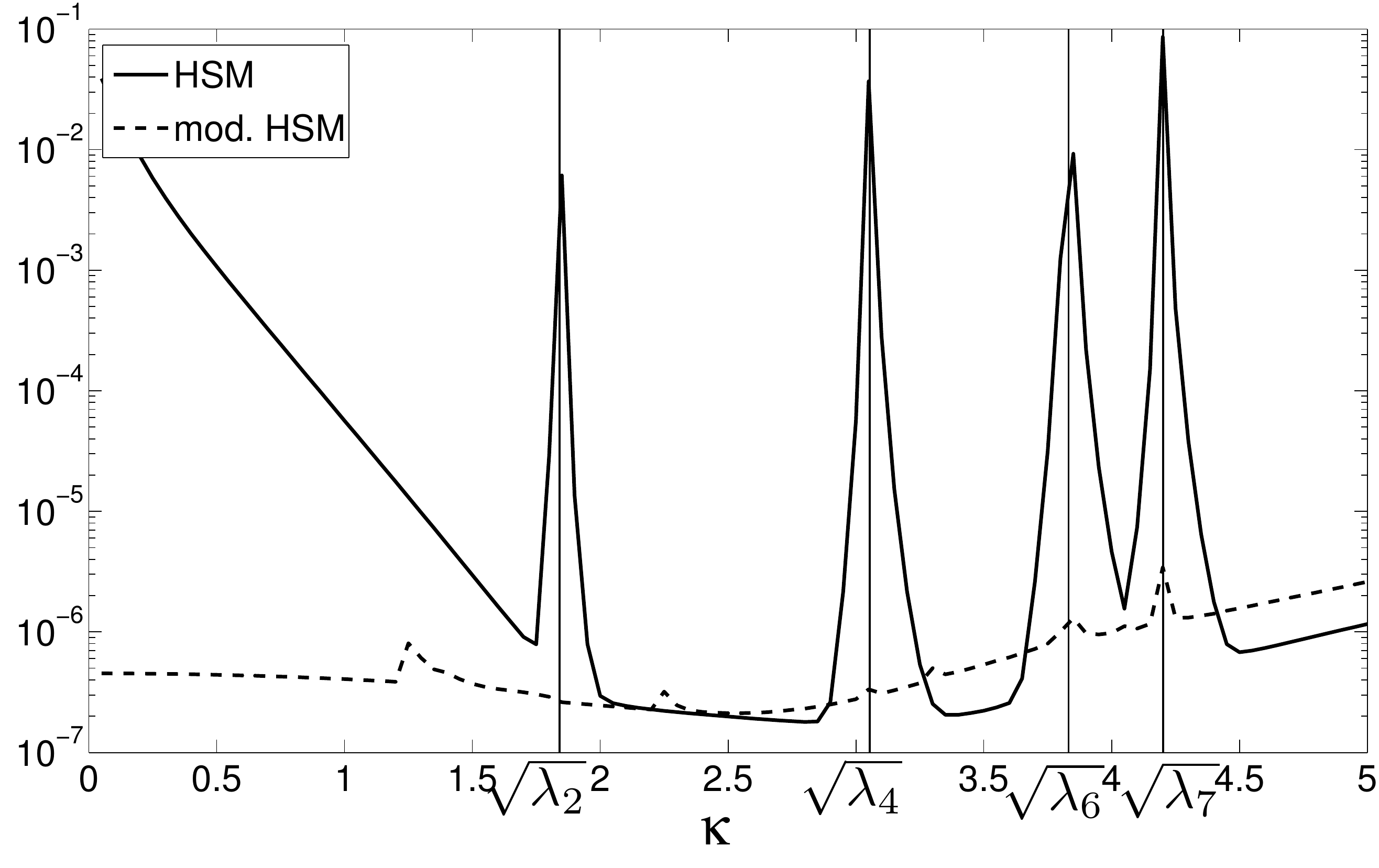}}
  \caption{relative $H^1(\Oi)$-error for $\kappa_0=2+2i$. 
The dots in the left panel represent the pure finite element error with given 
Dirichlet boundary data. For the modified HSM see Sec.~\ref{Sec:Mod_HSM}.}
 \end{center}
\end{figure}

Last, we have computed dependence of the error on the frequency $\kappa$ with fixed $\kappa_0=2+2i$, $N=10$, and fixed finite element discretization. In Sec.~\ref{Sec:Mod_HSM} we have already
mentioned the problem with $\kappa^2 \approx \lambda_n$, which can be seen in Fig.~\ref{Fig:Err_k}. The modified Hardy space method of Sec.~\ref{Sec:Mod_HSM} resolves the problem
completely, However, this modification cannot easily be used for resonance problems since it would lead to nonlinear eigenvalue problems. 

\subsection{Resonance problem}
\begin{figure}
  \capstart
 \begin{center}
  \includegraphics[width=1\textwidth]{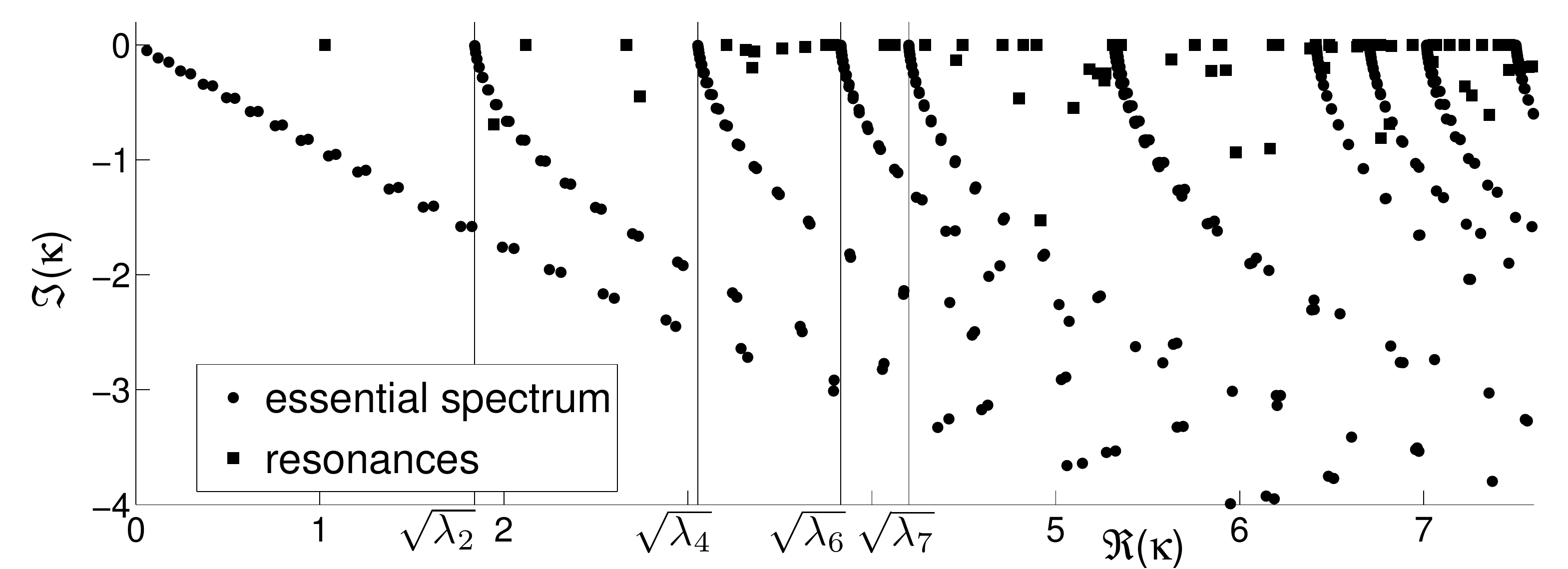}
  \caption{Computed resonances of a circular ring cavity, 
see Fig.~\ref{Fig:ResPlots}. With dots we indicate the discretization
of the continuous spectrum (see e.g.\ \cite{KimPasciak:09} for the analogous 
situation of PML with $\sigma=i /\kappa_0$). 
They build the boundary of the admissible set 
(cf.\ Fig.~\ref{Fig:Res_RingCav} with Fig.~\ref{Fig:domainkappa}).
The boxes indicate computed resonances.
}
  \label{Fig:Res_RingCav}
 \end{center}
\end{figure}

There exist numerical convergence studies to acoustic and electromagnetic resonance problems using the Hardy space method in \cite{NannenSchaedle:09, Nannenetal:13}. Here, we only
present one simple 3d resonance problem, which is an extension of the 2d waveguide cavity problems in \cite{HeinKochNannen:12}. The domain is given by a circular ring 
cavity of radius $2$ and length
$1$ connected with two circular waveguides with radius $1$: 
$\Omega = (-\infty,-0.5) \times B_1(0) \cup
(-0.5, 0.5) \times B_2(0) \cup (0.5,\infty) \times B_1(0)$.

We chose $\Oi:= \Omega \cap (-1,1) \times B_2(0)$ and discretized the
resonance problem with a finite element mesh with 
maximal mesh size $h=0.5$ and $654$ volume elements,
isoparametric elements of order $p=14$ and the Hardy space method for the two waveguides with $\kappa_0=2+2i$ and $N=25$. The first $1000$ resonances computed with
a shift and invert Arnoldi algorithm with fixed shift $\rho=10-i$, the sparse direct solver MUMPS and a Krylov space of dimension $2000$ are given in Fig.~\ref{Fig:Res_RingCav}.

For a closed cylinder of length $1$ and radius $2$ the resonances are 
\begin{equation}\label{eq:EV_closed_cylinder}
 \kappa_{m,n,l}=\sqrt{\left(\frac{\mu_{m,n}}{2}\right)^2+(l \pi)^2},\qquad m,l\in \setN_0, n \in \setN.
\end{equation}
The resonance functions in Fig.~\ref{Fig:ResPlots} are perturbations of the closed cavity eigenfunctions (compare the resonance function in Fig.~\ref{Fig:ResPlotsa} with the
second mode in Fig.~\ref{Fig:ModesWG}). 
For a complex resonance, the imaginary part reflects the energy loss per cycle. Since the only possible energy loss is
the energy radiated to infinity and since only guided modes radiate energy to infinity (see Sec.~\ref{sec:setting}), the resonances
are real, if the resonance function is orthogonal to the finitely many guided modes. This is the case in panel (a): For $\kappa$
with $\Re(\kappa)<\sqrt{\lambda_2}$ only the plane wave is guided and since the resonance function is antisymmetric with respect to
the centerline of the waveguide, it is orthogonal to all guided modes. This suggests that 
the imaginary part of the computed resonance $k_1$ is a numerical error. 

\begin{figure}
  \capstart
 \begin{center}
  \subfigure[$\kappa \approx 1.0273 -10^{-9}i$ \label{Fig:ResPlotsa} ]{\includegraphics[width=0.3\textwidth]{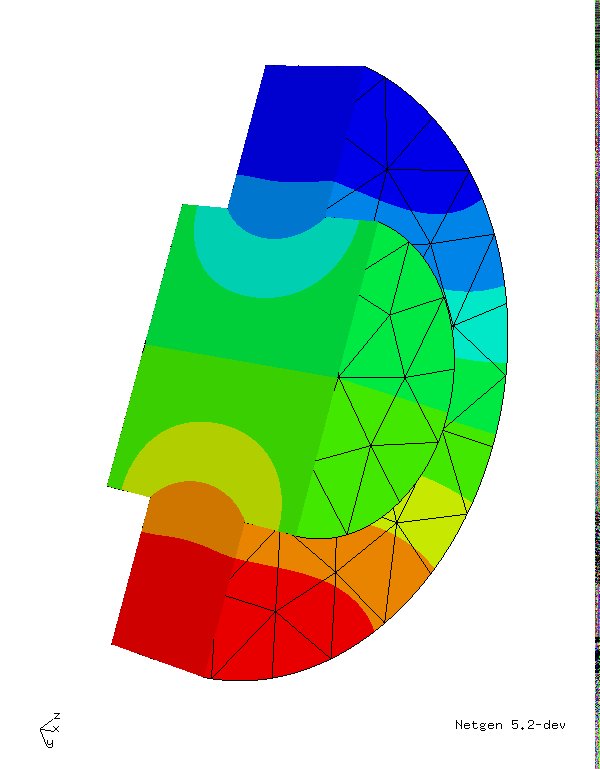}}
  \subfigure[$\kappa \approx 7.2206 -10^{-8}i$ \label{Fig:ResPlotsb} ]{\includegraphics[width=0.3\textwidth]{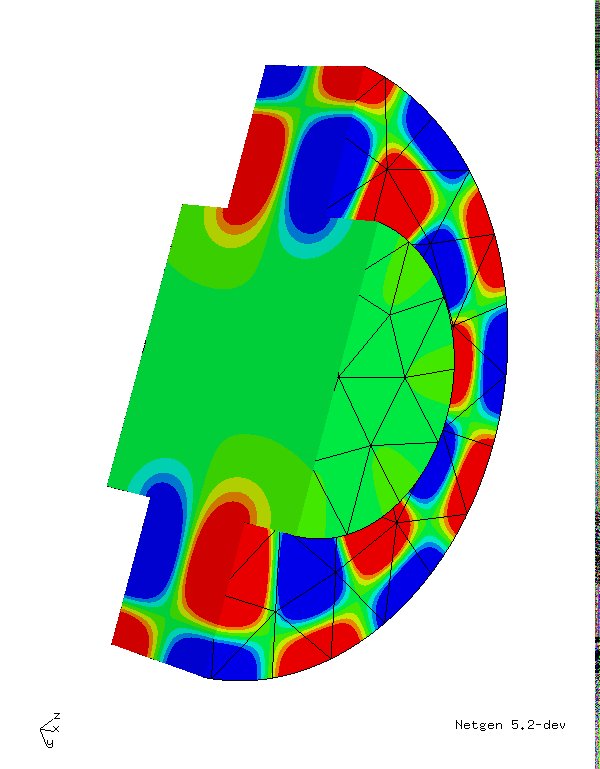}}
  \subfigure[$\kappa \approx  6.8247 - 8.1\cdot 10^{-3}i$ \label{Fig:ResPlotsc}]{\includegraphics[width=0.3\textwidth]{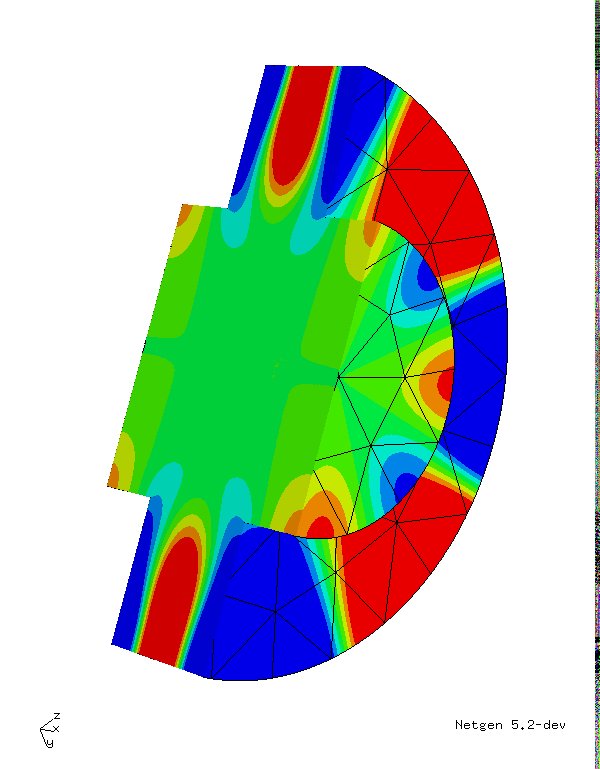}}
  \caption{The real part of $3$ resonance functions of a circular ring cavity 
are displayed  (cf.\ Fig.~\ref{Fig:ResPlots}). These resonances correspond to 
the following eigenvalues of the closed cylinder 
(see \eqref{eq:EV_closed_cylinder}): 
(a): $\kappa_{1,1,0} \approx  0.92059$; 
(b): $\kappa_{7,2,1} \approx  7.18897$;
(c): $\kappa_{5,1,2} \approx  6.82257$.
}
  \label{Fig:ResPlots}
 \end{center}
\end{figure}

These computations show that resonances in domains including open waveguides
can be computed naturally and reliably by the Hardy space method 
since it leads to a discrete eigenvalue problem. In contrast, 
methods which rely on a modal decomposition lead to discrete system 
which depend on the unknown $\kappa$ in a much more complicated way. 

\begin{acknowledgements}
The authors dedicate this work to Werner Koch for his inspiration, generosity and enthusiasm  
concerning the topic of resonances in waveguides. Unfortunately he passed away 
on August 28, 2012.  

Moreover, we would like to thank an anonymous referee for detailed and helpful suggestions 
and corrections. Financial support by the German Science Foundation through grant 
HO 2551/5 is gratefully acknowledged. 
\end{acknowledgements}

\bibliographystyle{siam} \bibliography{bibliography}
\end{document}